\documentclass{article}
\usepackage[utf8]{inputenc}
\usepackage[british]{babel}
\usepackage{amsmath, amsfonts, mathtools, amsthm, amssymb}
\usepackage{graphicx}
\usepackage[dvipsnames]{xcolor}
\usepackage{tikz}
\usepackage{geometry}
\usepackage[hidelinks]{hyperref}
\usepackage{aliascnt}
\usepackage{enumitem}
\usepackage{quiver}
\usepackage{yhmath}
\usepackage{bm, bbm}
\usepackage{subcaption}
\usepackage{lscape}
\usepackage{rotating}
\usepackage{etoolbox}

\usepackage{tocloft}

\theoremstyle{plain}
\newtheorem{thm}{Theorem}[section]

\newaliascnt{thmdef}{thm}
\newtheorem{thmdef}[thmdef]{Theorem/Definition} 
\aliascntresetthe{thmdef}

\newaliascnt{lem}{thm}
\newtheorem{lem}[lem]{Lemma} 
\aliascntresetthe{lem}

\newaliascnt{klem}{thm}
\newtheorem{klem}[klem]{Key Lemma} 
\aliascntresetthe{klem}

\newaliascnt{prop}{thm}
\newtheorem{prop}[prop]{Proposition} 
\aliascntresetthe{prop}

\newaliascnt{conj}{thm}
\aliascntresetthe{conj}

\newaliascnt{cor}{thm}
\newtheorem{cor}[cor]{Corollary} 
\aliascntresetthe{cor}

\theoremstyle{definition}
\newaliascnt{defn}{thm}
\newtheorem{defn}[defn]{Definition} 
\aliascntresetthe{defn}

\theoremstyle{remark}
\newaliascnt{ex}{thm}
\newtheorem{ex}[ex]{Example} 
\aliascntresetthe{ex}

\newaliascnt{rem}{thm}
\newtheorem{rem}[rem]{Remark} 
\aliascntresetthe{rem}

\newaliascnt{note}{thm}
\aliascntresetthe{note}

\DeclareMathOperator{\im}{im}
\newcommand{\filt}{\operatorname{Filt}_K}
\newcommand{\modspace}[1]{\mathbb{M}_{#1}}
\newcommand{\scpx}{\operatorname{\textbf{SCpx}}}
\newcommand{\vect}{\operatorname{\textbf{vec}}}

\newcommand{\en}{\operatorname{\text{End}}}
\newcommand{\aut}{\operatorname{\text{Aut}}}
\newcommand{\epi}{\operatorname{\text{Epi}}}
\newcommand{\odim}{\operatorname{odim}}
\newcommand{\inc}{\operatorname{inc}}
\newcommand{\id}{\operatorname{id}}
\newcommand{\preim}[2]{MPH_{|\mathcal{#2}}^{-1} ([\bm{#1}])}
\newcommand{\preimov}[2]{\overline{MPH_{|\mathcal{#2}}^{-1} ([\bm{#1}])}}

\usepackage[
  natbib=true,
  backend=biber,
  style=numeric-comp,
  sorting=nyt,
  maxbibnames=99]{biblatex}
\usepackage{csquotes}
\title{\vspace{-20pt}The fiber of multiparameter persistent homology for simplicial complexes}
\author{Heather A. Harrington\textsuperscript{1,2,3,4,5}, Ulrike Tillmann\textsuperscript{1,6} and Maria Torras-Pérez\textsuperscript{1}}
\date{}

\begin{document}

\maketitle

\noindent\textbf{1} Mathematical Institute, University of Oxford, Oxford, UK. \\
\textbf{2} Max Planck Institute of Molecular Cell Biology and Genetics, Dresden, Germany. \\
\textbf{3} Centre for Systems Biology Dresden, Dresden, Germany. \\
\textbf{4} Faculty of Mathematics, Technische Universität Dresden, Dresden, Germany. \\
\textbf{5} Physics of Life, Technische Universität Dresden, Dresden, Germany. \\
\textbf{6} Isaac Newton Institute for Mathematical Sciences, University of Cambridge, Cambridge, UK.

\vspace{12pt}
\begin{abstract}
Vector-valued functions on finite simplicial complexes give rise to multiparameter sublevel set filtrations and corresponding persistent homology.  
We study the associated inverse problem for multiparameter persistent homology (MPH) of $n$-filters on a fixed finite simplicial complex.
We endow both the space of filters and the moduli space of essentially finite persistence modules with stratifications, where strata are given by orbits of natural actions of order-isomorphisms of the unit $n$-cube. 
The MPH map is then shown to be equivariant and strongly stratified.
Over each stratum in the image, the MPH map restricts to a trivial fiber bundle whose fiber is a polyhedral complex.
We provide an upper bound on the dimension of the fibers in terms of multigraded Betti numbers, recovering as a special case the known one-parameter bound obtained by Leygonie and Tillmann in 2022 \cite{LeygonieTillmann2022}.
\end{abstract}

\tableofcontents

\newpage

\renewcommand{\baselinestretch}{1.1}

\section*{Introduction}

Persistent homology (PH) is a fundamental concept in topological data analysis. Starting from a filtered space derived from data, PH records how homology classes evolve throughout the filtration, and produces an algebraic descriptor known as the barcode.
Its stability and computability have made PH a useful tool in a wide range of applications (e.g. \cite{Nardini2021,Gardner2022,Pritchard2023,Hickok2024,Maksymiuk2026}).

In many situations, however, a single filtration parameter is not sufficient to encode the relevant structure of the data.
This limitation motivated the introduction of multiparameter persistent homology \cite{Carlsson2009}, where spaces are filtered simultaneously by several parameters (e.g. scale and density).
Multiparameter persistent homology (MPH) has been a rapidly developing area, with multiple review articles having already appeared \cite{Botnan2023, BauerBrustleScoccola2026}.
Recent work has advanced the structure theory and invariants of multiparameter persistence modules \cite{Bjerkevik2025,BauerScoccola2025,Bauer2026,FersztandJacquard2024,Botnan2024,Brustle2025}, established new stability results \cite{Blumberg2022,Botnan2024,Oudot2024,Fersztand2024}, and developed computational methods and learning pipelines \cite{Dey2025,FersztandJendrysiak2026,Scoccola2024,Kerber2024}. At the same time, MPH has begun to appear in applications, especially to biomedical data \cite{Keller2018,Vipond2021,Benjamin2024, Chung2022}.

A foundational question for any data descriptor is how much information is lost when passing from the original data to the descriptor. This question belongs to the broader class of inverse problems, which include questions of injectivity, such as describing the preimages of a descriptor value, and questions of realisability, such as determining which descriptor values can arise from a prescribed class of data. Inverse problems have been a recurrent theme in persistent homology.

In this paper, we study the inverse problem for multiparameter persistent homology.
On the side of realisability, the question was settled positively for homology of degree at least $1$ when MPH was introduced in \cite{Carlsson2009}.
The remaining case of degree 0 has recently been explored in \cite{BauerBotnan2026}, where the authors showed that 0th homology is far from surjective.

We focus on the complementary question of injectivity. To the best of our knowledge, this is the first such study for MPH.
Specifically, the data in our setting consist of vector-valued filters on a fixed finite simplicial complex,
which determine multiparameter sublevel set filtrations and generalise the study of the fiber problem by Leygonie and the second author in \cite{LeygonieTillmann2022} from one parameter to several parameters.

As to be expected, there are several subtleties that need to be addressed and taken into account when moving from one-parameter to the multiparameter case.
In the one-parameter case, the persistent homology of a filter is uniquely described by barcodes.
Passing to several parameters requires us to work directly with the more complex moduli space of isomorphism classes of persistence modules for which no discrete complete invariant exists.
In this setting,
the minimal grid of a module plays part of the role of barcode endpoints in one parameter,
the interleaving distance no longer admits a description in terms of matchings of bars, and reparametrisations of the interval have to be replaced by order-preserving reparametrisations of higher dimensional cubes. 
Nevertheless, with this more complicated set-up we are able to describe the geometry of the fibers of MPH and in particular obtain a useful bound on their dimension.

\vspace{6pt}
\noindent
{\bf Content and results.} We give a brief summary of the results.
Given a fixed simplicial complex $K$, let $\filt^n$ denote the space of $n$-filter functions on $K$. These are functions 
\begin{equation*}
    f: K \to I^n, \quad \text{where} \quad I = [0,1],
\end{equation*}
that are monotone with respect to inclusion and the product order on $I^n$. Each $n$-filter function (or $n$-filter for short) gives rise to a sublevel set filtration of $K$: for $t \in \mathbb{R}^n$, let
\begin{equation*}
    K(f)_t = \{\sigma \in K \mid f(\sigma) \leq t\}.
\end{equation*}
Passing to degree $p \geq 0$ homology over a fixed field $\Bbbk$ produces an $n$-parameter persistence module, that is, a functor from the poset $(\mathbb{R}^n, \leq)$ to the category of finite-dimensional $\Bbbk$-vector spaces $\vect$,
\begin{equation*}
    H_p \circ K(f): (\mathbb{R}^n, \leq) \to \vect.
\end{equation*}
The finiteness of $K$ ensures this persistence module is essentially finite, meaning that it is determined by its restriction to a finite grid. Indeed, any such grid containing $\im f\subset I^n$ will serve.
The minimal such grid will be a central object for this work (\autoref{thm:min-grid}).
Our main technical input (\autoref{klem:nearby-grids}) establishes a local persistence property of minimal grids: the minimal grid of a module must be contained, up to a small displacement, in the minimal grid of any sufficiently close module.

We define $\modspace{n}$ to be the moduli space of isomorphism classes of essentially finite persistence modules that admit a grid in $I^n$ equipped with the interleaving metric. For $n=1$, this space is the space of barcodes whose finite endpoints lie in $I$ with the bottleneck distance.
We define the multiparameter persistent homology map by descending to isomorphism classes and combining all homological degrees (\autoref{defn:MPH-map}):
\begin{equation*}
    MPH := (MPH_0, \dots, MPH_{\dim K}): \filt^n \to \modspace{n}^{\dim K +1},
\end{equation*}
where $MPH_p(f) := [H_p \circ K(f)]$ for each $p = 0,\dots, \dim K$.
With this set-up, studying the information loss of multiparameter persistent homology translates to studying the fiber over a point in the image 
\begin{equation*}
    \modspace{n}^K := MPH(\filt^n).
\end{equation*}
To do so, we study more closely the structure of both $\filt^n$ and $\modspace{n}^K$ via actions induced by transformations of the parameter space $I^n$ which in turn leads us to study its group of automorphisms first.

The group of order isomorphisms of $I^n$ contains one connected component for each permutation of the $n$ axes of the cube, each of which is isomorphic to $\aut (I, \leq)^n$. Indeed, as a group,
\begin{equation*}
    \aut (I^n, \leq) \cong \aut (I, \leq)^n \rtimes \Sigma_n.
\end{equation*}
Since the parameters by which we filter a simplicial complex carry some meaning (e.g. scale and density) and are generally not interchangeable, we ignore the permutations and restrict ourselves to the identity component
$\aut_0 (I^n, \leq) \cong \aut (I, \leq)^n$.
We show that its closure in $\en(I^n, \leq)$ is 
\begin{equation*}
    \overline{\aut_{0}(I^n, \leq)} \cong \epi(I, \leq)^n,
\end{equation*}
where $\epi(I,\leq)$ is the set of surjective order-preserving maps of the interval.

Both $\aut_0(I^n,\leq)$ and its closure act naturally on $\filt^n$ by post-composition and on $\modspace{n}$ by moving, and in the closure possibly collapsing, the minimal grid of a module.
Crucially, the $MPH$ map is shown to be equivariant with respect to these actions (\autoref{thm:equivariance}).

An $n$-filter induces $n$ preorders on the simplices of $K$, which remain fixed within a $\aut_0(I^n,\leq)$-orbit. Thus, an orbit in $\filt^n$ is parametrised by the ordered sequences of non-boundary values that filters take in each coordinate. Similarly, an $\aut_0(I^n,\leq)$-orbit in $\modspace{n}$ is parametrised by the interior points of minimal grids. \autoref{klem:nearby-grids} is then used to prove that this parametrisation is continuous. In both settings, the resulting coordinate map identifies each orbit homeomorphically with a product of $n$ open standard simplices of possibly different dimensions.

From \autoref{klem:nearby-grids}, we can also determine sufficient conditions for when a given module is in the $\aut_0(I^n,\leq)$- or $\overline{\aut_0(I^n,\leq)}$-orbit of a sufficiently close module (\autoref{lem:closer-than-granularity}, \autoref{lem:squeeze-grid}), which allows us to understand how different orbits glue together.
We then show $\filt^n$ and $\modspace{n}^K$ are naturally stratified with the strata given by the orbits of the respective action of $\aut_0(I^n,\leq)$, and that the closure of each orbit is the orbit of the closure of $\aut_0(I^n,\leq)$. As a consequence of its equivariance, we deduce $MPH$ is a strongly stratified map, sending each stratum in $\filt^n$ surjectively onto a stratum of $\modspace{n}^K$ (\autoref{thm:MPH-stratified-map}).

Under the identification of strata in $\filt^n$ and $\modspace{n}^K$ with products of simplices, 
the $MPH$ map restricted to the closure of a filter stratum corresponds to a projection recording precisely those filter values that belong to the minimal grid of the resulting persistence module. 
Consequently, we find that the $MPH$ map restricted to a stratum in $\modspace{n}^K$ is a trivial fiber bundle with fiber a polyhedral complex (\autoref{thm:topology-of-fiber}).

Building on recent results by Guidolin and Landi on homological Morse numbers \cite{GuidolinLandi2023}, we obtain an upper bound on the dimension of the fibers (\autoref{prop:bound-dim-fiber}): for $[\bm{M}] \in \modspace{n}^K$,
\begin{equation*}
    \dim MPH^{-1}([\bm{M}]) \leq \frac{n}{2} (|K|-\mathcal{L}(\bm{M})),
\end{equation*}
where $|K|$ is the number of simplices in $K$ and $\mathcal{L}(\bm{M})$ is a non-negative integer that solely depends on the Betti tables of $\bm{M}$. For $n=1$, $\mathcal{L}(\bm{M})$ is the number of endpoints in the total barcode, thus recovering the bound from \cite{LeygonieTillmann2022}. We also show that for an $n$-filter $f$ that is injective in each coordinate, the bound on the dimension of the fiber of $MPH(f)$ is 0.

In the final section, we study the example of 2-filters on the 1-simplex, and explicitly describe the strata in the image of $MPH$, as well as the fibers over each stratum. In this case, we show that the bound on the dimension of the fiber is attained for every module stratum. For an arbitrary $K$, we do not know if the bound is sharp for all module strata.

\vspace{6pt}
\noindent
{\bf Related work and outlook.}
As in \cite{LeygonieTillmann2022}, we work in a setting where the simplicial complex $K$ is fixed and consider sublevel set filtrations on $K$. That work was followed by the design and implementation of an algorithm for recovering the polyhedra in the fiber of persistent homology \cite{Leygonie2024}. Injectivity questions have also been studied in several other settings for one-parameter persistence. We briefly review recent advances in these directions, which we expect to admit suitable multiparameter analogues.

A continuous version of the interval problem was studied in \cite{Curry2018}, where a bound was given on the number of connected components of the fiber of persistent homology for continuous functions on the interval. A related version for certain families of functions on the sphere was studied in \cite{Catanzaro2020}. The fibers of persistent homology for Morse functions on compact manifolds with boundary were studied in \cite{Leygonie2022}, where each path component was shown to be an orbit under isotopies of the manifold. The same authors later studied the problem when the space is not a manifold but an arbitrary geometric tree \cite{Beers2025}.

For point cloud data, \cite{Gameiro2016} introduced a continuation method for inverse problems from persistence diagrams to point clouds. In \cite{Smith2024}, the authors described generic families of point clouds with identical, or even trivial, one-dimensional persistence. Most recently, \cite{Beers2024} studied spaces of point clouds with a fixed barcode using tools from real algebraic geometry and rigidity theory.

\vspace{6pt}
\noindent
{\bf Acknowledgements.}
The last author would like to thank Francesca Tombari for insightful suggestions, and Jan Jendrysiak and Lukas Waas for helpful discussions and valuable comments on earlier versions of this manuscript.
MT is funded by the EPSRC grant EP/W524311/1. MT, HAH and UT are grateful for the support provided by the UK Centre for Topological Data Analysis, EPSRC grant EP/Z531224/1.
This paper has been proofread with the help of ChatGPT.

\newpage

\section{The space of multiparameter persistence modules}
\label{sec:mpms}

In this section we review the basic notions of multiparameter persistence modules (defined in $\mathbb{R}^n$) and the interleaving distance. We focus on essentially finite modules, that is, those which are completely determined by their restriction to a finite grid. 
Among all such grids there is a unique minimal one, which will play a central role throughout the paper. 
We analyse how interleavings restrict the possible minimal grids of nearby modules.
Finally, we introduce Betti tables, invariants of multiparameter persistent modules which will be key to our analysis of the fiber of $MPH$.

\subsection{Basic definitions}

Throughout this paper, we fix $n\geq 1$ the number of parameters. We write $\leq$ for the product order in $\mathbb{R}^n$, that is, for $x,y \in \mathbb{R}^n$, $x \leq y$ if and only if $x_i \leq y_i$ for all $i=1,\dots, n$.

\begin{defn}[$n$-parameter persistence module]
    Let $\vect$ be the category of finite-dimensional vector spaces and linear maps. For any poset category $P$, a $P$-persistence module is a functor $P \to \vect$. Together with natural transformations between them, $P$-persistence modules form a category which is denoted by $\vect^{P}$. In particular, an $n$-parameter persistence module is an $\mathbb{R}^n$-persistence module ${M: (\mathbb{R}^n, \leq)\to \vect}$. We denote the vector space of $M$ at grade $r \in \mathbb{R}^n$ by $M_r$, and for $s\in \mathbb{R}^n$ with $r \leq s$, we write $M_{r,s}: M_r \to M_s$ for the structure map. We denote the isomorphism class of $M$ by $[M]$.
\end{defn}

\begin{rem}
    Note that we are restricting our definition of persistence modules to require that they be pointwise finite-dimensional, that is, such that $M_r$ is a finite-dimensional vector space for all $r\in \mathbb{R}^n$. All persistence modules obtained as the homology of a filtration of a finite data set satisfy this condition.
\end{rem}

One can endow $n$-parameter persistence modules with an algebraically-defined distance known as the interleaving distance \cite{Chazal2009, Lesnick2015}.

\begin{defn}[$\epsilon$-shift functor]
Let $\vec v \in \mathbb{R}^n$. The $\vec v$-shift functor
    \begin{equation*}
        (-)[\vec v]: \vect^{\mathbb{R}^n} \to \vect^{\mathbb{R}^n}
    \end{equation*}
    is defined as follows. For any $M \in \vect^{\mathbb{R}^n}$, $M[\vec v]: (\mathbb{R}^n, \leq) \to \vect$, assigns to any $r \in \mathbb{R}^n$ the vector space $(M[\vec v])_r = M_{r + \vec v}$, 
    and to comparable grades $r \leq s$ the linear map $(M[\vec v])_{r,s} = M_{r + \vec v, s + \vec v}$.
    For natural transformations $\eta: M \Rightarrow N$, $\eta[\vec v]_r = \eta_{r+\vec v}$ for each $r \in \mathbb{R}^n$.
    Finally, note that the internal maps of $M$ canonically assemble into a natural transformation $S^{M,\vec v}: M \to M[\vec v]$, with $S^{M,\vec v}_r = M_{r, r + \vec v}$.

    Let $\epsilon \geq 0$ and denote $\vec\epsilon = (\epsilon, \epsilon, \dots, \epsilon) \in \mathbb{R}^n$. A shift by $\pm \vec\epsilon$ will be denoted simply as $M[\pm \epsilon]$. 
\end{defn}

\begin{defn}[$\epsilon$-interleaving]
    For any $\epsilon \geq 0$, an $\epsilon$-interleaving between functors $M, N: (\mathbb{R}^n, \leq) \to \vect$ consists of a pair of natural transformations
    \begin{equation*}
        \gamma: M \to N[\epsilon], \quad \kappa: N \to M[\epsilon]
    \end{equation*}
    such that
    \begin{equation*}
        \kappa[\epsilon] \circ \gamma = S^{M, 2\vec\epsilon}, \quad
        \gamma[\epsilon] \circ \kappa = S^{N, 2\vec\epsilon}.
    \end{equation*}
\end{defn}

\begin{rem}
\label{rem:interleaving-iso-class}
    A 0-interleaving is an isomorphism.
\end{rem}

\begin{defn}[interleaving distance]
    The interleaving distance between two $n$-persistence modules $M,N$ is
    \begin{equation*}
        d_I (M,N) = \inf \{ \epsilon \mid \text{there exists an $\epsilon$-interleaving between $M$ and $N$}\}.
    \end{equation*}
\end{defn}

This defines an extended pseudometric on $\vect^{\mathbb{R}^n}$.

\begin{rem}
    If $M, N$ are $\epsilon$-interleaved, and $M' \cong M$, $N' \cong N$, then $M'$ and $N'$ are $\epsilon$-interleaved. Hence, the interleaving distance is well defined on isomorphism classes of $n$-persistence modules.
\end{rem}

\subsection{Essentially finite modules and minimal grids}

For the rest of the paper, we will focus on a simpler type of modules known as essentially finite persistence modules. These are modules that can be described completely by specifying their vector spaces and maps on a finite grid of $\mathbb{R}^n$. First, we need to consider a method to extend filtrations or persistence modules defined on a grid to the whole space.

\begin{defn}[grid]
    A (finite) grid in $\mathbb{R}^n$ is a product $G = G_1 \times G_2 \times \cdots \times G_n$, where each $G_i \subset \mathbb{R}$ is a finite ordered set. 
\end{defn}

\begin{defn}[restriction of a module along a grid]
    Let $M \in \vect^{\mathbb{R}^n}$ and let $G$ be a grid in $\mathbb{R}^n$, with the induced partial order. Denote the inclusion functor by $\inc_G: (G, \leq) \hookrightarrow (\mathbb{R}^n, \leq)$. The restriction of $M$ onto $G$ is the persistence module
    \begin{equation*}
        M|_G := M \circ \inc_G: (G, \leq) \to \vect.
    \end{equation*}
\end{defn}

\begin{defn}[extension along a grid]
    Let $G = G_1 \times G_2 \times \cdots \times G_n$ be a grid in $\mathbb{R}^n$. If $G \neq \varnothing$, then it has a unique minimal element, denoted by $\min G$. Given $r\in \mathbb{R}^n$ with $r\geq \min G$, let
    \begin{equation*}
        \lfloor r \rfloor_G = \max \{g \in G \mid g \leq r\}
    \end{equation*}
    denote the floor of $r$ with respect to $G$.
    Let $N \in \vect^G$. Then, its extension along $G$ is an $n$-parameter persistence module
    \begin{equation}
        \widehat{N}^G: (\mathbb{R}^n, \leq) \to \vect
    \end{equation}
    defined as follows. For all $r,s \in \mathbb{R}^n$ with $r \leq s$,
    \begin{equation*}
        \widehat{N}^G_r =
        \begin{cases}
            N_{\lfloor r \rfloor} & \text{ if } r \geq \min G, \\
            0 & \text{ otherwise,}
        \end{cases}
        \qquad
        \widehat{N}^G_{r,s} =
        \begin{cases}
            N_{\lfloor r \rfloor, \lfloor s \rfloor}  & \text{ if } r \geq \min G,\\
            0 & \text{ otherwise.}
        \end{cases}
    \end{equation*}
    For natural transformations $\gamma: N \Rightarrow N'$ between $N,N'\in \vect^G$,
    \begin{equation*}
        \widehat{\gamma}^G_{r} =
        \begin{cases}
            \gamma_{\lfloor r \rfloor} & \text{ if } r \geq \min G\\
            0 & \text{ otherwise.}
        \end{cases}
    \end{equation*}
    We omit $G$ and write $\lfloor - \rfloor$ or $\widehat{(-)}$ when the choice of grid is clear from context.
\end{defn}

\begin{rem}
Extension along a grid $G$ defines a functor $\widehat{(-)}^G: \vect^{G} \to \vect^{\mathbb{R}^n}$. In categorical terms this is a left Kan extension, often denoted by $\operatorname{\text{Lan}}_G$.
\end{rem}

\begin{defn}[essentially finite]
    An $n$-parameter persistence module $M \in \vect^{\mathbb{R}^n}$ is essentially finite if there exists a grid $G$ such that $M$ is isomorphic to the extension along $G$ of its own restriction onto $G$,
    \begin{equation*}
        M \cong \widehat{M|_G}.
    \end{equation*}
    We say $G$ is a grid for $M$.
\end{defn}

\begin{rem}
The restriction of a module along the empty grid is the unique functor $\varnothing \to \vect$. Its extension along the empty grid is the 0 module. Hence, the empty grid is a grid solely for the 0 module.
\end{rem}

\begin{rem}
    Let $M$ be a one-parameter persistence module whose barcode only contains intervals of the form $[b,d)$ with $b < d$. Then the endpoints of bars in the barcode form a grid for $M$.
\end{rem}

\begin{rem}
\label{rem:grid-iso-class}
    Let $M, N$ be two isomorphic essentially finite modules. Then $G$ is a grid for $M$ if and only if it is a grid for $N$.
    Hence, we can talk about grids for the isomorphism class of an $n$-persistence module.
\end{rem}

The interleaving distance for essentially finite modules is particularly well-behaved, as is captured by the following result.

\begin{thm}[Closure Theorem, Theorem 6.1 in \cite{Lesnick2015}]
\label{thm:closure-thm}
    If two essentially finite $n$-parameter persistence modules are at an interleaving distance $\epsilon$, then they are $\epsilon$-interleaved. Applied to $\epsilon = 0$, this implies that the interleaving distance descends to an extended metric on isomorphism classes of essentially finite $n$-persistence modules.
\end{thm}

The following is an alternative characterisation of essentially finite modules which will be useful.

\begin{lem}[characterisation of essentially finite persistence modules]
\label{lem:ess-fin-characterisation}
Let $M \in \vect^{\mathbb{R}^n}$. $M$ is an essentially finite persistence module with grid $G$ if and only if
    \begin{enumerate}[label=(\roman*)]
        \item for every $r,s \in \mathbb{R}^n$ with $\min G \leq r \leq s$ such that $\lfloor r \rfloor = \lfloor s \rfloor$, the structure map $M_{r,s}$ is an isomorphism.
        \item for $r \in \mathbb{R}^n$ such that $r \not\geq \min G$, $M_r = 0$,
    \end{enumerate}
\end{lem}

\begin{proof}
The forward implication follows directly from the definition of extension.
Conversely, assume (i) and (ii) and let us build $\eta: \widehat{M|_G}\Rightarrow M$. For $r\geq\min G$, define
$\eta_r=M_{\lfloor r\rfloor_G,r}$, and for $r\not\geq\min G$ let
$\eta_r$ be the 0 map. These form a natural transformation, and by (i) and (ii), every $\eta_r$ is an isomorphism, hence $\eta$ is a natural
isomorphism.
\end{proof}

The rest of this section will be dedicated to showing that there exists a unique minimal grid for an essentially finite module.

\begin{lem}
\label{lem:non-minimal-grid}
    Let $M$ be an essentially finite module with grid $G = G_1 \times \dots \times G_n$. Let $e_j$ denote the vector that has a 1 as the $j$-th coordinate and 0 otherwise. Assume that for some $j$ there exist $x\in G_j$ and a $\delta > 0$ such that there are no points of $G_j$ in $(x-\delta,x)$ and for all $g\in G$ with $j$-th coordinate $g_j = x$, the map $M_{g-\delta e_j, g}$ is an isomorphism. Then the grid that results from deleting $x$, 
    \begin{equation*}
    G' = G_1 \times \dots\times (G_j \setminus \{x\}) \times\dots\times G_n
    \end{equation*}
    is a grid for $M$.
\end{lem}

\begin{proof}
    We can assume without loss of generality that $j = 1$, $x \in G_1$.
    First, we show that if for all $g\in G$ with first coordinate $g_1 = x$, the map $M_{g-\delta e_1, g}$ is an isomorphism, then for all $r \in \mathbb{R}^n$ with $r_1 = x$, $M_{r-\delta e_1, r}$ is an isomorphism. If $r\not\geq \min G$, necessarily $r_k < \min G_k$ for some $k \neq 1$, so $r-\delta e_1 \not\geq \min G$. At both grades $M$ is 0 and $M_{r-\delta e_1, r} = 0$ is an isomorphism.
    
    Assume $r\geq \min G$. Since there are no points of $G_1$ in $(x-\delta, x)$,
    $\lfloor \lfloor r \rfloor_G - \delta e_1 \rfloor_G = \lfloor r - \delta e_1 \rfloor_G$ and hence
    \begin{equation*}
        M_{r-\delta e_1, r} = M_{\lfloor r \rfloor_G,r} \circ M_{\lfloor r \rfloor_G - \delta e_1, \lfloor r \rfloor_G} \circ (M_{\lfloor r \rfloor_G - \delta e_1,r-\delta e_1})^{-1}
    \end{equation*}
    is an isomorphism.
    
    To prove $G'$ is a grid for $M$, we use the characterisation from \autoref{lem:ess-fin-characterisation}. We assume $G' \neq \varnothing$ for now, or equivalently $|G_1| \geq 2$.
    If they exist, we denote the points in $G_1$ immediately preceding or following $x$ as $x^-$ and $x^+$, and otherwise we set $x^- = -\infty$, $x^+ = +\infty$.

    To prove (i), let $r,s \in \mathbb{R}^n$ satisfy $\min G' \leq r \leq s$ and $\lfloor r \rfloor_{G'} = \lfloor s \rfloor_{G'}$. 
    Write $\lfloor r \rfloor_{G'_1}$ for the first coordinate of $\lfloor r \rfloor_{G'}$.  Since $\min G \leq r \leq s$ and for $j>1$ we have $\lfloor r \rfloor_{G_j} = \lfloor s \rfloor_{G_j}$, the only nontrivial case is when $\lfloor r \rfloor_{G'_1} = \lfloor s \rfloor_{G'_1}$ but $(\lfloor r \rfloor_{G})_1 \ne (\lfloor s \rfloor_{G})_1$, which occurs exactly when
    \begin{equation*}
        x^- \leq r_1 < x \leq s_1 < x^+.
    \end{equation*}

    Since there are no points of $G_1$ in $(x-\delta,x)$, $\lfloor r \rfloor_G = \lfloor (x-\delta, r_2, \dots, r_n) \rfloor_G$, and also $\lfloor s \rfloor_G = \lfloor (x, r_2, \dots, r_n) \rfloor_G$. The map $M_{(x-\delta, r_2, \dots, r_n), (x, r_2, \dots, r_n)}$ is an isomorphism as discussed above. If $r_1 \leq x-\delta$, then
    \begin{equation*}
        M_{r,s} =
        M_{(x, r_2, \dots, r_n), s}
        \circ
        M_{(x-\delta, r_2, \dots, r_n), (x, r_2, \dots, r_n)}
        \circ
        M_{r, (x-\delta, r_2, \dots, r_n)}.
    \end{equation*}
    If $r_1 > x-\delta$, the last morphism is substituted by $M_{(x-\delta, r_2, \dots, r_n),r}^{-1}$. This proves $M_{r,s}$ is an isomorphism.

    To prove (ii), let $r \not\geq \min G'$. Then for some $k=1,\dots,n$, $r_k < \min G'_k$. If $k \geq 2$, then $G'_k = G_k$ and hence $\min G'_k = \min G_k$, so $r \not\geq \min G$ and $M_r = 0$ because $G$ is a grid for $M$.

    Assume $k=1$ and $r_1 < \min G'_1$. 
    If $x > \min G_1$, then again $\min G'_1 = \min G_1$ and $M_r = 0$. 
    If $x = \min G_1$ and $r_1 < x$, we likewise have $r \not\geq \min G$ and $M_r = 0$.
    
    It remains to consider $x \leq r_1 < \min G'_1 = x^+$. 
    Then $\lfloor r \rfloor_G = \lfloor (x, r_2,\dots,r_n) \rfloor_G$, and by hypothesis 
    $M_{(x-\delta, r_2,\dots,r_n), (x, r_2,\dots,r_n)}$ is an isomorphism. Here $x=\min G_1$, so $x-\delta<\min G_1$, and $(x-\delta, r_2,\dots,r_n) \not\geq \min G$. Therefore, 
    \begin{equation}
    \label{eq:travel-to-0}
        0 = M_{(x - \delta, r_2, \dots, r_n)} \cong M_{(x, r_2, \dots, r_n)} \cong M_r.
    \end{equation}

    Finally, let us consider the case $G'=\varnothing$, i.e.\ $G_1=\{x\}$. Since the extension along an empty grid is the zero module, showing that $G'$ is a grid for $M$ amounts to proving $M=0$. 
    For all $r \not\geq \min G$, $M_r=0$, and if $r \geq \min G$, the same reasoning as above yields \eqref{eq:travel-to-0}. 
    Thus $M=0$, completing the proof.
\end{proof}

\begin{thmdef}
    \label{thm:min-grid}
    Let $M$ be an essentially finite module. The set of grids for $M$ with the partial order induced by inclusion contains a unique minimal element, which we will refer to as the minimal grid for $M$.
\end{thmdef}

\begin{rem}
\label{rem:0modulegrid}
    As a convention, we will fix that the minimal grid of the 0 module is $\varnothing^n$, so every coordinate set is empty.
\end{rem}

\begin{proof}
    Since $M$ is essentially finite there exists at least one grid for $M$.
    If some grid for $M$ is empty then it is the unique minimal grid, so we are done.
    Let $G = G_1 \times \dots \times G_n$ and $H = H_1 \times \dots \times H_n$ be two non-empty grids for $M$, and assume they are both minimal with respect to inclusion.

    Assume $G \neq H$ to reach contradiction. Without loss of generality, assume there exists $x \in G_1$ such that $x \notin H_1$. Let 
    \begin{equation*}
        \delta = \min_{y \in (G_1 \setminus\{x\}) \cup H_1} | y-x | / 2.
    \end{equation*}
    Since $x \notin H_1$ and $H_1 \neq \varnothing$, the minimum is well-defined and $\delta > 0$, and it ensures there are no points of $G_1$ in $(x-\delta,x)$.

    For each $k \geq 2$, choose $g_k\in G_k$ (so $(x,g_2,\dots,g_n)$ ranges over all $g\in G$ with first coordinate $x$).

    By choice of $\delta$, the nearest $H_1$-point to $x$ has distance at least $2\delta$, hence $(x,g_2, \dots, g_n) \not\geq \min H$ if and only if $(x - \delta,g_2, \dots, g_n) \not\geq \min H$. In that case $M_{(x - \delta,g_2, \dots, g_n)} = 0$ and $M_{(x,g_2, \dots, g_n)} = 0$ and the structure map between them is an isomorphism. If $(x,g_2, \dots, g_n) \geq \min H$, then also $(x - \delta,g_2, \dots, g_n) \geq \min H$ and again by the choice of $\delta$ we have
    \begin{equation*}
        \lfloor (x-\delta, g_2, \dots, g_n) \rfloor_H = \lfloor(x, g_2, \dots, g_n) \rfloor_H
    \end{equation*}
    so $M_{(x-\delta, g_2, \dots, g_n), (x, g_2, \dots, g_n)}$ is an isomorphism.

    Hence, by \autoref{lem:non-minimal-grid}, removing $x$ from $G_1$ yields a strictly smaller grid for $M$, contradicting minimality of $G$. Therefore $G=H$, and minimal grids are unique.

    Finally, every grid contains an inclusion-minimal subgrid which is a grid for $M$. By uniqueness, this subgrid is the grid constructed above. Consequently, it is contained in every grid for $M$.
\end{proof}

\begin{rem}
    The minimal grid of an isomorphism class of $n$-persistence modules is also well defined.
    The minimal grid has appeared previously in the literature without a dedicated name (e.g. as the finite aligned subgrid of \cite[Definitions 7.1 and 7.3]{Blanchette2025}) and has also recently been called the induced grid \cite[Definition 2.5]{FersztandJendrysiak2026}. 
\end{rem}

To finish this section, we relate the intrinsic notion of minimal grid to the notion of minimal presentation of a module (more on this perspective on \autoref{subsec:betti}).

\begin{lem}[from Proposition 2.9 of \cite{LesnickWright2015}]
    An $n$-parameter persistence module is essentially finite if and only if it is finitely presentable.
\end{lem}

From the proof of this lemma, it follows that any grid of a module contains the grades of all generators and relations in a minimal presentation; and conversely, that a grid containing the generators and relations of a presentation is a grid of the module. As a consequence, we find that the minimal grid of a module can be read off from a minimal presentation.

\begin{lem}
\label{lem:min-grid-min-presentation}
    The minimal grid of a module is the smallest grid containing the grades of all generators and relations in a minimal presentation of the module.
\end{lem}

\subsection{Interleaving distance and grids}

In what follows, we introduce \autoref{klem:nearby-grids}, which describes necessary conditions on minimal grids of modules that are sufficiently close to a given module.

\begin{defn}[granularity of a grid]
\label{defn:ell_G}
    Let $G$ be a grid. We define the granularity of $G$ as
    \begin{equation*}
        \ell_G = \min_{\substack{g,g'\in G\\g\neq g'}} \|g-g'\|_\infty,
    \end{equation*}
    for $|G|\geq 2$ and set $\ell_G = \infty$ otherwise. Note that for $|G| \geq 2$, $\ell_G=
    \min_{\substack{1\leq j\leq n\\|G_j|\geq2}}
    \min_{\substack{a,b\in G_j\\a\neq b}}|a-b|$.
\end{defn}

\begin{defn}[granularity of a module]
\label{defn:ell_M}
    Let $M: (\mathbb{R}^n, \leq) \to \vect$ be an essentially finite module and $G$ its minimal grid. Then we define $\ell_M = \ell_G$. Since the minimal grid is invariant under isomorphism, granularity is well defined for isomorphism classes of essentially finite modules.
\end{defn}

\begin{klem}
\label{klem:nearby-grids}
    Let $M: (\mathbb{R}^n, \leq) \to \vect$ be essentially finite with minimal grid $G$.
    Fix $\epsilon < \ell_M/2$. Let $N: (\mathbb{R}^n, \leq) \to \vect$ be essentially finite with minimal grid $H$. Assume $d_I(M,N) \leq \epsilon$, with $\gamma: M \to N[\epsilon]$, $\kappa: N \to M[\epsilon]$ an $\epsilon$-interleaving. Then
    \begin{enumerate}[label=(\roman*)]
        \item for all $j = 1, \dots, n$ and all $x \in G_j$, there exists $y \in H_j$ such that
        \begin{equation*}
            |x-y| \leq \epsilon.
        \end{equation*}
        \item if $H\neq\varnothing$, for any $r \in \mathbb{R}^n$ such that $r-\vec\epsilon \geq \min H$ and $\lfloor r - \vec\epsilon \rfloor_H = \lfloor r + \vec\epsilon \rfloor_H$, the map
        \begin{equation*}
            \gamma_r: M_r \to N_{r + \vec\epsilon},
        \end{equation*}
        is an isomorphism. In particular, composed with $(N_{r,r+\vec\epsilon})^{-1}$, it gives an isomorphism between $N_r$ and $M_r$ which is natural in $r$.
    \end{enumerate}
\end{klem}

Before proceeding to the proof of this result, let us clarify its contents. Note, firstly, that the lemma is not symmetric with respect to the modules $M$ and $N$. We start by fixing a module and considering its granularity $\ell_M$. Then we take a second module which is closer to $M$ than half its granularity. The symmetrised version of this result will be used later in \autoref{lem:closer-than-granularity}.

The first half of the statement implies that the grid of $N$ must contain a subgrid that is $\epsilon$-close to $G$. In the $n=1$ case, this would correspond to the fact that all endpoints of $M$ are $\epsilon$-matched to endpoints of $N$.
However, the grid of $N$ might be much bigger. In $n=1$, $N$ might contain many other unmatched bars.

The second half means that if the grid $H$ has gaps bigger than $2\epsilon$, then the modules $M$ and $N$ at the centre of the gaps are isomorphic. Intuitively, if something in $N$ ``happens for longer than $2\epsilon$", then ``it must coincide with $M$".

The proof of this lemma will take up the remainder of this subsection.

\begin{proof}[Proof of \autoref{klem:nearby-grids}]
We prove (i) by the contrapositive. Without loss of generality, assume that there exists $x \in G_1$ such that for all $y \in H_1$,
\begin{equation*}
    |x - y| > \epsilon.
\end{equation*}
We will prove that in this case, $G$ is not the minimal grid of $M$.
If $H \neq \varnothing$, let $d(x,H_1) = \min_{y \in H_1} |x-y|$ and set $d(x,H_1) = \infty$ otherwise. Choose a $\delta > 0$ such that 
\begin{equation*}
    \delta < \min \{ \ell_M- 2\epsilon, d(x,H_1) - \epsilon \}.
\end{equation*}

This choice ensures that there are no points of $G_1$ other than $x$ in the interval $(x-(2\epsilon+\delta), x + (2\epsilon + \delta))$, and that there are no points of $H_1$ in the interval $(x-(\epsilon + \delta), x + (\epsilon + \delta))$.

Now let $g\in G$ with first coordinate $g_1 = x$. We can fit the structure map $M_{g - (2\epsilon + \delta)e_1, g}$ as the bottom row of the following commutative diagram with the interleaving maps: 
$$
\begin{tikzcd} [column sep=0.3cm, cells={nodes={text width=2.1cm, align=center}}]
	&& {M_{g + 2\vec\epsilon- (2\epsilon + \delta)e_1}}
    &&& {M_{g + 2\vec\epsilon}} \\
	& {N_{g + \vec\epsilon- (2\epsilon + \delta)e_1}}
    &&& {N_{g + \vec\epsilon}} \\
	{M_{g - (2\epsilon + \delta)e_1}}
    &&& {M_g}
	\arrow[from=1-3, to=1-6]
	\arrow["{^{(3)}\kappa_{g + \vec\epsilon- (2\epsilon + \delta)e_1}}"'{pos=0.5}, two heads, from=2-2, to=1-3]
	\arrow["\cong^{(4)}"{pos = 0.5},from=2-2, to=2-5]
	\arrow["^{(3)}{\kappa_{g + \vec\epsilon}}"{pos=0.5}, two heads, from=2-5, to=1-6]
	\arrow["\cong^{(2)}"{pos = 0.5}, curve={height=-30pt}, from=3-1, to=1-3]
	\arrow["^{(3)}{\gamma_{g - (2\epsilon + \delta)e_1}}"'{pos=0.5}, hook, from=3-1, to=2-2]
	\arrow[from=3-1, to=3-4]
	\arrow["\cong^{(1)}"'{pos = 0.5},curve={height=30pt}, from=3-4, to=1-6]
	\arrow["^{(3)}{\gamma_g}"{pos=0.5}, hook, from=3-4, to=2-5]
\end{tikzcd}
$$

We start by showing that the right (1) and left (2) diagonal composites are isomorphisms, and hence the two $\gamma$ maps are injective and the two $\kappa$ maps are surjective (3); and that the structure map of $N$ in the centre (4) is also an isomorphism.
\begin{enumerate}[label = (\arabic*)]
    \item By the definition of an interleaving, the composition of the two diagonal maps on the right coincides with the structure map: $\kappa_{g + \vec\epsilon} \circ \gamma_g = M_{g, g+2\vec\epsilon}$. Since $2\epsilon <\ell_M$, $g = \lfloor g \rfloor_G = \lfloor g+2\vec\epsilon \rfloor_G$ and the map is an isomorphism.
    
    \item Similarly, the composition of the two diagonal maps on the left,
    \begin{equation*}
        \kappa_{g + \vec\epsilon - (2\epsilon + \delta)e_1} \circ \gamma_{g - (2\epsilon + \delta)e_1} = M_{g - (2\epsilon + \delta)e_1, g+2\vec\epsilon - (2\epsilon + \delta)e_1},
    \end{equation*}
    is an isomorphism. 
    First, note that if $g+2\vec\epsilon-(2\epsilon+\delta)e_1\not\geq\min G$, then also $g-(2\epsilon+\delta)e_1\not\geq\min G$, and the map is the isomorphism $0\to0$. Otherwise, both grades lie above $\min G$. Then, for $k \geq 2$, $(\lfloor g - (2\epsilon + \delta)e_1 \rfloor_{G})_k = (\lfloor g \rfloor_{G})_k$, and $(\lfloor g + 2\vec\epsilon- (2\epsilon + \delta)e_1 \rfloor_{G})_k = (\lfloor g + 2\vec\epsilon \rfloor_{G})_k$, so again since $2\epsilon <\ell_M$,
    \begin{equation*}
        (\lfloor g - (2\epsilon + \delta)e_1 \rfloor_{G})_k = (\lfloor g + 2\vec\epsilon- (2\epsilon + \delta)e_1 \rfloor_{G})_k
    \end{equation*}
    For $k=1$, the two coordinates are $x-\delta$ and $x-(2\epsilon +\delta)$. By the choice of $\epsilon$, $\delta$, there is no element of $G_1$ between them, so also
    \begin{equation*}
        (\lfloor g - (2\epsilon + \delta)e_1 \rfloor_{G})_1 = (\lfloor g + 2\vec\epsilon- (2\epsilon + \delta)e_1 \rfloor_{G})_1
    \end{equation*}
    and the map is an isomorphism.
    
    \item The fact that the two compositions in the diagonals are isomorphisms (1,2) allows us to deduce injectivity of the maps $\gamma_{g - (2\epsilon + \delta)e_1}$, $\gamma_g$ and surjectivity of the maps $\kappa_{g + \vec\epsilon- (2\epsilon + \delta)e_1}$, $\kappa_{g + \vec\epsilon}$.

    \item Note that if $g+\vec\epsilon\not\geq\min H$, then also $g+\vec\epsilon-(2\epsilon+\delta)e_1\not\geq\min H$, and the map is the isomorphism $0\to0$. Otherwise, both grades lie above $\min H$. Then, since there are no points of $H_1$ in the interval $(x-(\epsilon + \delta), x + (\epsilon + \delta))$, this implies
    \begin{equation*}
        \lfloor g +\vec\epsilon - (2\epsilon + \delta)e_1 \rfloor_H
        = \lfloor g + \vec\epsilon \rfloor_H,
    \end{equation*}
    and hence the central horizontal map $N_{g +\vec\epsilon - (2\epsilon + \delta)e_1, g +\vec\epsilon}$ is also an isomorphism.
\end{enumerate}

We are ready to prove from the diagram that the map $M_{g - (2\epsilon + \delta)e_1, g}$ is an isomorphism. Consider the commutativity of the upper square, 
\begin{equation*}
    \kappa_{g +\vec\epsilon} \circ N_{g +\vec\epsilon - (2\epsilon + \delta)e_1, g +\vec\epsilon} = M_{g +2\vec\epsilon - (2\epsilon + \delta)e_1, g +2\vec\epsilon} \circ \kappa_{g +\vec\epsilon - (2\epsilon + \delta)e_1}.
\end{equation*}
By (3) and (4), the left hand side is surjective, so $M_{g + 2\vec\epsilon- (2\epsilon + \delta)e_1, g + 2\vec\epsilon}$ is surjective. Now by the commutativity of the outer square, and (1) and (2), surjectivity of $M_{g + 2\vec\epsilon- (2\epsilon + \delta)e_1, g + 2\vec\epsilon}$ implies surjectivity of $M_{g - (2\epsilon + \delta)e_1, g}$. Finally, consider commutativity of the lower square,
\begin{equation*}
    N_{g +\vec\epsilon - (2\epsilon + \delta)e_1, g +\vec\epsilon} \circ \gamma_{g - (2\epsilon + \delta)e_1}= \gamma_{g} \circ M_{g - (2\epsilon + \delta)e_1, g}.
\end{equation*}
By (3) and (4), the left hand side is injective, which implies $M_{g - (2\epsilon + \delta)e_1, g}$ is injective and hence an isomorphism.

Hence, by \autoref{lem:non-minimal-grid}, removing $x$ from $G_1$ yields a strictly smaller grid for $M$, contradicting minimality of $G$ and proving (i).

To prove (ii), let $r \in \mathbb{R}^n$, $r = (r_1,\dots, r_n)$ be such that $\lfloor r - \vec\epsilon \rfloor_H = \lfloor r + \vec\epsilon \rfloor_H$. Since $\epsilon < \ell_M/2$, we can find $s = (s_1, \dots, s_n) \in \mathbb{R}^n$ such that for all $j = 1, \dots, n$, $r_j \in [s_j, s_j + 2\epsilon]$ and $(s_j, s_j + 2\epsilon] \cap G_j = \varnothing$. Note that if $r_j \in G_j$, this is satisfied by $s_j = r_j$.

We can build the following commutative diagram:
$$
\begin{tikzcd}
	&& {N_{r + \vec\epsilon}} \\
	& {M_r} && {M_{s+ 2\vec\epsilon}} \\
	{N_{r - \vec\epsilon}} && {N_{s+ \vec\epsilon}} \\
	& {M_{s}}
	\arrow["{^{(2)}\gamma_r}", two heads, from=2-2, to=1-3]
	\arrow["{\cong^{(1)}}", curve={height=-50pt}, from=3-1, to=1-3]
	\arrow["{\kappa_{r - \vec\epsilon}}", from=3-1, to=2-2]
	\arrow["{\cong^{(6)}}", from=3-3, to=1-3]
	\arrow["{\kappa_{s+ \vec\epsilon}}"', from=3-3, to=2-4]
	\arrow["{\cong^{(5)}}", from=4-2, to=2-2]
	\arrow["{\cong^{(3)}}"', curve={height=50pt}, from=4-2, to=2-4]
	\arrow["{^{(4)}\gamma_s}"', hook, from=4-2, to=3-3]
\end{tikzcd}
$$

We start by showing that the structure maps appearing in this diagram are all isomorphisms. By construction, for each map either both grades lie below the support determined by $G$, in which case the structure map is the isomorphism $0\to0$, or both lie above $\min G$:
\begin{enumerate}[label = (\arabic*)]
    \item Since $\lfloor r - \vec\epsilon \rfloor_H = \lfloor r + \vec\epsilon \rfloor_H$, the structure map of $N$ is an isomorphism.
    \item By the commutativity of the upper diagonal, $\gamma_r \circ \kappa_{r - \vec\epsilon} = N_{r - \vec\epsilon, r + \vec\epsilon}$. Since the map is an isomorphism by (1), $\gamma_r$ is surjective.
    \item By construction, $\lfloor s \rfloor_G = \lfloor s + 2\vec\epsilon \rfloor_G$, so the structure map of $M$ is an isomorphism.
    \item By the commutativity of the lower diagonal, $\kappa_{s+ \vec\epsilon} \circ \gamma_{s} = M_{s, s + 2\vec\epsilon} $. Since the map is an isomorphism by (3), $\gamma_s$ is injective.
    \item By construction, $\lfloor s \rfloor_G = \lfloor r \rfloor_G$, so the structure map of $M$ is an isomorphism.
    \item By construction, $r- \vec\epsilon \leq s + \vec\epsilon \leq r + \vec\epsilon$, and since $\lfloor r - \vec\epsilon \rfloor_H = \lfloor r + \vec\epsilon \rfloor_H$, $\lfloor s + \vec\epsilon \rfloor_H = \lfloor r + \vec\epsilon \rfloor_H$ and the structure map of $N$ is an isomorphism.
    \item By the commutativity of the inner parallelepiped, $\gamma_r = N_{s + \vec\epsilon, r + \vec\epsilon} \circ \gamma_s \circ M_{s,r}^{-1}$. By (4), (5) and (6), the composition is injective, so $\gamma_r$ is injective, and an isomorphism.
\end{enumerate}

Since $\lfloor r \rfloor_H = \lfloor r + \vec\epsilon \rfloor_H$, we find $N_r \cong N_{r + \vec\epsilon} \cong M_r$, which finishes the proof.
\end{proof}

\begin{rem}
    Since the minimal grid of a module is the smallest grid containing the grades of the generators and relations in a minimal presentation (\autoref{lem:min-grid-min-presentation}), a version of the first part of this lemma can also be derived from the stability of multigraded Betti numbers established in \cite[Theorem 1.1]{OudotScoccola2024}. This argument, however, requires the stronger assumption that $d_I(M,N) \leq \epsilon/(n^2-1)$.
\end{rem}

\subsection{Betti tables and the Koszul complex of a persistence module}
\label{subsec:betti}

It is a standard result \cite{Carlsson2009,Corbet2018} that the category of $\mathbb{Z}^n$-persistence modules over $\Bbbk$ is equivalent to the category of $n$-graded modules over the ring of polynomials in $n$ variables over $\Bbbk$. Similarly, the category of $n$-parameter persistence module is equivalent to the category of $n$-graded modules over the ring of polynomials in $n$ variables over $\Bbbk$ with non-negative real-valued exponents \cite[Section 2.1]{Lesnick2015}.

However, since we only consider essentially finite modules, let
$G$ be a grid and let
$\rho_j:G_j\to\mathbb{Z}$ be an injective order-preserving map.
Regrading $M|_G$ along $\rho=\rho_1\times\dots\times \rho_n$ and extending it to
$\mathbb Z^n$ preserves free modules, minimal free resolutions, and the
multiplicities of their graded summands. We therefore define and compute
the Betti tables of $M$ through this regraded $\mathbb Z^n$-module.
All applications of Hilbert's syzygy theorem and of the Koszul complex
below are understood after this regrading.

Free persistence modules are direct sums of modules $F_r$ with a single generator at degree $r$ and no relations. For $t\in \mathbb{R}^n$, $(F_r)_t = \Bbbk$ for $t \geq r$ and 0 otherwise, and all the nonzero structure maps are identities.

\begin{defn}[Betti tables]
    Let $M$ be a finitely presented $n$-parameter persistence module, and let $P_\bullet \to M$ be a minimal free resolution of $M$. Then, for each $i \geq 0$, write $P_i = \bigoplus_{r\in \mathbb{R}^n} F_r^{\xi_i(r)}$,
    where $\xi_i(r) \neq 0$ for finitely many $r \in \mathbb{R}^n$. The function 
    \begin{equation*}
        \xi_i: \mathbb{R}^n \to \mathbb{N}
    \end{equation*}
    is called the $i$-th Betti table of $M$. 
    Since any two minimal free resolution of $M$ are isomorphic, Betti tables are isomorphism invariants of $M$.
    By Hilbert's syzygy theorem, we know the minimal resolution has length at most $n$, that is, $P_{i>n} = 0$, and hence $\xi_{i>n} = 0$.
\end{defn}

\begin{rem}
    For $n=1$, $\xi_0(t)$ and $\xi_1(t)$ indicate the number of bars starting and ending at $t \in \mathbb{R}$, respectively.
\end{rem}

We now introduce a method from commutative algebra that allows us to evaluate Betti tables at a given grade from local information about the module, avoiding the need to compute the resolution \cite{Miller2004}. We first justify the method for $\mathbb{Z}^n$-persistence modules, and then give an explicit description for essentially finite modules.

For a persistence module $V: \mathbb{Z}^n \to \vect$, Betti tables $\xi_i$ at $x \in \mathbb{Z}^n$ can be computed as follows: take a (not necessarily minimal) resolution of $V$, and take the tensor product $-\otimes_{\Bbbk[x_1, \dots, x_n]} \Bbbk$ to obtain a chain complex. The value of $\xi_i(x)$ corresponds to the dimension of the $i$-th homology of this chain complex at $x$, i.e. $\xi_i(x) = \dim_{\Bbbk} (\operatorname{Tor}_i(V,\Bbbk) (x))$.

To show this, note that since the computation of $\operatorname{Tor}$ is independent of the choice of resolution, we may start with a minimal free resolution of $V$. Tensoring with $\Bbbk$ removes all structure maps between different grades. Each free summand $F_r$ therefore becomes a module with a single copy of $\Bbbk$ at grade $r$ and 0 elsewhere.
The only maps that survive in the resulting chain complex are those between generators at the same grade. In a minimal resolution such maps do not occur, so every differential becomes zero after tensoring with $\Bbbk$. Taking homology does not change the chain complex, and the dimension of the $i$-th homology group at grade $x$ is simply the number of generators of grade $x$ appearing in the $i$-th module of the minimal resolution, namely $\xi_i(x)$.

By the symmetry of $\operatorname{Tor}$, this computation is equivalent to taking first a resolution of $\Bbbk$, such as the standard Koszul complex $\mathbb{K_*}$, and then taking the tensor product with the module $V$, before computing homology. At each grade, this amounts to computing homology of a chain complex of vector spaces: $\mathbb{K}_i(V)(x) := (\mathbb{K}_i \otimes_{\Bbbk[x_1, \dots, x_n]} V)_x$.

The fact that essentially finite modules are determined by their restriction to a finite grid allows us to translate this approach directly to our context.
Given an essentially finite persistence module, we can restrict it to its minimal grid, which injects coordinate-wise into $\mathbb{Z}^n$. The restriction of the module can then be extended to a $\mathbb{Z}^n$-module as a left Kan extension, whose Betti tables can be computed as described above. The Betti tables of the original module can be read off the result. This computation is explicitly described as follows.

\begin{defn}[Koszul complex of $M$ at grade $g \in G$]
\label{defn:koszul-cpx}
    Let $M$ be an essentially finite $n$-parameter persistence module with minimal grid $G$. Let $\ell_M$ be the granularity of $M$ (\autoref{defn:ell_M}) and fix $\epsilon < \ell_M$.

    For any subset $\alpha \subset \{1, \dots, n\}$, let $e_\alpha = \sum_{j \in \alpha}e_j$. The Koszul complex of $M$ at grade $g \in G$, $\mathbb{K}_*(M)(g)$, is a chain complex of vector spaces. For each $i \geq 0$,
    \begin{equation*}
        \mathbb{K}_i(M)(g) =
        \bigoplus_{\substack{\alpha\subset \{1, \dots, n\} \\ |\alpha| = i}}
        M_{g - \epsilon e_\alpha}.
    \end{equation*}

    The differentials $d_i : \mathbb{K}_i(M)(g) \to \mathbb{K}_{i-1}(M)(g)$ are defined from the internal maps of $M$ as follows. For each summand $M_{g - \epsilon e_\alpha}$ with $\alpha = \{j_1 < j_2 < \dots < j_i\}$,
    \begin{equation*}
        d_i|_{M_{g - \epsilon e_\alpha}} = \sum_{k = 1}^{i} (-1)^{i-k} M_{g - \epsilon e_\alpha, g - \epsilon e_{\alpha\setminus\{j_k\}}}
    \end{equation*}
\end{defn}

\begin{lem}[adapted from Lemma 1.32, \cite{Miller2004}]
\label{lem:koszul-cpx}
    The $i$-th Betti table $\xi_i: \mathbb{R}^n \to \mathbb{N}$ of $M$ at a point of the minimal grid $g \in G$ is
    \begin{equation*}
        \xi_i(g) = \dim H_i(\mathbb{K}_*(M)(g)).
    \end{equation*}
\end{lem}

\newpage

\section{Multiparameter persistent homology}
\label{sec:mph}

In this section we introduce the multiparameter persistent homology map $MPH$, which is the main object of study of the paper. For a fixed finite simplicial complex $K$, we first define the space of $n$-filter functions $f: K \to I^n$ and the associated sublevel set filtrations of $K$. 
Then, the homology functor applied to these filtrations yields essentially finite $n$-parameter persistence modules that admit grids contained in $I^n$.
We define the multiparameter persistent homology map $MPH$ as a map from the space of $n$-filters to the moduli space of isomorphism classes of such modules.
By the stability of multiparameter persistent homology, the map $MPH$ is Lipschitz with respect to the supremum norm on filters and the interleaving distance on persistence modules.

\subsection{The space of \texorpdfstring{$n$}{n}-filters}

Let $\scpx$ denote the category of finite (abstract) simplicial complexes and inclusions and let $K$ be a fixed finite simplicial complex. Let $I = [0,1]$ denote the closed unit interval in $\mathbb{R}$, and let $I^n$ be the $n$-cube with the $\| \cdot \|_\infty$ metric.

\begin{defn}[$n$-filter function]
    An $n$-filter function, or $n$-filter for short, is a map $f: K \to I^n$ that is monotone with respect to inclusions, i.e. for all $\tau, \sigma \in K$ such that $\tau \subset \sigma$, $f(\tau) \leq f(\sigma)$. Throughout the paper, we will denote by $f_j$ the projection of an $n$-filter $f$ to the $j$-th coordinate.
    We denote the set of all $n$-filter functions on $K$ as $\filt^n$. We define the sup-norm distance between $f, g \in \filt^n$ as
    \begin{equation*}
        d_\infty(f,g) = \max_{\tau \in K} \|f(\tau) -g(\tau) \|_\infty.
    \end{equation*}
\end{defn}

The space $\filt^n$ can be seen as a subspace of $I^{n|K|}$. An $n$-filter gives rise to an $n$-parameter filtration in the following way.

\begin{defn}[sublevel set filtration]
    Let $f: K \to I^n$ be an $n$-filter. The sublevel set filtration associated to $f$ is the functor $K(f): (\mathbb{R}^n, \leq) \to \textbf{SCpx}$ that sends each $t \in \mathbb{R}^n$ to
    \begin{equation*}
        K(f)_t = f^{-1}(\{x \in \mathbb{R}^n \mid x \leq t\}),
    \end{equation*}
    and each relation $t \leq t'$ to the inclusion map $K(f)_t \hookrightarrow K(f)_{t'}$.
\end{defn}

\begin{rem}[Filtrations vs. filters]
In the one-parameter setting, all filtrations of a finite simplicial complex are ``essentially" (up to boundary issues) sublevel set filtrations for some filter function. The \v{C}ech filtration of a point cloud of size $N$, for example, corresponds to a sublevel set filtration on the standard $(N-1)$-dimensional simplex.

In the $n$-parameter case, however, filtrations can be multicritical, with simplices having multiple ``birth" times. Although looking at filters in this case is more restrictive, there are many filtrations which belong to this category, such as the function-Rips, subdivision-Rips or the interlevel filtrations. Others such as the degree-Rips and multicover filtrations are multicritical \cite{Lesnick-notes}.
\end{rem}

\subsection{The \texorpdfstring{$MPH$}{MPH} map and the moduli space of \texorpdfstring{$n$}{n}-parameter persistence modules}

For each $p \geq 0$, the $p$-th homology is a functor $H_p: \scpx \to \vect$ that sends a finite simplicial complex to its $p$-th homology with coefficients in a fixed field $\Bbbk$. Pre-composed with $K(f)$, it yields the $p$-th persistent homology of $f$:
\begin{equation*}
    H_p \circ K(f): (\mathbb{R}^n, \leq) \to \vect.
\end{equation*}

From the construction of the sublevel set filtration, we see that any grid containing $\im f \subset I^n$ is a grid for $H_p \circ K(f)$, hence it is an essentially finite $n$-parameter persistence module. It is natural to consider persistence modules up to isomorphism.

\begin{defn}[moduli space $\modspace{n}$]
We define $\modspace{n}$ to be the space of isomorphism classes of essentially finite $n$-parameter persistence modules that admit a grid contained in $I^n$, endowed with the interleaving distance, $d_I([M], [N]) = d_I(M,N)$. By \autoref{thm:closure-thm}, $\modspace{n}$ is an extended metric space.
\end{defn}

\begin{rem}
    The space $\modspace{1}$ can be identified with the space of barcodes, that is, finite multisets of intervals, whose finite endpoints are contained in $I$, equipped with the bottleneck distance \cite{CrawleyBoevey2015,Chazal2016, Lesnick2015}. 
    A barcode records the unique decomposition of a one-parameter persistence module into interval modules (the only indecomposable one-parameter modules).
    For $n>1$, persistence modules still decompose into indecomposables \cite{Botnan2020}, but this decomposition does not reduce the algebraic complexity, since multiparameter persistence modules have wild representation type \cite{Botnan2023}. There does not exist a complete barcode-like invariant of multiparameter persistent modules given by a set of `nice' regions of $\mathbb{R}^n$.
\end{rem}

\begin{rem}
    The extended metric space $\modspace{n}$ is not complete. A characterization of its closure is given in \cite{Bauer2026}.
\end{rem}

\begin{defn}[$MPH$ map]
    \label{defn:MPH-map}
    The $p$-th multiparameter persistent homology map is defined as
    \begin{equation*}
    \begin{array}{rcl}
        MPH_p: \filt^n &\to& \modspace{n} \\
        f & \mapsto & [H_p \circ K(f)],
    \end{array}
    \end{equation*}
    and combining all non-trivial homological degrees, we define the multiparameter persistent homology map, central to this work:
    \begin{equation}
    \label{eq:MPH}
        MPH := (MPH_0, \dots, MPH_{\dim K}): \filt^n \to \modspace{n}^{\dim K +1}.
    \end{equation}
\end{defn}

\begin{defn}[image of $MPH$]
    We denote the image of the $MPH$ map as
    \begin{equation}
\label{eq:MPH-image}
    \modspace{n}^K := MPH(\filt^n) \subseteq \modspace{n}^{\dim K +1}.
\end{equation}
\end{defn}

A few more words about the product space $\modspace{n}^{\dim K+1}$. Any $([M^0], \dots, [M^{\dim K}]) \in \modspace{n}^{\dim K +1}$ can be written as the component-wise isomorphism class of the tuple $\bm{M} = (M^0, \dots, M^{\dim K})$, denoted by $[\bm{M}]$. We endow $\modspace{n}^{\dim K +1}$ with the supremum product metric, which we denote by $d_I$ with a slight abuse of notation. Let $[\bm{M}], [\bm{N}] \in \modspace{n}^{\dim K +1}$, then
\begin{equation*}
    d_I([\bm{M}],[\bm{N}]) = \max_{i} d_I(M^i, N^i).
\end{equation*}

The definitions of minimal grid and granularity generalise directly to tuples of essentially finite persistence modules.

\begin{defn}[minimal grid for tuples]
\label{defn:grid_tuple}
Let $\bm{M} = (M^0, \dots, M^d)$ be a tuple of essentially finite persistence modules admitting a grid in $I^n$. A grid for $\bm{M}$ is a grid which is simultaneously a grid for each $M^i$. The minimal grid $G$ of $\bm{M}$ is the smallest grid containing the minimal grid of each $M^i$. The granularity of $\bm{M}$ is defined as the granularity of the minimal grid $\ell_{\bm{M}} = \ell_G$.
\end{defn}

\begin{rem}
    The minimal grid and granularity are well-defined for isomorphism classes of tuples of persistence modules, $\ell_{[\bm{M}]} = \ell_{\bm{M}}$.
\end{rem}

We finish this section with one of the most central results in persistent homology: the Stability Theorem, adapted to our context.

\begin{thm}[Stability Theorem, Theorem 5.3 in \cite{Lesnick2015}]
\label{thm:stability}
The interleaving distance $d_I$ on $n$-parameter persistence modules is stable, i.e. for all $n$-filters $f,g \in \filt^n$,
\begin{equation*}
    d_I(MPH_p(f), MPH_p(g)) \leq d_\infty(f, g).
\end{equation*}
Equivalently, the map $MPH_p$ is Lipschitz continuous, and thus so is $MPH$.
\end{thm}

\begin{rem}
When the coefficient field $\Bbbk$ is prime, the interleaving distance is universal among stable pseudometrics on multidimensional persistence modules: any pseudometric satisfying the same stability result is bounded above by $d_I$ \cite[Corollary~5.6]{Lesnick2015}. In this precise sense, $d_I$ is the most discriminating stable distance on multiparameter persistence modules. The author of \cite{Lesnick2015} also conjectured this to be true for an arbitrary field.
\end{rem}

\section{Equivariance of the \texorpdfstring{$MPH$}{MPH} map}
\label{sec:actions}

In this section we introduce natural actions of order-preserving reparametrisations of the $n$-cube on the spaces of $n$-filters and $n$-parameter persistence modules. These actions will underlie the stratifications considered later.

We begin by analysing the structure of the group of order isomorphisms of $I^n$ and the closure of its identity component in the space of order-preserving maps of $I^n$. We then define actions on $\filt^n$ and on $n$-parameter persistence modules, and study their effect on essentially finite modules and their minimal grids. In particular, we show that, up to isomorphism, the action on an essentially finite module is determined by the action on its minimal grid.

These actions descend to the moduli space of isomorphism classes $\modspace{n}$, and the $MPH$ map is equivariant with respect to the corresponding actions on $\filt^n$ and $\modspace{n}$. We next prove continuity of the actions, and conclude with results relating interleaving distance to the action on $\modspace{n}$, giving conditions under which modules close to a fixed module can be obtained from it by the action.

\subsection{Order endomorphisms and isomorphisms of \texorpdfstring{$I^n$}{the n-cube}}

\begin{defn}[order endofunctors]
    Let $\en (I^n, \leq)$ be the space of endofunctors of $(I^n, \leq)$, that is, maps $\phi: I^n \to I^n$ such that for all $x,y \in I^n$,
    \begin{equation*}
        x \leq y \implies \phi(x) \leq \phi(y),
    \end{equation*}
    endowed with the sup-norm distance,
    \begin{equation*}
        d_\infty(\phi,\psi) = \sup_{x \in I^n} \|\phi(x) - \psi(x)\|_\infty.
    \end{equation*}
\end{defn}

We will consider this action only for a subset of endofunctors of $I^n$ with more algebraic structure: the group of order isomorphisms, as well as its closure in $\en(I^n, \leq)$.

\begin{defn}[order isomorphisms]
    Let $\aut (I^n, \leq)$ denote the group of automorphisms of $(I^n, \leq)$. These correspond to maps $\phi: I^n \to I^n$ which are order-preserving, bijective, and with an order-preserving inverse. Equivalently, maps $\phi: I^n \to I^n$ that are bijective and such that for all $x,y \in I^n$,
    \begin{equation*}
        x \leq y \iff \phi(x) \leq \phi(y).
    \end{equation*}
\end{defn}

\begin{rem}
    For a poset category like $(I^n, \leq)$, any endofunctor that yields an equivalence of categories is in fact an automorphism (invertible on the nose), so one could  define $\aut (I^n, \leq)$ equivalently as the set of endofunctors of $I^n$ that yield equivalences.
\end{rem}

For $n=1$, the total order in the interval rigidly determines the interaction of its order-preserving morphisms with the real line topology. Any surjective map $\phi: I \to I$ that preserves order is continuous. Hence, all maps in $\aut (I, \leq)$ are homeomorphisms. Conversely, any homeomorphism of $I$ to itself that fixes endpoints (or, equivalently, preserves orientation) is an order isomorphism. Hence, we can think of $\aut (I, \leq)$ as the space of orientation preserving homeomorphisms of $I$. Here are two examples that show how the relation between order and topology is not so straightforward for $n>1$.

\begin{ex}
    The map $\phi: I^2 \to I^2$ defined by
    \begin{equation*}
    \phi(x,y) =
    \begin{cases}
        (x,y^2) & \text{if } x = 0, \\
        (x,y) & \text{otherwise,}
    \end{cases}
    \end{equation*}
    is a surjective and order-preserving map, but it is not continuous.
\end{ex}

\begin{ex}
\label{ex:end-homeo-not-iso}
    The map $\phi: I^2 \to I^2$ defined by $\phi(x,y) = (x, y^{\frac{1}{x+0.1}})$ is an order-preserving homeomorphism but it is not an order isomorphism. 
    Let $r(x) = 1/(x+0.1)$. Note that for $y \in I$, $y^{r(x)}$ increases when $r(x)$ decreases, and $r(x)$ decreases when $x$ increases, so for $(x,y) \leq (x',y')$
    \begin{equation*}
        \phi(x,y) = (x, y^{\frac{1}{x+0.1}}) \leq (x', y^{\frac{1}{x'+0.1}}) \leq (x', y'^{\frac{1}{x'+0.1}}) = \phi(x',y').
    \end{equation*}
    The map $\phi$ is continuous in $I^2$ and it is bijective with continuous inverse $\phi^{-1}(x,y) = (x, y^{x+0.1})$.
    However, $\phi$ creates new relations between non-comparable elements, for example $a = (0, 0.8)$ and $b = (1,0.1)$, $a \not\leq b$ and $b \not\leq a$, but $\phi(a) \leq \phi(b)$, so it is not an order isomorphism, the inverse map is not order-preserving.
\end{ex}

However, order isomorphisms of $I^n$ do have a rigid structure.

\begin{prop}
\label{prop:order-iso}
    Let $\phi \in \aut (I^n, \leq)$. There exists a permutation $\sigma: \{1,2,\dots,n\} \to \{1,2,\dots,n\}$ and $n$ order-preserving bijections $\phi_j: I \to I$ such that for all $x \in I^n$, $x = (x_1, x_2, \dots, x_n)$,
    \begin{equation*}
        \phi (x) = \left( \phi_1(x_{\sigma(1)}), \phi_2(x_{\sigma(2)}), \dots, \phi_n(x_{\sigma(n)}) \right).
    \end{equation*}
    Moreover, every order-preserving bijection $\phi_j: I \to I$ is a homeomorphism, so $\phi$ is a homeomorphism.
\end{prop}

\begin{proof}
For $x,y \in I^n$ define the join $x \vee y$ as the unique minimal element of the set $\{ z \in I^n \mid x \leq z, y \leq z \}$. More explicitly,
\begin{equation*}
    x \vee y = \left( \max(x_1, y_1), \dots, \max(x_n,y_n)\right).
\end{equation*}

One can check that an order isomorphism $\phi$ preserves joins, that is, for all $x,y \in I^n$, $\phi (x \vee y) = \phi (x) \vee \phi (y)$. Let $e_j$ denote the point which has a 1 in the $j$-th coordinate and 0 for the rest. Then, any point $x \in I^n$, $x = (x_1, \dots, x_n)$ can be written as $x = \vee_{j=1}^n x_j e_j$ and
\begin{equation*}
    \phi(x) = \vee_{j=1}^n\phi(x_j e_j).
\end{equation*}

Thus, $\phi$ is determined by its restriction to the axes of the cube. Points in the axes are irreducible with respect to the join operation, meaning they cannot be written as a join of two strictly smaller elements. The corners $e_j$ are the maximal join-irreducibles. One can check that an order isomorphism also preserves join-irreducibles and maximal join-irreducibles. Since $\phi$ is a bijection, necessarily it maps each axis to another axis bijectively, according to a permutation $\sigma$. Let $\sigma(j)$ be such that the $\sigma(j)$-th axis is mapped by $\phi$ onto the $j$-th axis. Define $\phi_j: I \to I$ as the restriction of $\phi$ to the $\sigma(j)$-th axis, projected onto the $j$-th axis. Then
\begin{equation*}
    \phi(x) = \vee_{j=1}^n\phi_j(x_{\sigma(j)}) e_{j} = \left( \phi_1(x_{\sigma(1)}), \phi_2(x_{\sigma(2)}), \dots, \phi_n(x_{\sigma(n)}) \right),
\end{equation*}

which is the desired form.\footnote{We learnt this argument from an answer given by Prof. Keith Kearnes to \href{https://math.stackexchange.com/questions/2041788/order-isomorphisms-of-0-1n}{a question in math.stackexchange.com.}} Since every $\phi_j \in \aut(I, \leq)$, they are homeomorphisms, and so is $\phi$.

\end{proof}

\begin{cor}[group structure and topology of the action group]
\label{cor:structure-action-group}
The group of order isomorphisms of $I^n$ can be written as a semidirect product
\begin{equation*}
    \aut (I^n, \leq) \cong \aut (I, \leq)^n \rtimes \Sigma_n,
\end{equation*}
where $\Sigma_n$ denotes the group of permutations of $n$ elements. As a topological space, $\aut (I^n, \leq) \subset {\en (I^n, \leq)}$ is formed of $n!$ path-connected components isomorphic to $\aut (I, \leq)^n$, each one corresponding to a permutation of the axes of the $n$-cube.
\end{cor}

\begin{proof}
We regard $\aut (I, \leq)^n$ as a subgroup of $\aut (I^n, \leq)$ by the inclusion $(\phi_1, \dots, \phi_n) \mapsto \phi_1 \times \dots \times \phi_n$. Let $h: \Sigma_n \to \aut (I^n, \leq)$ be the inclusion of the permutation group as permutations of the axes. For $\sigma \in \Sigma_n$, $h_\sigma(x_1, \dots, x_n) = (x_{\sigma(1)}, \dots, x_{\sigma(n)})$.

By \autoref{prop:order-iso}, every element $\phi \in \aut (I^n, \leq)$ uniquely decomposes as $(\phi_1 \times \dots \times \phi_n) \circ h_\sigma$ for some $\sigma \in \Sigma_n$ and $\phi_j \in \aut(I, \leq)$. It remains to check that the subgroup $\aut (I, \leq)^n \subset \aut (I^n, \leq)$ is normal. It suffices to check that $\Sigma_n$ normalises $\aut (I, \leq)^n$. Let $\phi = \phi_1\times \dots \times \phi_n \in \aut (I, \leq)^n$, $\sigma \in \Sigma_n$, and $x \in I^n$. Then
\begin{align*}
    h_\sigma \circ \phi \circ h_\sigma^{-1} (x) & = h_\sigma \circ \phi \circ h_{\sigma^{-1}} (x) = h_\sigma (\phi_1(x_{\sigma^{-1}(1)}), \dots, \phi_n(x_{\sigma^{-1}(n)}) ) = (\phi_{\sigma(1)}(x_{1}), \dots, \phi_{\sigma(n)}(x_n)) = \\ 
    &= (\phi_{\sigma(1)} \times \dots \times \phi_{\sigma(n)})(x),
\end{align*}
so $h_\sigma \circ \phi \circ h_\sigma^{-1} \in \aut (I, \leq)^n$.

Every map $\phi \circ h_\sigma$ is connected to $\id \circ~h_\sigma$ via the straight-line homotopy, so each copy of $\aut(I, \leq)^n$ is path-connected. Consider now the continuous map that evaluates each morphism in $\aut(I^n, \leq)$ at the points $e_1, \dots, e_n$. For any morphism $\phi \circ h_\sigma$ this evaluation is $(e_{\sigma(1)}, \dots, e_{\sigma(n)})$, so the image of this map is a discrete set of $n!$ points. By continuity, the evaluation map is constant on paths. It follows that $\aut(I^n, \leq)$ has exactly $n!$ path-connected components, each homeomorphic to $\aut(I,\leq)^n$ and corresponding to a permutation of the axes.

\end{proof}

In most contexts, each of the parameters by which we filter a simplicial complex carries some distinct meaning (e.g. scale and density). Hence, it is reasonable to consider only those order isomorphisms which preserve the parameters' labels.

\begin{defn}
    We denote by $\aut_{0}(I^n, \leq)$ the connected component of $\aut (I^n, \leq)$ containing the identity. These correspond to order isomorphisms which fix the corners of the $n$-cube.
\end{defn}

We will be interested in the closure of this connected component, which has a simple monoid structure.

\begin{lem}
\label{lem:closure-aut}
    Let $\epi(I, \leq)$ denote the surjective maps of $\en (I, \leq)$. Then
    \begin{equation*}
        \overline{\aut_{0}(I^n, \leq)} \cong \epi(I, \leq)^n = \overline{\aut(I, \leq)}^n.
    \end{equation*}
\end{lem}

\begin{rem}
    Here the closure of $\aut_{0}(I^n, \leq)$ is taken in $\en(I^n,\leq)$, where maps are not required to be continuous. We find, however, that the closure only contains continuous maps (surjective order-preserving maps of the interval are continuous).
\end{rem}

\begin{proof}
    Note that the second equality is a consequence of the first equality with $n=1$. Hence, it suffices to prove the first equality. Recall $\aut_{0}(I^n, \leq) \cong \aut (I, \leq)^n$. Note, however, that the closure in $\en(I^n, \leq)$ need not, a priori, be a product. Let $\pi_i: I^n \to I$ and $\inc_i: I \to I^n$ be the $i$-th axis projection and inclusion respectively. We regard $\epi (I, \leq)^n$ as a subspace of $\en (I^n, \leq)$ by the inclusion $(\phi_1, \dots, \phi_n) \mapsto \phi_1 \times \dots \times \phi_n$.
    
    Consider a sequence $\phi^k  \in \aut_{0}(I^n, \leq)$ which converges to $\phi \in \en(I^n,\leq)$ in the sup-norm.
    We can write $\phi^k = \phi_1^k \times \dots \times \phi_n^k$ uniquely for some $\phi_i^k \in \aut(I,\leq)$.
    A priori, we cannot assume $\phi$ is a product, or even continuous. For any $\epsilon>0$ there exists an $N$ such that for all $k\geq N$, $d_\infty(\phi^k, \phi) < \epsilon$. This implies that for every axis $i = 1, \dots, n$, $|\pi_i \circ ( \phi^k - \phi)(x)| < \epsilon$ for all $x \in I^n$. In particular, we can restrict $x$ to the $i$-th axis. We have $\phi^k_i = \pi_i \circ \phi^k \circ \inc_i$ and let $\tilde{\phi} = \phi_1 \times \dots \times \phi_n$, where $\phi_i: = \pi_i \circ \phi \circ \inc_i$. Then for all $x_i \in I$,
    \begin{equation*}
        |\phi^k_i(x_i) - \phi_i (x_i)| < \epsilon,
    \end{equation*}
    This implies that each sequence $\phi^k_i$ uniformly converges to $\phi_i$, so $\phi_i$ is order-preserving, continuous and surjective, and $\tilde{\phi} \in \epi(I,\leq)^n$. Moreover, since this is satisfied for all $i$, we also find that $d_\infty(\phi^k, \tilde{\phi}) < \epsilon$, so $\phi^k$ converges to $\tilde{\phi}$ in the sup-norm. By uniqueness of the limit, $\phi = \tilde{\phi}$, which proves $\overline{\aut_0 (I^n, \leq)} \subseteq \epi(I, \leq)^n$.

    Let us now consider $\phi = \phi_1 \times \dots \times \phi_n \in \epi(I,\leq)^n$. We define a sequence of functions $\phi^k = \phi_1^k \times \dots \times \phi_n^k$ as follows. For each $i = 1,\dots,n$, and $x_i \in I$,
    \begin{equation*}
        \phi^k_i(x_i) = \frac1k x_i + \left(1- \frac1k \right)\phi_i(x_i).
    \end{equation*}

    Each $\phi^k_i$ fixes endpoints and is continuous, so it is surjective, and it is strictly increasing, so injective and an order-isomorphism, $\phi^k \in \aut (I, \leq)^n$. For any $\epsilon > 0$, let $k > 1/\epsilon$,
    \begin{equation*}
        | \phi^k_i(x_i) - \phi_i(x_i) | = \left| \frac1k (x_i - \phi_i(x_i)) \right| \leq \frac1k < \epsilon,
    \end{equation*}
    for all $x_i \in I$, so $d_\infty(\phi^k, \phi) < \epsilon$ and the sequence converges to $\phi$ in the sup-norm. We find $\epi(I, \leq)^n \subseteq \overline{\aut_0 (I^n, \leq)}$, which finishes the proof.
\end{proof}

\subsection{Actions on \texorpdfstring{$\filt^n$}{n-filters} and \texorpdfstring{$\vect^{\mathbb{R}^n}$}{n-parameter persistence modules} and equivariance of the \texorpdfstring{$MPH$}{MPH} map}

A simplicial complex $K$ can be seen as a poset with order given by inclusions $(K, \subseteq)$. Then, any $n$-filter $f \in \filt^n$ is simply a functor $f: (K, \subseteq) \to (I^n, \leq)$. The set of endofunctors of $(I^n, \leq)$ acts on the set of filters naturally by postcomposition. Let $\phi \in \en (I^n, \leq)$ and $f \in \filt^n$, then
\begin{equation*}
    \phi.f = \phi \circ f.
\end{equation*}

Post-composition descends to a group action of $\aut_0(I^n, \leq)$ and a monoid action of $\overline{\aut_0(I^n, \leq)}$. In the rest of this section we define their actions on the space of $n$-parameter persistence modules $\vect^{\mathbb{R}^n}$, and check that they descend to well-defined actions in $\modspace{n}$, which are determined by what happens on minimal grids. Finally, we prove the equivariance of the $MPH_p$ map with respect to these actions.

First, note that since the boundary of $I$ is fixed by elements in $\epi (I, \leq)$, they extend by the identity to surjective endomorphisms of $\mathbb{R}$. We can use this procedure component-wise to extend the elements of $\overline{\aut_0 (I^n, \leq)}$ to maps of $\mathbb{R}^n$ which are continuous and non-decreasing in every coordinate. In the case of $\aut_0 (I^n, \leq)$ the extended maps are also homeomorphisms. We will tacitly extend our maps without changing notation when the context requires it.

\begin{defn}
    Let $\phi \in \overline{\aut_0 (I^n, \leq)}$, $\phi = \phi_1 \times \dots \times \phi_n$ with $\phi_j \in \epi(I,\leq)$. Note that for every $t \in \mathbb{R}^n$, $\phi^{-1}(t)$ is a product of $n$ closed, bounded intervals. This set has a unique maximal value, so we define the map $\max \phi^{-1}: \mathbb{R}^n \to \mathbb{R}^n$ as the assignment $t \mapsto (\max \phi_1^{-1} (t_1), \dots, \max \phi_n^{-1} (t_n))$.
\end{defn}

\begin{rem}[Remark 1.6]
    The map $\max \phi^{-1}$ defines an endofunctor of $(\mathbb{R}^n, \leq)$.
\end{rem}

\begin{prop}[action on persistence modules]
\label{prop:defn-action-modules}
    The monoid $\overline{\aut_0 (I^n, \leq)}$ acts on $\vect^{\mathbb{R}^n}$ as follows. Let $\phi \in \overline{\aut_0 (I^n, \leq)}$ and $M\in \vect^{\mathbb{R}^n}$, then
    \begin{equation*}
        \phi.M = M \circ \max \phi^{-1},
    \end{equation*}
    which restricts to a group action of $\aut_0 (I^n, \leq)$. The action descends to a well-defined action on $\modspace{n}$, and hence it acts diagonally on $\modspace{n}^{\dim K +1}$. Let $[\bm{M}] = [(M^0, \dots, M^{\dim K})]\in \modspace{n}^{\dim K +1}$, then
    \begin{equation*}
        \phi.[\bm{M}] = (\phi.[M^0], \dots, \phi.[M^{\dim K}]) = ([\phi.M^0], \dots, [\phi.M^{\dim K}]) = [\phi.\bm{M}]
    \end{equation*}
\end{prop}

\begin{proof}
    It is clear that the identity acts trivially. Now let $\phi, \psi \in \overline{\aut_0 (I^n, \leq)}$ and $M \in \vect^{\mathbb{R}^n}$. To prove $(\phi \circ\psi).M = \phi. (\psi.M)$ it suffices to show that
    \begin{equation*}
        \max (\phi \circ \psi)^{-1} = \max \psi^{-1} \circ \max \phi^{-1}
    \end{equation*}
    Let $t \in \mathbb{R}^n$. Since $\max \phi^{-1} (t) \in \phi^{-1} (t)$, we get $\psi^{-1} (\max \phi^{-1} (t)) \subset \psi^{-1}(\phi^{-1} (t))$, and
    \begin{equation*}
        \max \psi^{-1} (\max \phi^{-1} (t)) \leq \max (\phi \circ \psi)^{-1} (t)
    \end{equation*}
    Let $x \in (\phi \circ \psi)^{-1} (t)$, we have $\psi (x) \in \phi^{-1} (t)$, so $\psi (x) \leq \max \phi^{-1} (t)$. Since $\max \psi^{-1}$ is order-preserving, we get
    \begin{equation*}
        x \leq \max \psi^{-1} (\psi (x)) \leq \max \psi^{-1} (\max \phi^{-1} (t))
    \end{equation*}
    Since the inequality holds for every $x \in (\phi \circ \psi)^{-1} (t)$, we find
    \begin{equation*}
        \max (\phi \circ \psi)^{-1} (t) \leq \max \psi^{-1} (\max \phi^{-1} (t)),
    \end{equation*}
    which yields the equality.
    
    Finally, we check that the action is well-defined for isomorphism classes of persistence modules. Let $M, N \in \vect^{\mathbb{R}^n}$ be isomorphic, with $\eta: M \Rightarrow N$ a natural isomorphism. Then for any $\phi \in \overline{\aut_0 (I^n, \leq)}$, we define $\eta \circ \max \phi^{-1}: \phi.M \Rightarrow \phi.N$ at each $t \in \mathbb{R}^n$ by the map $\eta_{\max \phi^{-1}(t)}$, which yields a natural isomorphism between $\phi.M$ and $\phi.N$.
\end{proof}

As a direct generalisation of Lemma 1.5 in \cite{LeygonieTillmann2022}, we find that the multiparameter persistent homology map is equivariant under the actions we have defined.

\begin{thm}[equivariance]
\label{thm:equivariance}
The $MPH$ map is equivariant under the action of $\overline{\aut_0 (I^n, \leq)}$ on $n$-filters and persistence modules, that is, for all $f \in \filt^n$ and $\phi \in \overline{\aut_0 (I^n, \leq)}$.
\begin{equation*}
    MPH(\phi.f) = \phi.MPH(f).
\end{equation*}
\end{thm}

\begin{proof}
We prove equivariance at the level of persistence modules. Let $t \in \mathbb{R}^n$ and let $p \geq 0$ denote a fixed homology degree. By \autoref{lem:closure-aut}, $\phi = \phi_1 \times \dots \times \phi_n$ for $\phi_j \in \epi(I,\leq)$. We have
\begin{equation*}
    (\phi.(H_p \circ K(f)))_t = (H_p \circ K(f)) \circ \max \phi^{-1}(t) = H_p ( K(f)_{\max \phi^{-1}(t)}), 
\end{equation*}

where
\begin{align*}
    K(f)_{\max \phi^{-1}(t)} &=
    \{\sigma \in K \mid f(\sigma) \leq \max \phi^{-1}(t)\} = \\
    &=\{\sigma \in K \mid f_j(\sigma) \leq \max \phi_j^{-1}(t_j) \text{ for all } j = 1,\dots, n\} = \\
    &=\{\sigma \in K \mid \phi_j \circ f_j(\sigma) \leq t_j \text{ for all } j = 1,\dots, n\} \\
    &=\{\sigma \in K \mid \phi \circ f(\sigma) \leq t\} = \\
    &= K(\phi \circ f)_t.
\end{align*}

Hence, we find $(\phi.MPH_p(f)) = [H_p(K(\phi \circ f))] = MPH_p(\phi.f)$, and $MPH(\phi.f) = \phi.MPH(f)$.
\end{proof}

\begin{rem}
    We can actually define an action of the whole $\en(\mathbb{R}^n, \leq)$ on $\vect^{\mathbb{R}^n}$ by left Kan extensions (up to isomorphism): for $\phi \in \en(I^n, \leq)$ and $M \in \vect^{\mathbb{R}^n}$, $\phi.M := \operatorname{Lan}_\phi M$. Concretely, for $x \in \mathbb{R}^n$,
    \begin{equation*}
        (\phi.M)_x = \underset{\phi(y) \leq x}{\operatorname{colim}} M_y.
    \end{equation*}
    For the extension of $\phi \in \overline{\aut_0(I^n,\leq)}$, this definition coincides with \autoref{prop:defn-action-modules}.
    The equivariance of the action, however, fails in general. We illustrate this with an example.

    Let $\phi$ be the order-preserving homeomorphism of $\mathbb{R}^n$ obtained as the following extension of the map defined in \autoref{ex:end-homeo-not-iso}. For $(x,y) \in \mathbb{R}^2$,
    \begin{equation*}
        \phi(x,y) =
        \begin{cases}
            (x,y) & \text{if } y\leq 0 \text{ or } y\geq 1, \\
            (x,y^{10}) & \text{if } y \in (0,1) \text{ and } x \leq 0, \\
            (x,y^{1/(x+0.1)}) & \text{if } x, y \in (0,1), \\
            (x,y^{1/1.1}) & \text{if } y \in (0,1) \text{ and } x \geq 1.
        \end{cases}
    \end{equation*}

    Let $a = (0, 0.8)$ and $b = (1,0.1)$, $c=(1,1)$, and let $o$ be the origin. Note that $a,b$ are non-comparable, but $\phi(a) \leq \phi(b)$. Let $f$ be a filter on the triangle as sketched in \autoref{fig:counterexample-equivariance}.

    \begin{figure}[h!]
        \centering
        \includegraphics[width=0.7\linewidth]{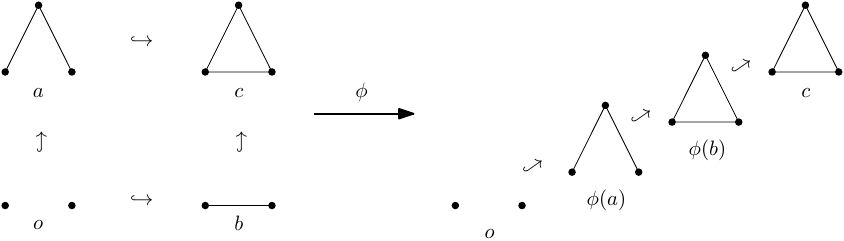}
        \caption{Example of a 2-filter on the triangle and $\phi \in \en(I^n,\leq)$ for which the equivariance of $MPH$ fails.}
        \label{fig:counterexample-equivariance}
    \end{figure}

    The module $MPH_1(f)$ is free with a single generator at grade $c$, and one can verify that acting by $\phi$ does not change the isomorphism type of the module. Indeed, for $x \not\geq c$, any $y \in \mathbb{R}^n$ satisfying $\phi(y) \leq x$ must also satisfy $y \not\geq c$. The vector space $(\phi.MPH_1(f))_x$ then is the colimit of a diagram with only 0 vector spaces, hence it is 0. For $x \geq c$, $x$ is the maximum of the set $\{y \mid \phi(y) \leq x\}$. Hence, $(\phi.MPH_1(f))_x = \Bbbk$, with identities as structure maps.
    The module $MPH_1(\phi.f)$, however, is free with a single generator at grade $\phi(b)$. Hence $MPH_1(\phi.f) \neq \phi.MPH_1(f)$.

    Note that this is a consequence of the failure of degree-wise homology not sending homotopy colimits to colimits of vector spaces. This issue would not occur at the level of the derived ($\infty$-)category of chain complexes. We thank Lukas Waas for pointing this out.
\end{rem}

\subsection{Continuity of the actions}

The continuity of the maps in $\overline{\aut_0 (I^n, \leq)} \cong \epi(I,\leq)^n$ ensures the continuity of the action of this monoid on $\filt^n$.
We now prove the continuity of the monoid action on persistence modules. 

\begin{thm}[continuity of the action]
\label{thm:continuity-action-modules}
    The map induced by the action,
    \begin{equation*}
    \begin{array}{rcl}
        \overline{\aut_0 (I^n, \leq)} \times \vect^{\mathbb{R}^n} &  \to & \vect^{\mathbb{R}^n}\\
        (\phi, M) & \mapsto & \phi.M,
    \end{array}
    \end{equation*}
    \begin{enumerate}[label=(\roman*)]
        \item is Lipschitz with respect to the first component: Fix $M \in \vect^{\mathbb{R}^n}$. For all $\phi, \psi \in \overline{\aut_0 (I^n, \leq)}$,
        \begin{equation*}
            d_I(\phi.M, \psi.M) \leq d_\infty(\phi, \psi).
        \end{equation*}
        \item is uniformly continuous with respect to the second component: Fix $\phi \in \overline{\aut_0 (I^n, \leq)}$ and let $\epsilon >0$. There exists $\eta > 0$ such that for all $x \in \mathbb{R}^n$,
        \begin{equation}
        \label{eq:ineq}
            \phi(x + \vec\eta) \leq \phi(x) + \vec\epsilon,
        \end{equation}
        and for all $M,N \in \vect^{\mathbb{R}^n}$ such that $d_I(M,N) < \eta$,
        \begin{equation*}
            d_I(\phi.M, \phi.N) \leq \epsilon.
        \end{equation*}
        \item and hence is continuous, and the action of $\overline{\aut_0 (I^n, \leq)}$ on $\vect^{\mathbb{R}^n}$ is continuous.
    \end{enumerate}
\end{thm}

\begin{proof}

{
We start by proving (i). Let $\delta = d_\infty(\phi, \psi)$. It is enough to show that there is a $\delta$-interleaving between $\phi.M, \psi.M$. We are thus seeking to define a pair of natural transformations
\begin{equation*}
    \gamma: \phi.M \Rightarrow (\psi.M)[\delta], \quad \kappa: \psi.M \Rightarrow (\phi.M)[\delta]
\end{equation*}
such that
\begin{equation*}
    \kappa[\delta] \circ \gamma = S^{\phi.M, 2\vec\delta}, \quad
    \gamma[\delta] \circ \kappa = S^{\psi.M, 2\vec\delta}.
\end{equation*}

We start by defining $\gamma$. For each $t \in \mathbb{R}^n$, there is an obvious choice for
\begin{equation*}
    \gamma_t: M_{\max \phi^{-1}(t)} \to M_{\max \psi^{-1}(t + \vec\delta)},
\end{equation*}
namely the internal map of $M$, if it exists. To show it exists, we need to show 
\begin{equation*}
    \max \phi^{-1}(t) \leq \max \psi^{-1}(t + \vec\delta).
\end{equation*}

First, recall that since $d_\infty(\phi, \psi) \leq \delta$, for every $x\in \mathbb{R}^n$,
\begin{equation*}
    \psi(x) \leq \phi(x) + \vec\delta,
\end{equation*}
Setting $x = \max \phi^{-1}(t)$, we find 
\begin{equation*}
    \psi(\max \phi^{-1} (t)) \leq \phi(\max\phi^{-1} (t)) + \vec\delta = t + \vec\delta.
\end{equation*}
As $\psi$ is order preserving, so is $\max \psi ^{-1}$. Hence,  applying $\max \psi^{-1}$ on each side of the above inequality, we get
\begin{equation*}
    \max \phi^{-1}(t) \leq \max \psi^{-1} (\psi(\max \phi^{-1} (t))) \leq \max \psi^{-1}(t + \vec\delta),
\end{equation*}
where the first inequality is by definition of $\max\psi^{-1}$. Hence, we can define
\begin{equation*}
    \gamma_t = M_{\max \phi^{-1}(t), \max \psi^{-1}(t + \vec\delta)}.
\end{equation*}

Similarly, we can define
\begin{equation*}
    \kappa_t: M_{\max \psi^{-1}(t)} \to M_{\max \phi^{-1}(t + \vec\delta)}
\end{equation*}
to be $\kappa_t = M_{\max \psi^{-1}(t), \max \phi^{-1}(t + \vec\delta)}$. The relations between natural transformations are automatically satisfied, since all the maps coincide with internal maps of $M$. Hence, $d_I(\phi.M, \psi.M) \leq \delta$.
}

{
To prove (ii), recall that each component $\phi_j$ is continuous in $I$ and the identity outside of $I$, and thus uniformly continuous everywhere. Hence, we can pick $\eta > 0$ such that condition \eqref{eq:ineq} is satisfied for all $x \in \mathbb{R}^n$.

Assume now that $d_I(M,N) < \eta$. Then there exists an interleaving between $M$ and $N$ given by natural transformations
\begin{equation*}
    \gamma: M \Rightarrow N^{\vec\eta}, \quad \kappa: N \Rightarrow M^{\vec\eta}.
\end{equation*}

We will use these natural transformations to define an $\epsilon$-interleaving between $\phi.M$ and $\phi.N$:
\begin{equation*}
    \tilde\gamma: \phi.M \Rightarrow (\phi.N)[\epsilon], \quad \tilde\kappa: \phi.N \Rightarrow (\phi.M)[\epsilon].
\end{equation*}

Let $t \in \mathbb{R}^n$. We will define $\tilde\gamma_t$ using the following diagram
\begin{center}
\begin{tikzcd}
M_{\max \phi^{-1}(t)} \arrow{rdd}[swap]{\gamma_{\max \phi^{-1}(t)}} \arrow{rr}{\tilde\gamma_t} &               & N_{\max \phi^{-1}(t+\vec\epsilon)} \\
                          &               &   \\
                          & N_{\max \phi^{-1}(t)+ \vec\eta} \arrow[ruu] &  
\end{tikzcd}
\end{center}

We can choose the upwards map to be the internal map of $N$. To see it exists, we consider the inequality in \eqref{eq:ineq} for $x = \max \phi^{-1}(t)$ with $t \in \mathbb{R}^n$. Applying $\max \phi^{-1}$ to both sides of the above inequality, we find
\begin{equation*}
    \max \phi^{-1}(t) + \vec\eta \leq \max  \phi^{-1}(\phi(\max  \phi^{-1}(t) + \vec\eta)) \leq  \max \phi^{-1} (t + \vec\epsilon).
\end{equation*}

Hence, we can define
\begin{equation*}
    \tilde\gamma_t = N_{\max \phi^{-1}(t) + \vec\eta, \max \phi^{-1} (t + \vec\epsilon)} \circ \gamma_{\max \phi^{-1}(t)},
\end{equation*}

and with an analogous reasoning we can also define
\begin{equation*}
    \tilde\kappa_t = M_{\max \phi^{-1}(t) + \vec\eta, \max \phi^{-1} (t + \vec\epsilon)} \circ \kappa_{\max \phi^{-1}(t)}.
\end{equation*}

These define natural transformations and they satisfy the commutativity condition with the internal maps:
\begin{align*}
    (\tilde\kappa[\epsilon] \circ \tilde\gamma)_t &=
    (M_{\max \phi^{-1}(t+ \vec\epsilon) + \vec\eta, \max \phi^{-1} (t + 2\vec\epsilon)} \circ \kappa_{\max \phi^{-1}(t + \vec\epsilon)}) \circ
    (N_{\max \phi^{-1}(t) + \vec\eta, \max \phi^{-1} (t + \vec\epsilon)} \circ \gamma_{\max \phi^{-1}(t)}) = \\
    &= M_{\max \phi^{-1}(t+ \vec\epsilon) + \vec\eta, \max \phi^{-1} (t + 2\vec\epsilon)} \circ ( \kappa_{\max \phi^{-1}(t + \vec\epsilon)} \circ
    \gamma_{\max \phi^{-1} (t + \vec\epsilon) - \vec\eta}) \circ
    M_{\max \phi^{-1}(t), \max \phi^{-1} (t + \vec\epsilon) - \vec\eta} = \\
    &= M_{\max \phi^{-1}(t+ \vec\epsilon) + \vec\eta, \max \phi^{-1} (t + 2\vec\epsilon)} \circ
    M_{\max \phi^{-1} (t + \vec\epsilon) - \vec\eta, \max \phi^{-1} (t + \vec\epsilon) + \vec\eta} \circ
    M_{\max \phi^{-1}(t), \max \phi^{-1} (t + \vec\epsilon) - \vec\eta} = \\
    &= (S^{\phi.M, 2\epsilon})_t,
\end{align*}

where for the first equality we have used that $\gamma$ is a natural transformation, and for the second we have used that $\gamma$ and $\kappa$ form an interleaving. Similarly $(\tilde\gamma[\epsilon] \circ \tilde\kappa)_t = (S^{\phi.N, 2\epsilon})_t$. We conclude that $\phi.M$ and $\phi.N$ are $\epsilon$-interleaved, and hence $d_I(\phi.M, \phi.N) \leq \epsilon$.

}

{
Finally, we use (i) and (ii) to prove (iii), the continuity of the map induced by the action. Let $\phi \in \overline{\aut_0 (I^n, \leq)}$, $M \in \vect^{\mathbb{R}^n}$ and $\epsilon > 0$. As above, we can choose $\eta >0$ such that for all $x \in \mathbb{R}^n$,
\begin{equation*}
    \phi(x + \vec\eta) \leq \phi(x) + \frac{\vec\epsilon}{2},
\end{equation*}
Now, let $\delta < \min(\eta, \epsilon/2)$. By (i) and (ii), for all $\psi \in \overline{\aut_0 (I^n, \leq)}$ and $N \in \vect^{\mathbb{R}^n}$ with 
\begin{equation*}
    d_\infty(\phi, \psi) < \delta, \quad 
    d_I(M,N) < \delta,
\end{equation*}

we have
\begin{equation*}
    d_I(\phi.M, \psi.N) \leq d_I(\phi.M, \phi.N) + d_I(\phi.N, \psi.N) \leq \frac{\epsilon}{2} + d_\infty(\phi, \psi) < \epsilon.
\end{equation*}
}
\end{proof}

\begin{rem}
\label{rem:continuity-iso-classes}
    The continuity of the action on $\vect^{\mathbb{R}^n}$ ensures the continuity of the action on the moduli space $\modspace{n}$, as well as in $\modspace{n}^{\dim K +1}$.
\end{rem}

\subsection{The action on grids}

The definition of the action for the one-parameter case can be described in terms of the barcode (as in \cite{LeygonieTillmann2022}). Given a one-parameter persistence module $M$ with barcode $B$, and $\phi \in \overline{\aut (I, \leq)}$, $\phi.M$ is the module whose barcode is
\begin{equation*}
    \{[\phi(b), \phi(d)) \mid [b,d) \in B, \phi(b) \neq \phi(d) \}.
\end{equation*}
Analogously, in the $n$-parameter setting the action corresponds to stretching or expanding a grid of the module, allowing for collapses but no crossings.

\begin{prop}
\label{prop:grid-under-action}
    Let $M$ be an essentially finite persistence module and $G$ a grid for $M$. For any $\phi \in \overline{\aut_0 (I^n, \leq)}$, $\phi.M$ is essentially finite with grid $\phi(G)$. Moreover, if $G$ is a minimal grid for $M$ and $\phi \in \aut_0(I^n,\leq)$, then $\phi(G)$ is a minimal grid for $\phi.M$.
\end{prop}

\begin{proof}
    We use the characterisation of essentially finite from \autoref{lem:ess-fin-characterisation}. To check (ii), note that for any $t \in \mathbb{R}^n$, if $\max \phi^{-1}(t) \geq \min G$, then $t = \phi (\max \phi^{-1}(t)) \geq \phi(\min G) = \min \phi(G)$, due to $\phi$ being order-preserving. By the contrapositive, for any $t \not\geq \min \phi(G)$, $\max \phi^{-1}(t) \not\geq \min G$, and hence $(\phi.M)_t = M_{\max \phi^{-1}(t)} = 0$.

    To check (i), it suffices to check that for all $t \in \mathbb{R}^n$ such that $t \geq \min \phi(G)$, the map 
    \begin{equation*}
        (\phi.M)_{\lfloor t \rfloor_{\phi(G)}, t} = M_{\max \phi^{-1} (\lfloor t \rfloor_{\phi(G)}), \max \phi^{-1} (t)}
    \end{equation*}
    is an isomorphism. Since $M$ is essentially finite with grid $G$, it suffices to check that
    \begin{equation}
    \label{eq:structure-map-phi-grid}
        \lfloor \max \phi^{-1} (\lfloor t \rfloor_{\phi(G)}) \rfloor_{G} = \lfloor \max \phi^{-1} (t) \rfloor_{G}.
    \end{equation}
    It is clear that $\lfloor \max \phi^{-1} (\lfloor t \rfloor_{\phi(G)}) \rfloor_{G} \leq \lfloor \max \phi^{-1} (t) \rfloor_{G}$. To show the opposite inequality, we first show that
    \begin{equation}
    \label{eq:floor-phi}
        \lfloor t \rfloor_{\phi(G)} = \phi (\lfloor \max \phi^{-1} (t) \rfloor_{G}).
    \end{equation}
    Note that the right hand side is an element of $\phi(G)$. Since $\lfloor \max \phi^{-1} (t) \rfloor_{G} \leq \max \phi^{-1} (t)$, we have $\phi (\lfloor \max \phi^{-1} (t) \rfloor_{G}) \leq \phi (\max \phi^{-1} (t)) = t$. To check it is the maximal grid element below $t$, assume there exists some $g \in G$ such that
    \begin{equation*}
        \phi (\lfloor \max \phi^{-1} (t) \rfloor_{G}) \leq \phi(g) \leq t.
    \end{equation*}
    Using the second inequality, we have that in the original grid, $g \leq \max \phi^{-1}(\phi(g)) \leq \max \phi^{-1} (t)$. By the definition of the floor function, this implies $g \leq \lfloor \max \phi^{-1} (t) \rfloor_{G}$, and hence $\phi(g) \leq \phi(\lfloor \max \phi^{-1} (t) \rfloor_{G})$, which yields the equality in \eqref{eq:floor-phi}.

    Finally, we use \eqref{eq:floor-phi} to check \eqref{eq:structure-map-phi-grid},
    \begin{equation*}
        \max \phi^{-1} (\lfloor t \rfloor_{\phi(G)}) = \max \phi^{-1} (\phi (\lfloor \max \phi^{-1} (t) \rfloor_{G})) \geq \lfloor \max \phi^{-1} (t) \rfloor_{G}.
    \end{equation*}
    Since $\lfloor \max \phi^{-1} (t) \rfloor_{G} \in G$, by definition of the floor function,
    \begin{equation*}
        \lfloor \max \phi^{-1} (\lfloor t \rfloor_{\phi(G)}) \rfloor_{G} \geq \lfloor \max \phi^{-1} (t) \rfloor_{G},
    \end{equation*}
    which yields \eqref{eq:structure-map-phi-grid} and completes the proof.

    We prove the statement about minimality by the contrapositive. Let $G$ be a grid for a module $M$ and let $\phi \in \aut_0(I^n,\leq)$. Assume $\phi(G)$ is not a minimal grid for $\phi.M$. Then we can take a strictly smaller grid $G'$ for $\phi.M$. Since $\phi^{-1} \in \aut_0(I^n,\leq)$, by the previous proof, $\phi^{-1}(G')$ is a grid for $M$ and it contains strictly less points than $G$, so $G$ is not minimal.
\end{proof}

A consequence of this result is that it suffices to know how a map acts on a grid of a module to determine the isomorphism class of its action on the module.

\begin{prop}
\label{prop:same-grid-action-iso}
    Let $\phi, \psi \in \overline{\aut_0 (I^n, \leq)}$ and let $M$ be an essentially finite persistence module with grid $G \subseteq I^n$. If $\phi|_G = \psi|_G$, then $\phi.M \cong \psi.M$.
\end{prop}

\begin{rem}
    We can consider the space of grids contained in $I^n$ with the Hausdorff distance. The maps in $\overline{\aut_0 (I^n, \leq)}$ act on this space simply by taking the image grid. By \autoref{prop:grid-under-action}, the map taking each point in $\modspace{n}$ to its minimal grids is equivariant with respect to these actions. The above result states that the action on grids determines the action on persistence modules.
\end{rem}

\begin{proof} 
    By \autoref{prop:grid-under-action}, $\phi(G) = \psi(G)$ is a grid both for $\phi.M$ and $\psi.M$. Thus it suffices to show there is a natural isomorphism between the restrictions $\eta: (\phi.M)|_{\phi(G)} \Rightarrow (\psi.M)|_{\phi(G)}$.

    The key observation to make here is that the condition $\phi|_G = \psi|_G$ implies that for every $g \in G$, 
    \begin{equation}
    \label{eq:phi-psi-coincide-on-grid}
        \phi^{-1}(\phi(g)) \cap G = \psi^{-1}(\phi(g)) \cap G.
    \end{equation}
    We have $g \in \phi^{-1}(\phi(g))$, so $g \in \psi^{-1}(\phi(g))$. For any $g' \in G$ with
    \begin{equation*}
        g \leq g' \leq \max \phi^{-1}(\phi(g)),
    \end{equation*}
    since $\phi$ is order preserving, $\phi(g') = \phi(g)$, and so $g' \in \phi^{-1}(\phi(g))$, and by \eqref{eq:phi-psi-coincide-on-grid}, $g' \in \psi^{-1}(\phi(g))$. Hence,
    \begin{equation*}
        g' \leq \max \psi^{-1}(\phi(g)).
    \end{equation*}
    Conversely, for any $g' \in G$ with
    \begin{equation*}
        g \leq g' \leq \max \psi^{-1}(\phi(g)),
    \end{equation*}
    applying $\psi$ and recalling $\psi(g) = \phi(g)$, we again find $g' \in \psi^{-1}(\phi(g))$, so by \eqref{eq:phi-psi-coincide-on-grid}, $g' \in \phi^{-1}(\phi(g))$ and 
    \begin{equation*}
        g' \leq \max \phi^{-1}(\phi(g)).
    \end{equation*}
    
    Hence, we find that $\lfloor \max \phi^{-1}(\phi(g)) \rfloor_G = \lfloor \max \psi^{-1}(\phi(g)) \rfloor_G$. This allows us to define $\eta_g$ as the composition of two isomorphisms:
    \begin{equation*}
        \eta_g =
        M_{\lfloor \max \phi^{-1}(\phi(g)) \rfloor_G, \max \psi^{-1}(\phi(g))}
        \circ
        M_{\lfloor \max \phi^{-1}(\phi(g)) \rfloor_G, \max \phi^{-1}(\phi(g))}^{-1}.
    \end{equation*}

    Since each $\eta_g$ is a composition of structure maps of $M$, the commutativity of these maps with the structural maps of $\phi.M$ and $\psi.M$ is directly satisfied, and we find,
    \begin{equation*}
        \phi.M \cong \widehat{(\phi.M)|_{\phi(G)}} \cong \widehat{(\psi.M)|_{\phi(G)}} \cong \psi.M.
    \end{equation*}
\end{proof}

\subsection{Interleaving distance and the action on \texorpdfstring{$\modspace{n}$}{modules}}

To finish this section, we introduce three results that relate the interleaving distance between modules to the action we defined, which will be useful later.

We first note that \autoref{klem:nearby-grids}, as well as both \autoref{prop:grid-under-action} and \autoref{prop:same-grid-action-iso} hold for tuples of essentially finite persistence modules and their grids as defined in \autoref{defn:grid_tuple}.

In what follows, we will use a refinement of the granularity that accounts for the boundary effects:
\begin{defn}
    Let $M$ be an essentially finite persistence module with minimal grid $G \subset I^n$. Define an extended grid $\overline{G}$ by adding the boundary points 0 and 1 to every coordinate set: $\overline{G}_j = G_j \cup \{0,1\}$ for all $j$. The granularity of this grid is also an isomorphism invariant of $M$, which we will denote by $\bar\ell_{M} := \ell_{\overline{G}}$, and it is bounded by $\ell_M$. For a tuple of persistence modules $\bm{M}$ with grid in $I^n$, $\bar\ell_{\bm{M}}$ is defined analogously.
\end{defn}

The first lemma uses a symmetrised version of \autoref{klem:nearby-grids} to show that if two modules are closer than half the (boundary-corrected) granularity of both, then they are actually in the same orbit.

\begin{lem}
\label{lem:closer-than-granularity}
    Let $[\bm{M}], [\bm{N}] \in \modspace{n}^{d+1}$ for some $d\geq 0$. If $d_I(\bm{M},\bm{N}) \leq \epsilon < \min\{\bar\ell_{\bm{M}},\bar\ell_{\bm{N}}\}/2$, then there exists $\psi \in \aut_0(I^n,\leq)$ such that $\psi.\bm{N} \cong \bm{M}$ with $d_\infty(\psi, \id) \leq \epsilon$.
\end{lem}

\begin{rem}
    In the $n=1$ case, the conditions would imply that all bars in both barcodes are longer than $2\epsilon$, so all bars in $M$ are $\epsilon$-matched to a bar in $N$ and vice versa. The result then states that there exists an order isomorphism $\psi:I \to I$, $\epsilon$-close to $\id$, that realises the matching of endpoints.
\end{rem}

\begin{proof}
Let $G,H$ be the minimal grids for $\bm{M} = (M^0, \dots, M^d)$ and $\bm{N} = (N^0, \dots, N^d)$ respectively. Recall that the minimal grid for $\bm{M}$ contains the minimal grid of each $M^i$. Since $d_I(\bm{M},\bm{N}) \leq \epsilon$, then
\begin{equation*}
    d_I(M^i, N^i) \leq \epsilon < \min\{\bar\ell_{\bm{M}},\bar\ell_{\bm{N}}\}/2 \leq
    \min\{\bar\ell_{M^i},\bar\ell_{N^i}\}/2 \leq
    \min\{\ell_{M^i},\ell_{N^i}\}/2
\end{equation*}
for all $i$. By \autoref{klem:nearby-grids}(i), for all $j$ and all $x \in G_j$ there exists $y \in H_j$ with $|x-y| \leq \epsilon$. Since $\epsilon < \bar\ell_{\bm{M}}/2$, this implies $|G_j \cap (0,1)| \leq |H_j \cap (0,1)|$. The symmetric statement yields $|H_j \cap (0,1)| \leq |G_j \cap (0,1)|$, so we find that $|G_j \cap (0,1)| = |H_j \cap (0,1)|$ and $G \cap (0,1)$ and $H \cap (0,1)$ are at distance at most $\epsilon$ in the Hausdorff distance. Moreover, $0 \in G_j$ if and only if $0 \in H_j$, and the same holds for $1$.

Thus, we can build $\psi_j: I \to I$ as follows. We set $\psi_j(0) = 0$, $\psi_j(1) = 1$, and we send the elements of $H_j \cap (0,1)$ to the elements of $G_j \cap (0,1)$ in order, and extend linearly to $I$. The product map $\psi = \psi_1 \times \dots \times \psi_n$ belongs to $\aut_0(I^n,\leq)$ and satisfies $d_\infty(\psi, \id) \leq \epsilon$.

Since $\psi(H) = G$ is a grid for $\psi.N^i$ for all $i$, it suffices to prove that $\psi.N^i|_G \cong M^i_G$. Let $\gamma^i: M^i \Rightarrow (N^i)[\epsilon]$ denote the first part of the interleaving. We use the following three observations to build a natural isomorphism $\lambda^i: M^i \Rightarrow \psi.N^i$. Let $g \in G$, $h \in H$ with $\psi(h) = g$.
\begin{enumerate}[label=(\arabic*)]
\item Since $\epsilon < \ell_{\bm{N}}/2$, $\lfloor h\rfloor_H = \lfloor h + 2\vec\epsilon\rfloor_H$, so the structure map $N^i_{h, h + 2\vec\epsilon}$ is an isomorphism. 
\item The above point also ensures that $h + \vec\epsilon$ satisfies the condition in \autoref{klem:nearby-grids}(ii), so $\gamma^i_{h + \vec\epsilon}: M^i_{h + \vec\epsilon} \to N^i_{h + 2\vec\epsilon}$ is an isomorphism.
\item Finally, by construction $\|g-h\|_\infty \leq \epsilon$, which implies $g \leq h + \vec\epsilon \leq g + 2\vec\epsilon$, and since $\epsilon< \ell_{\bm{M}}/2$, $\lfloor h + \vec\epsilon \rfloor_G = g$ and the structure map $M^i_{g, h + \vec\epsilon}$ is an isomorphism.
\end{enumerate}

Hence, we can define $\lambda^i_g: M^i_g \to (\psi.N^i)_g = N^i_h$ as the composition
\begin{equation*}
    \lambda_g^i = (N^i_{h,h+2\epsilon})^{-1} \circ \gamma^i_{h + \vec\epsilon} \circ M^i_{g, h + \vec\epsilon},
\end{equation*}
which is an isomorphism. Checking that the commutativity relations are satisfied is direct since the map is a composition of structure maps and natural transformations. Hence, $\psi.N^i \cong M^i$ and $\psi.\bm{N} \cong \bm{M}$.
\end{proof}

The following lemma deals with a slightly different situation, in which we have no control over the granularity of the second module $N$, but we have some control over the number of points in each coordinate set of its minimal grid. In this case, the fixed module $M$ is in the closure of the orbit of the second module $N$. Intuitively, $M$ can be reached from $N$ if we allow some collapses of its minimal grid.

\begin{lem}
\label{lem:squeeze-grid}
    Let $[\bm{M}] \in \modspace{n}^{d+1}$ for some $d\geq 0$, and let $[\bm{N}] \in \modspace{n}^{d+1}$. If
    \begin{equation*}
        d_I(\bm{M},\bm{N}) \leq \epsilon < \frac{\bar\ell_{\bm{M}}}{2\max_j|H_j|},
    \end{equation*}
    then there exists a map $\psi \in \overline{\aut_0(I^n, \leq)}$ such that
    \begin{equation*}
        \psi.\bm{N} \cong \bm{M},
    \end{equation*}
    with $d_\infty(\psi, \id) \leq (2\max_j|H_j|-1)\epsilon$.
\end{lem}

\begin{rem}
    In the $n=1$ case, the conditions imply that an $\epsilon$-matching matches all bars in $M$ to bars in $N$, and collapses the rest of bars in $N$, which are at most $2\epsilon$ long. The result states that if we have some control over the number of small bars in $N$ relative to the granularity of $M$, we can find $\psi \in \epi(I, \leq)$ that realises the matchings of endpoints and the collapses.
\end{rem}

\begin{proof}
Let $G$ be the minimal grid of $\bm{M}$, and let $H\neq \varnothing$ be a grid for $\bm{N}$. 
Let $j = 1,\dots,n$. We start by partitioning $H_j$ into the equivalence classes of the following relation. For all $y, y' \in H_j$, we deem them equivalent if there exists a sequence $y_0 = y,y_1, \dots, y_k = y' \in H_j$ such that $|y_{i+1}-y_i| \leq 2\epsilon$ for all $i = 0, \dots, k-1$. We denote the equivalence class of $y \in H_j$ as $\mathcal{P}_j^{y}$.

By the choice of $\epsilon$, each equivalence class is $\epsilon$-close to at most one point of $\overline{G}_j$. We define $\psi_j: I \to I$ by fixing its value on the boundary points and on $H_j$, and then extending linearly on $I$. We fix $\psi_j(0) = 0$ and $\psi_j(1) = 1$. For each equivalence class $\mathcal{P}_j^{y}$ in $H_j$, if it is $\epsilon$-close to a point of $\overline{G}_j$, we define $\psi_j$ to be that point on all elements of $\mathcal{P}_j^{y}$. Otherwise, $\psi_j$ on that class is defined to be an arbitrary point of $\mathcal{P}_j^{y}$. We get $\psi_j \in \epi(I, \leq)$.

Since the span of an equivalence class is at most $2\epsilon(|H_j|-1)$, the map $\psi_j$ differs from the identity at most $(2|H_j|-1)\epsilon$. So $\psi = \psi_1 \times \dots \times \psi_n$ belongs to $\overline{\aut_0(I^n, \leq)}$ and it satisfies the requirement $d_\infty(\psi, \id) \leq (2\max_j|H_j|-1)\epsilon$.

We will prove that $\psi.N^i \cong M^i$ for all $i$. Note that our choice of $\epsilon$ in particular ensures 
\begin{equation*}
    d_I(M^i,N^i) \leq \epsilon < \ell_{\overline{G}}/2 \leq \ell_{M^i}/2,
\end{equation*}
so the hypotheses of \autoref{klem:nearby-grids} are satisfied for $M^i$, $N^i$. Hence, by (i), for all $j$ and all $x \in G_j$ belonging to the minimal grid of $M^i$, there exists some $y \in H_j$ belonging to the minimal grid of $N^i$ with $|x-y| \leq \epsilon$. The equivalence class of $y$ is thus sent by $\psi_j$ to $x$. Therefore, $G \subseteq \psi(H)$ (since the boundary points are also clearly in the image), which implies $\psi(H)$ is a grid for $M^i$ for all $i$. We know by \autoref{prop:grid-under-action} that $\psi(H)$ is a grid for $\psi.N^i$. Hence, it suffices to prove $\psi.N^i |_{\psi(H)} \cong M^i|_{\psi(H)}$.

We will build a natural isomorphism $\lambda^i: M^i|_{\psi(H)} \Rightarrow \psi.N^i |_{\psi(H)}$. Let $\psi(h) \in \psi(H)$ with $h = (h_1, \dots, h_n)$. To define $\lambda^i_{\psi(h)}$, we first check that certain maps are isomorphisms, which fit in the diagram below.

\begin{enumerate}[label=(\arabic*)]
    \item By construction of $\psi$,
    \begin{equation*}
        \max\psi^{-1}(\psi(h)) = \max\mathcal{P}_1^{h_1} \times \dots \times \max\mathcal{P}_n^{h_n}.
    \end{equation*}
    By definition, two distinct classes in $H_j$ are separated by an interval of length strictly greater than $2\epsilon$. For all $j$, then,
    \begin{equation*}
        (\max \mathcal{P}_j^{h_j}, \max \mathcal{P}_j^{h_j} + 2\epsilon] \cap H_j = \varnothing,
    \end{equation*}
    so $\lfloor \max \psi^{-1} (\psi(h)) \rfloor_H = \lfloor \max \psi^{-1} (\psi(h)) + 2\vec\epsilon \rfloor_H$, and thus the structure map
    \begin{equation*}
        N^i_{\max \psi^{-1} (\psi(h)), \max \psi^{-1} (\psi(h)) + 2\vec\epsilon}
    \end{equation*}
    is an isomorphism.

    \item Moreover, the previous point shows that $\max \psi^{-1} (\psi(h)) + \vec\epsilon$ satisfies the conditions of \autoref{klem:nearby-grids}(ii), so denoting $\gamma^i: M \Rightarrow N[\epsilon]$ the first part of the $\epsilon$-interleaving between $M^i$ and $N^i$, then
    \begin{equation*}
        \gamma^i_{\max \psi^{-1} (\psi(h)) + \vec\epsilon}: M^i_{\max \psi^{-1} (\psi(h)) + \vec\epsilon} \to N^i_{\max \psi^{-1} (\psi(h)) + 2\vec\epsilon}
    \end{equation*}
    is an isomorphism.

    \item By definition, for every coordinate $\psi(h)_j = \psi_j(h_j)$, there exists $y \in \mathcal{P}_j^{h_j}$ with $|\psi(h)_j-y| \leq \epsilon$. Then 
    \begin{equation*}
        \min \mathcal{P}_j^{h_j} - \epsilon \leq y-\epsilon\leq \psi(h)_j \leq y + \epsilon \leq \max \mathcal{P}_j^{h_j} + \epsilon.
    \end{equation*}
    The set $\mathcal{P}_j^{h_j}$ is either $\epsilon$-close to a single point in $\overline{G}_j$, in which case this point is $\psi_j(h_j)$, or it is not. In both cases,
    \begin{equation*}
        (\psi(h)_j, \max \mathcal{P}_j^{h_j} + \epsilon] \cap \overline{G}_j = \varnothing,
    \end{equation*}
    so $\lfloor \max \psi^{-1} (\psi(h)) + \vec\epsilon \rfloor_{\overline{G}} = \lfloor \psi(h) \rfloor_{\overline{G}}$ and the map 
    \begin{equation*}
        M^i_{\psi(h),\max \psi^{-1} (\psi(h)) + \vec\epsilon}
    \end{equation*}
    is an isomorphism.
\end{enumerate}

We define 
\begin{equation*}
    \lambda^i_{\psi(h)}: M^i_{\psi(h)} \to (\psi.N^i)_{\psi(h)},
\end{equation*}
as the composition
\begin{equation*}
    \lambda^i_{\psi(h)} = (N^i_{\max \psi^{-1} (\psi(h)), \max \psi^{-1} (\psi(h)) + 2\vec\epsilon})^{-1} \circ \gamma^i_{\max \psi^{-1} (\psi(h)) + \vec\epsilon} \circ M^i_{\psi(h),\max \psi^{-1} (\psi(h)) + \vec\epsilon}
\end{equation*}

which makes the following diagram commutative:
$$
\begin{tikzcd}
	{(\psi.N^i)_{\psi(h)} = N^i_{\max \psi^{-1}(\psi(h))}} && 
    N^i_{\max \psi^{-1}(\psi(h))+2\vec\epsilon} \\
	{M^i_{\psi(h)}} &
    {M^i_{\max \psi^{-1}(\psi(h))+\vec\epsilon}}
    \arrow["\lambda^i_{\psi(h)}"', from=2-1, to=1-1]
	\arrow["\cong^{(1)}", from=1-1, to=1-3]
	\arrow["\cong^{(3)}", from=2-1, to=2-2]
	\arrow["^{(2)}{\gamma^i_{\max\psi^{-1}(\psi(h))+\vec\epsilon}}"', from=2-2, to=1-3]
\end{tikzcd}
$$

Since $\lambda^i$ is built as a composition of structure maps and interleaving maps, the commutativity with the structure maps is ensured, and we find $\psi.N^i |_{\psi(H)} \cong M^i|_{\psi(H)}$ which implies $\psi.N^i \cong M^i$ and $\psi.\bm{N} \cong \bm{M}$. 
\end{proof}

The last result in this section considers a more general scenario than the previous two, where we have no control over the granularity or the number of points in the grid of one of the modules. Although this proposition is not used in the rest of the paper, we include it here for completeness. This result is more naturally stated for essentially finite persistence modules with grids not necessarily contained in $I^n$.

First, note that the action on persistence modules readily extends to an action of ${\epi(\mathbb{R}, \leq)^n}$, where ${\epi(\mathbb{R}, \leq)}$ is the space of surjective order-preserving maps of $\mathbb{R}$ with the sup-norm. 
In this context, it still holds that the image of the grid will be a grid for the module under the action (first part of \autoref{prop:grid-under-action}), and that two maps whose restrictions to the grid are the same will yield isomorphic modules under the action (\autoref{prop:same-grid-action-iso}).

\begin{prop}
\label{prop:NcontainsM}
    Let $M, N \in \vect^{\mathbb{R}^n}$ be essentially finite modules such that $d_I(M,N) \leq \epsilon < \ell_M/2$. Let $\gamma: M \Rightarrow N[\epsilon], \kappa: N \Rightarrow M[\epsilon]$ be an $\epsilon$-interleaving. Then
    \begin{enumerate}[label=(\roman*)]
        \item there exists $\psi \in \epi(\mathbb{R}, \leq)^n$ with $d_\infty(\psi, \id) \leq \epsilon$ such that
        \begin{equation*}
            \psi.(\im\gamma)[-\epsilon] \cong M,
        \end{equation*}
        \item $d_I(N/(\im\gamma)[-\epsilon], 0) \leq \epsilon $.
    \end{enumerate}
\end{prop}

\begin{rem}
    Intuitively, part (i) states that $N$ contains an approximate copy of $M$, that is, a submodule that, after perturbing it slightly by the action, becomes isomorphic to $M$. Part (ii) then states that this submodule is basically all of $N$: if we quotient $N$ by this approximate copy of $M$, the result is close to 0.
\end{rem}

\begin{proof}
To begin with, we obtain an $\epsilon$-interleaving between $M$ and $(\im\gamma)[-\epsilon]$, seen as a submodule of $N$. The natural transformation $\gamma$ factors through the image, via $\tilde{\gamma}: M \Rightarrow \im \gamma$, and the transformation $\kappa$ can be restricted to the image submodule: $\tilde{\kappa} = \kappa|_{(\im\gamma)[-\epsilon]}: (\im\gamma)[-\epsilon] \Rightarrow M[\epsilon]$.

Let $G$ be the minimal grid of $M$. For each $j = 1, \dots, n$, let us build $\psi_j \in \epi(\mathbb{R}, \leq)$ as follows. The map $\psi_j$ fixes every $t \geq \max G_j + 2\epsilon$ and $t \leq \min G_j - 2\epsilon$. For each $x \in G_j$, $\psi_j$ collapses the interval $[x-\epsilon, x + \epsilon]$ to $x$. On the remaining intervals, we define $\psi_j$ by linear interpolation.

Then, for any $y \in G_j$, $\psi_j^{-1} (y) = [y-\epsilon, y + \epsilon]$. For any $y \notin G_j$, $y$ is contained in an interval $[x,x']$, with $x, x'$ two distinct, consecutive points in $G_j\cup \{\pm \infty\}$. Then $\psi_j^{-1}(y)$ is a single point contained in $(x + \epsilon, x'-\epsilon)$.

Let $\psi = \psi_1 \times\dots\times \psi_n$. By construction, the following equalities hold for any $r\in \mathbb{R}^n$,
\begin{equation}
\label{eq:prop4_25_relations}
    \lfloor r \rfloor_G
    = \lfloor \max \psi^{-1}(r) - \vec\epsilon \rfloor_G
    = \lfloor \max \psi^{-1}(r) + \vec\epsilon \rfloor_G.
\end{equation}

We will first show that $G$ is a grid for $\psi.(\im\gamma)[-\epsilon]$, using the characterisation from \autoref{lem:ess-fin-characterisation}. Let $g \in G$ and $r\in \mathbb{R}$, with $r\geq \min G$ and $\lfloor r \rfloor_G = g$. Note that $\max \psi^{-1}(g) = g + \vec\epsilon$. We will show that the structure map 
\begin{equation*}
    (\psi.(\im\gamma)[-\epsilon])_{g,r}: \im\gamma_{g} \to \im\gamma_{\max \psi^{-1}(r) -\vec\epsilon}
\end{equation*}
is an isomorphism. Since this map is the restriction of the structure map of $N$ to $\im \gamma [-\epsilon]$, we can fit it into a diagram using the restricted $\epsilon$-interleaving.

$$
\begin{tikzcd} [column sep=0.1cm, cells={nodes={text width=2cm, align=center}}]
	&& {M_{g + 2\vec\epsilon}}
    &&& {M_{\max \psi^{-1}(r) + \vec\epsilon}} \\
	& {\im\gamma_g}
    &&& {\im\gamma_{\max \psi^{-1}(r)- \vec\epsilon}} \\
	{M_{g}}
    &&& {M_{\max \psi^{-1}(r) - \vec\epsilon}}
	\arrow["\cong"{pos = 0.5}, from=1-3, to=1-6]
	\arrow["\tilde{\kappa}_{g + \vec\epsilon}"'{pos=0.5}, from=2-2, to=1-3]
	\arrow["(\psi.(\im\gamma){[-\epsilon]})_{g,r}", from=2-2, to=2-5]
	\arrow["\tilde{\kappa}_{\max \psi^{-1}(r)}"{pos=0.5}, from=2-5, to=1-6]
	\arrow["\cong"{pos = 0.5}, curve={height=-30pt}, from=3-1, to=1-3]
	\arrow["\tilde{\gamma}_g"'{pos=0.5}, hook, two heads, from=3-1, to=2-2]
	\arrow["\cong"{pos = 0.5}, from=3-1, to=3-4]
	\arrow["\cong"'{pos = 0.5},curve={height=50pt}, from=3-4, to=1-6]
	\arrow["\tilde{\gamma}_{\max \psi^{-1}(r) - \vec\epsilon}"{pos=0.5}, hook, two heads, from=3-4, to=2-5]
\end{tikzcd}
$$

The lower and upper maps, as well as both composition of the diagonal maps are all isomorphisms, since $\lfloor r \rfloor_G = g$ and the relations \eqref{eq:prop4_25_relations}. As a consequence, both $\tilde{\gamma}_g$ and $\tilde{\gamma}_{\max \psi^{-1}(r) -\vec\epsilon}$ are injective, and they are both surjective since they are restrictions onto the image, hence they are isomorphisms. By commutativity of the diagram, $(\psi.(\im\gamma)[-\epsilon])_{g,r}$ is also an isomorphism.

Let $r\not\geq \min G$. For some $j$, $r_j < \min G_j$. Then, by construction $\max \psi_j^{-1}(r_j) < \min G_j$, and thus $\max \psi^{-1} (r) - \vec\epsilon \not\geq \min G $. Then $(\psi.(\im\gamma)[-\epsilon])_r = \gamma_{\max \psi^{-1} (r) -\vec\epsilon }(M_{\max \psi^{-1} (r) -\vec\epsilon}) = 0$.

Thus, $G$ is a grid for $\psi.(\im\gamma)[-\epsilon]$, and hence proving (i) reduces to showing that the restriction to $G$ is isomorphic to $M|_G$. But for each $g\in G$, $\max\psi^{-1}(g) - \vec\epsilon = g$, so $(\psi.(\im\gamma)[-\epsilon])|_G = \im\gamma|_G$. We only need to consider the restriction $\tilde{\gamma}|_{G}: M|_G \Rightarrow \im\gamma|_G$, which is an isomorphism as we observed from the diagram above. This finishes the proof of (i).

Part (ii) is satisfied more generally by any interleaving, regardless if the distance is bounded by half the granularity. Checking that a module is at a distance $\epsilon$ from the 0 module is equivalent to checking that for any $r \in \mathbb{R}^n$, the structure map between grades $r$ and $r + 2\vec\epsilon$ is 0. In our case, this follows directly from the commutativity condition of the interleaving. For any $r \in \mathbb{R}^n$, $N_{r, r + 2\vec\epsilon} = \gamma_{r + \vec\epsilon} \circ \kappa_r$, hence $\im N_{r, r + 2\vec\epsilon} \subseteq \im \gamma_{r + \vec\epsilon} = (\im \gamma)[-\epsilon]_{r + 2\vec\epsilon}$, so taking the quotient by $(\im \gamma)[-\epsilon]$, the structure map becomes 0, which proves $d_I(N/(\im \gamma)[-\epsilon], 0) \leq \epsilon$.

\end{proof}

We can find a similar result for persistence modules with grids contained in $I^n$ and the action of $\overline{\aut_0(I^n,\leq)}$. In this case, however, we cannot ensure the grid of $(\im\gamma)[-\epsilon]$ lies in $I^n$, so the statement is slightly modified.

\begin{prop}
\label{prop:NcontainsMrestricted}
    Let $[M], [N] \in \modspace{n}$ with $d_I(M,N) \leq \epsilon < \ell_M/2$, and let $\gamma: M \Rightarrow N[\epsilon], \kappa: N \Rightarrow M[\epsilon]$ be an $\epsilon$-interleaving. Then there exists $\psi \in \overline{\aut_0(I^n,\leq)}$ with $d_\infty(\psi, \id) \leq 2\epsilon$ such that
    \begin{equation*}
        \psi.\im\gamma \cong M.
    \end{equation*}
\end{prop}

\begin{proof}
    Let $G$ be the minimal grid of $M$.
    For each $j$, let us build $\psi_j \in \overline{\aut_0(I,\leq)}$ as follows. Send $\min G_j$ to itself, and for any $x\in G_j$, $x\neq \min G_j$, collapse the interval $[x-2\epsilon, x]$ to $x$, and interpolate linearly. Let $\psi = \psi_1 \times\dots\times \psi_n$. Then by construction, the following equalities hold for any $r\in \mathbb{R}^n$,
\begin{equation}
    \lfloor r \rfloor_G
    = \lfloor \max \psi^{-1}(r) \rfloor_G
    = \lfloor \max \psi^{-1}(r) + 2\vec\epsilon \rfloor_G.
\end{equation}
The same strategy in the proof of \autoref{prop:NcontainsM} works.
\end{proof}

\newpage

\section{Stratifications on the space of \texorpdfstring{$n$}{n}-filters and \texorpdfstring{$n$}{n}-parameter persistence modules}
\label{sec:strat}

In this section we use the actions introduced in \autoref{sec:actions} to endow the space of $n$-filters and the image of the multiparameter persistent homology map with natural stratifications with strata given by ${\aut_0(I^n,\leq)}$-orbits. Furthermore, using its equivariance, we show that $MPH$ is a strongly stratified map. This enables us to identify the image $\modspace{n}^K$ as the quotient of $\filt^n$ by $MPH$.

We start by introducing a central definition for this section.

\begin{defn}[stratification]
    A stratification of a topological space $X$ is a (possibly infinite) nested sequence $\varnothing = X_{-1} \subseteq X_0 \subseteq X_1 \subseteq \cdots$ of closed subspaces $X_i \subseteq X$, $i \in \mathbb{N}$, such that $\cup_i X_i = X$ and the sets $X_i \setminus X_{i-1}$ are topological manifolds of dimension $i$. The path-connected components of $X_i \setminus X_{i-1}$ are called $i$-strata, or strata of dimension $i$. A stratified map between two stratified spaces $X,Y$ is a continuous map $f:X \to Y$ that preserves the stratification, that is, such that $f(X_i) \subset Y_i$ for all $i$. A stratified map is strongly stratified if it maps any stratum of $X$ surjectively to a stratum of $Y$.
\end{defn}

In the stratifications we will consider, strata will consist of products of open simplices of possibly different dimensions. For $i\geq 1$, we identify an $i$-dimensional open simplex with a subset of $\mathbb{R}^i$ as follows:
\begin{equation*}
    \mathring{\Delta}^{i} = \{(x_1, \dots, x_{i}) \mid 0 < x_1 < \cdots < x_{i} < 1\} \subset \mathbb{R}^{i},
\end{equation*}
and we denote its closure, the closed simplex, as $\Delta^{i}$. For $i=0$, $\mathring{\Delta}^{0} = \Delta^0 = \{*\}$. We consider products of simplices with the $\|\cdot\|_\infty$-metric.

\subsection{Stratification of the space of \texorpdfstring{$n$}{n}-filters}

The goal of this subsection is to build a stratification of the space of $n$-filters from the orbits of the action (\autoref{prop:filt-strat}), which can be characterised in the following way.

\begin{lem}[Characterisation of $\aut_0 (I^n, \leq)$-orbits]
\label{lem:char-orbits-aut-filt}
Two $n$-filters $f,g$ belong to the same $\aut_0 (I^n, \leq)$-orbit if and only if for all $j = 1, \dots, n$
\begin{enumerate}[label = (\roman*)]
    \item $f_j$ and $g_j$ induce the same preorder on the simplices, that is, for all $\tau, \sigma \in K$, 
    \begin{equation*}
        f_j(\tau) \leq f_j(\sigma) \iff g_j(\tau) \leq g_j(\sigma),
    \end{equation*}
    \item $f_j^{-1} (0) = g_j^{-1}(0)$ and $f_j^{-1} (1) = g_j^{-1}(1)$.
\end{enumerate}
\end{lem}

\begin{proof}
    Let $g = \phi.f$ for some $\phi \in \aut_0 (I^n, \leq) \cong \aut(I, \leq)^n$, then clearly (i) and (ii) are satisfied for all $j$. Assume that $f$ and $g$ satisfy (i) and (ii). Let us build maps $\phi_j: I \to I$ by fixing $\phi_j(0) = 0$, $\phi_j(1) = 1$, $\phi_j(f_j(\tau)) = g_j(\tau)$ for every $\tau\in K$, and interpolating linearly. Clearly $\phi_j \in \aut(I, \leq)$ for all $j$, and $g = \phi.f$, so $g \in \aut_0 (I^n, \leq).f$.
\end{proof}

\begin{rem}
    The $\aut_0 (I^n, \leq)$-orbits we defined can be seen as a compact version of the cells defined in \cite{Scoccola2024}.
\end{rem}

Following the same strategy, we can prove a similar characterisation of the orbit of the closure:

\begin{lem}[Characterisation of $\overline{\aut_0 (I^n, \leq)}$-orbits]
\label{lem:char-orbits-epi-filt}
    Let $f, g$ be $n$-filters. Then, $g \in \overline{\aut_0 (I^n, \leq)}.f$ if and only if for all $j = 1, \dots, n$
    \begin{enumerate}[label = (\roman*)]
    \item for all $\tau, \sigma \in K$, $f_j(\tau) \leq f_j(\sigma) \implies g_j(\tau) \leq g_j(\sigma)$,
    \item $f_j^{-1} (0) \subseteq g_j^{-1}(0)$ and $f_j^{-1} (1) \subseteq g_j^{-1}(1)$.
\end{enumerate}
\end{lem}

From this characterisation, we can prove that the closure of the orbit coincides with the orbit of the closure.

\begin{prop}
\label{prop:closure-orbit-filt}
    For every $f\in \filt^n$, we have
    \begin{equation*}
        \overline{\aut_0(I^n, \leq).f} = \overline{\aut_0(I^n, \leq)}.f.
    \end{equation*}
\end{prop}

\begin{proof}
We first prove the claim for the case $n=1$. Let $f: K \to I$ be a filter function and consider $g \in \overline{\aut(I, \leq).f}$. Then there exists a sequence $\phi^k \in \aut(I, \leq)$ such that $\phi^k.f$ converges to $g$ in the sup-norm. Note that this does not imply the convergence of the sequence $\phi^k$.
We know, however, that for every $\tau$, $\phi^k.f(\tau) \to g(\tau)$. Since each $\phi^k$ is order preserving,
\begin{equation*}
    f(\tau) \leq f(\sigma) \implies \phi^k.f(\tau) \leq \phi^k.f(\sigma).
\end{equation*}
By taking limits at both sides we find $g(\tau) \leq g(\sigma)$, so $g$ satisfies condition (i) of \autoref{lem:char-orbits-epi-filt}. If $f(\tau) = 0$ for some $\tau \in K$, then $\phi^k.f(\tau) = 0$ so in the limit $g(\tau) = 0$, and analogously if $f(\tau) = 1$, so condition (ii) of \autoref{lem:char-orbits-epi-filt} is also satisfied and $g \in \overline{\aut(I, \leq)}.f$.

Conversely, if $g\in \overline{\aut(I, \leq)}.f$, then there exists a sequence $\phi^k \in \aut(I, \leq)$ converging to $\phi \in \overline{\aut(I, \leq)}$, with $g = \phi.f$. Then the sequence $\phi^k.f \in \aut(I, \leq).f$ converges in sup-norm to $g$, so $g \in \overline{\aut(I, \leq).f}$.

We now consider an $n$-filter $f$ for $n>1$. Recall that we see $\filt^n$ embedded in $I^{n|K|}$, and
\begin{equation*}
    \aut_0(I^n, \leq).f \cong \aut(I, \leq).f_1 \times \cdots \times \aut(I, \leq).f_n
    \subseteq
    I^{|K|} \times \cdots \times I^{|K|},
\end{equation*}
so by the properties of product topology, the case $n=1$ and \autoref{lem:closure-aut},
\begin{equation}
\label{eq:closure-product}
    \overline{\aut_0(I^n, \leq).f}
    = \prod_{j=1}^n \overline{\aut(I, \leq).f_j}
    = \prod_{j=1}^n \overline{\aut(I, \leq)}.f_j
    = \epi(I,\leq)^n.f
    = \overline{\aut_0(I^n, \leq)}.f.
\end{equation}
\end{proof}

From \autoref{lem:char-orbits-aut-filt}, we see that inside an orbit, a filter is uniquely determined by the values of each $f_j$ that are not 0 or 1, listed in increasing order. From this identification we can build a homeomorphism between the orbit of $f$ and a product of open simplices, which will extend to the boundary. We denote by $\im f_j$ the set of distinct values taken by $f_j$, listed in increasing order.
\begin{defn}
    Given an $n$-filter $f$, we introduce the following notation,
    \begin{equation*}
        \kappa_j(f) := | \im f_j \cap (0,1) |,
    \end{equation*}
    and we define the orbit-dimension of $f$ as
    \begin{equation*}
        \odim(f) := \sum_{j=1}^n \kappa_j(f).
    \end{equation*}
\end{defn}

\begin{prop}
\label{prop:homeo-orbits-filt}
    The $\aut_0 (I^n, \leq)$-orbit of an $n$-filter $f$ is a topological manifold of dimension $\odim(f)$. Specifically, there is a homeomorphism $\aut_0(I^n, \leq).f \cong \mathring\Delta^{\kappa_1(f)} \times \dots \times \mathring\Delta^{\kappa_n(f)}$, which moreover extends to a homeomorphism of the closures: $\overline{\aut_0(I^n, \leq).f} \cong \Delta^{\kappa_1(f)} \times \dots \times \Delta^{\kappa_n(f)}$.
\end{prop}

\begin{proof}
    Assume that $\kappa_j(f) \neq 0$ for all $j = 1,\dots,n$.
    We define the following affine morphism (i.e. can be extended to an affine map between the ambient spaces):
    \begin{equation*}
    \begin{array}{rcl}
    \zeta: \aut_0 (I^n, \leq).f &\rightarrow &
    \mathring\Delta^{\kappa_1(f)} \times \dots \times \mathring\Delta^{\kappa_n(f)} \\
    g & \mapsto & (\im g_1 \cap (0,1), \dots, \im g_n \cap (0,1)),
    \end{array}
    \end{equation*}

    Let $\mathcal{O}_{j}: \{\sigma \in K \mid f_j(\sigma) \neq 0,1\} \to \{ 1, \dots, \kappa_j(f)\}$ be the surjective function corresponding to the pre-order induced by $f_j$ on simplices that take values different from 0 and 1. Then, we can define the inverse of $\zeta$,
    \begin{equation}
    \label{eq:filter-orbit-coord-chart}
        \mu: \mathring\Delta^{\kappa_1(f)} \times \dots \times \mathring\Delta^{\kappa_n(f)} \rightarrow \aut_0 (I^n, \leq).f,
    \end{equation}
    
    as follows. Consider $x = (x^1, \dots, x^n) $, with each $x^j \in \mathring\Delta^{\kappa_j(f)}$, then for $\sigma \in K$, 
    \begin{equation}
    \label{eq:homeo-simplex-orbit}
        \mu(x)_j(\sigma) =
        \begin{cases}
            0 & \text{if } f_j(\sigma) = 0, \\
            1 & \text{if } f_j(\sigma) = 1, \\
            x^j_{\mathcal{O}_j(\sigma)} & \text{otherwise.}
        \end{cases}
    \end{equation}

    If $\kappa_j(f) = 0$ for some $j$, the image of $f_j$ is contained in the boundary, so for every $g \in \aut_0(I^n,\leq).f$, $g_j = f_j$. We modify the components of the morphisms $\zeta$ and $\mu$ for which $\kappa_j(f) = 0$. For $\zeta$, instead of recording the values $\im g_j \cap (0,1)$, we simply map the filter to the unique element of $ \mathring{\Delta}^0$. For $\mu$, if $x^j \in \Delta^0$, then $\mu(x)_j = f_j$.

    Since both maps are continuous, and they are clearly inverses of each other, we find that $\zeta$ is the desired homeomorphism.

    Using that the closure of the orbit is the orbit of the closure (\autoref{prop:closure-orbit-filt}), we can extend the map $\mu$ to the boundary, $\overline\mu: \Delta^{\kappa_1(f)} \times \dots \times \Delta^{\kappa_n(f)} \to \overline{\aut_0(I^n, \leq).f} = \overline{\aut_0(I^n, \leq)}.f$ with the same definition as in \eqref{eq:homeo-simplex-orbit}. Using the characterisation from \autoref{lem:char-orbits-epi-filt}, it is easy to check that this is indeed well-defined, continuous and bijective, and since both are compact metric spaces, it is a homeomorphism.

    Alternatively, note that $\overline{\aut_0(I^n, \leq)}$ acts on $\Delta^{\kappa_1(f)} \times \dots \times \Delta^{\kappa_n(f)}$ block-diagonally, and $\mu$ is equivariant with respect to this action restricted to $\aut_0(I^n, \leq)$. We can extend $\mu$ to an equivariant map from the closed product of simplices, which will coincide with the previous definition. Any point in the closure $ \Delta^{\kappa_1(f)} \times \dots \times \Delta^{\kappa_n(f)}$ can be expressed as $\phi.x$ for some map $\phi \in \overline{\aut_0(I^n, \leq)}$ and interior point $x\in \mathring\Delta^{\kappa_1(f)} \times \dots \times \mathring\Delta^{\kappa_n(f)}$. Then, let $\overline{\mu}(\phi.x) = \phi.\overline{\mu}(x)$. This is easily shown to be a well-defined affine homeomorphism, and it coincides with the previous definition.
\end{proof}

Finally, we analyse how the closures of orbits interact, and use this to show that orbits form a stratification of $\filt^n$.

\begin{prop} 
\label{prop:facts-odim-closure}
    Let $f, g$ be $n$-filters. Then 
    \begin{enumerate}[label=(\roman*)]
        \item if $g \in \overline{\aut_0(I^n, \leq).f}$, then $\odim(g) \leq \odim(f)$,
        \item if $g \in \overline{\aut_0(I^n, \leq).f} \setminus \aut_0(I^n, \leq).f$, then $\odim(g) < \odim(f)$.
    \end{enumerate}
\end{prop}

\begin{proof}
By \autoref{prop:closure-orbit-filt} and \autoref{lem:closure-aut}
$g = \phi.f$ with $\phi \in \epi(I, \leq)^n$. 
\begin{enumerate}[label = (\roman*)]
    \item Consider $\psi \in \aut_0(I^n, \leq) \cong \aut(I, \leq)^n$, then $\psi \circ \phi \in \epi(I,\leq)^n$, so $\psi.g = (\psi \circ \phi).f \in \overline{\aut_0(I^n, \leq).f}$, hence $\aut_0(I^n, \leq).g \subseteq \overline{\aut_0(I^n, \leq).f}$.
    By this inclusion, and the homeomorphisms in \autoref{prop:homeo-orbits-filt}, $\odim(g) \leq \odim(f)$.
    \item If $g \in \overline{\aut_0(I^n, \leq).f} \setminus \aut_0(I^n, \leq).f$, then $\aut_0(I^n, \leq).g \cap \aut_0(I^n, \leq).f = \varnothing$ so by the inclusion in (i), the orbit of $g$ is part of the boundary of $\overline{\aut_0(I^n, \leq).f}$. The boundary of a product of closed simplices is a finite union of proper faces, each of dimension strictly smaller than the product, therefore $\odim(g) < \odim(f)$.
\end{enumerate}
\end{proof}

\begin{prop}
\label{prop:filt-strat}
    For each $i \geq 0$, let $F_i$ be the union of the $\aut_0(I^n, \leq)$-orbits of dimension at most $i$. This defines a stratification of $\filt^n$. The $i$-strata are given by the orbits of dimension $i$.
\end{prop}

\begin{proof}

Since $K$ is finite, by \autoref{lem:char-orbits-aut-filt} there are only finitely many distinct $\aut_0(I^n, \leq)$-orbits. The maximum dimension of the orbits is $n|K|$, which is reached when all $f_j$ are injective and $\im f_j \subset (0,1)$.
Hence, we have a finite sequence of subspaces
\begin{equation*}
    \varnothing = F_{-1} \subset F_0 \subset F_1 \subset \cdots \subset F_{n|K|} = \filt^n.
\end{equation*}

We first prove that each $F_i$ is closed. Let $g \in \overline{F_i}$. Then $g \in \overline{\aut_0(I^n, \leq).f}$ for some $n$-filter $f$ with $\odim(f) \leq i$. Then by \autoref{prop:facts-odim-closure}(i), $\odim(g) \leq \odim(f) \leq i$ and $g \in F_i$.

Note that $\filt^n$ is a subspace of $I^{n|K|} \subset \mathbb{R}^{n|K|}$, hence each $F_i \setminus F_{i-1}$ is Hausdorff and second-countable. The space $F_i \setminus F_{i-1}$ is formed by finitely many disjoint $i$-dimensional orbits: $\mathcal{F}_1, \dots, \mathcal{F}_{N_i}$, each homeomorphic to a product of open simplices whose dimensions sum up to $i$. \autoref{prop:facts-odim-closure}(ii) implies that for $i \neq j$, $\mathcal{F}_i \cap \overline{\mathcal{F}_j} = \varnothing$. Since there are finitely many orbits, for each $g\in \mathcal{F}_j$ we can find an open ball that only intersects $\mathcal{F}_j$. By \autoref{prop:homeo-orbits-filt}, this ball will be homeomorphic to an open ball in $\mathbb{R}^i$, so $F_i \setminus F_{i-1}$ is locally Euclidean, which proves it is an $i$-dimensional manifold.

Finally, we note that each orbit $\mathcal{F}_j$ is path-connected, and \autoref{prop:facts-odim-closure}(ii) ensures that no path that starts in the orbit $\mathcal{F}_j$ can reach $\mathcal{F}_k$ for $k \neq j$, so the $i$-dimensional orbits form the path-connected components of $F_i \setminus F_{i-1}$, or $i$-strata.
\end{proof}

\subsection{Stratification of the space of \texorpdfstring{$n$}{n}-parameter persistence modules}

We now embark on proving a similar stratification for $\modspace{n}^K \subseteq \modspace{n}^{\dim K + 1}$, the image of the $MPH$ map (\autoref{prop:mod-strat}).

\begin{defn}
    Let $G = G_1 \times \dots \times G_n$ be a minimal grid for $\bm{M}$. We introduce the following notation:
    \begin{equation*}
        \kappa_j(\bm{M}) := |G_j \cap (0,1)|
    \end{equation*}
    and we define the orbit-dimension of $\bm{M}$ as
    \begin{equation*}
        \odim(\bm{M}) := \sum_{j=1}^n \kappa_j(\bm{M}).
    \end{equation*}
    These quantities are invariant under isomorphism and can be assigned to the isomorphism class $[\bm{M}]$.
\end{defn}

\begin{rem}
    Note that by the convention established in \autoref{rem:0modulegrid}, these quantities are also determined for the 0 module $\kappa_j(0) = 0$, $\odim(0) = 0$.
\end{rem}

The following result shows that each orbit is homeomorphic to a product of open simplices.

\begin{prop}
\label{prop:homeo-orbits-modules}
The $\aut_0 (I^n, \leq)$-orbit of $[\bm{M}] \in \modspace{n}^{\dim K +1}$ is a topological manifold of dimension $\odim(\bm{M})$. Specifically, there is a homeomorphism $\aut_0 (I^n, \leq).[\bm{M}] \cong \mathring\Delta^{\kappa_1(\bm{M})} \times \cdots \times \mathring\Delta^{\kappa_n(\bm{M})}$.
\end{prop}

\begin{proof}
    Let $G$ be the minimal grid of $\bm{M}$, and assume for now that $\kappa_j(\bm{M}) \neq 0$ for all $j = 1,\dots,n$. Let us define a map
    \begin{equation*}
        \chi: \aut_0 (I^n, \leq).[\bm{M}] \to \mathring\Delta^{\kappa_1(\bm{M})} \times \cdots \times \mathring\Delta^{\kappa_n(\bm{M})}
    \end{equation*}
    as follows. Consider $[\bm{N}] \in \aut_0 (I^n, \leq).[\bm{M}]$, and let $H$ denote the minimal grid of $\bm{N}$. Then
    \begin{equation*}
        \chi: [\bm{N}] \mapsto (H_1 \cap (0,1), \dots, H_n\cap (0,1)).
    \end{equation*}
    The map does not depend on the choice of representative, and more specifically, if $\bm{N} = \phi.\bm{M}$ for $\phi = \phi_1 \times\dots\times \phi_n \in \aut_0 (I^n, \leq)$, then by \autoref{prop:grid-under-action}, $H = \phi(G)$ and
    \begin{equation*}
        |H_j \cap (0,1)| = |\phi_j(G_j) \cap (0,1)| = |G_j \cap (0,1)| = \kappa_j(\bm{M}),
    \end{equation*}
    so $\chi$ is well defined.
    
    Next, we prove that $\chi$ is a bijection. We build a map
    \begin{equation}
    \label{eq:module-orbit-coord-chart}
        \nu: \mathring\Delta^{\kappa_1(\bm{M})} \times \cdots \times \mathring\Delta^{\kappa_n(\bm{M})} \to \aut_0 (I^n, \leq).[\bm{M}].
    \end{equation}

    First, we build a lift, $\Phi: \mathring\Delta^{\kappa_1(\bm{M})} \times \cdots \times \mathring\Delta^{\kappa_n(\bm{M})} \to \aut_0 (I^n, \leq)$.
    Let $x = (x^1, \dots, x^n)$ with each $x^j \in \mathring\Delta^{\kappa_j(\bm{M})}$.
    Then $\Phi(x) = \Phi(x)_1 \times\dots\times \Phi(x)_n$ where each $\Phi(x)_j$ is defined as follows: $\Phi(x)_j$ sends 0 to 0, 1 to 1, the elements of $G_j \cap (0,1)$ to the coordinates of $x^j$ in order, and interpolates linearly between these points. We then define
    \begin{equation*}
        \nu: x \mapsto [\Phi(x).\bm{M}].
    \end{equation*}

    Now $\Phi(x)(G)$ is the minimal grid for $[\Phi(x).\bm{M}]$, but $\Phi(x)_j(G_j) \cap (0,1) = \Phi(x)_j(G_j\cap (0,1)) = x^j$, so $\chi \circ \nu$ is identity on $\mathring\Delta^{\kappa_1(\bm M)} \times \cdots \times \mathring\Delta^{\kappa_n(\bm M)}$.

    Conversely, let $\bm{N} = \phi.\bm{M}$, with $\phi = \phi_1\times\dots\times \phi_n$. $\Phi_j(\chi([\bm{N}]))$ sends 0 to 0, 1 to 1 and the elements of $G_j \cap (0,1)$ to the elements of $\phi_j(G_j) \cap(0,1)$ in order. Hence, $\Phi(\chi([\bm{N}]))|_G = \phi|_G$ and by \autoref{prop:same-grid-action-iso} applied component-wise, $\Phi(\chi([\bm{N}])).\bm{M} \cong \phi.\bm{M} = \bm{N}$ so $\nu \circ \chi$ is the identity on $\aut_0 (I^n, \leq).[\bm{M}]$.

    To prove the continuity of $\chi$, fix $\epsilon>0$ and let $[\bm{N}] \in \aut_0 (I^n, \leq).[\bm{M}]$ with minimal grid $H$. Recall
    \begin{equation*}
        \ell_{\bm{N}} = \min_{\substack{h,h'\in H\\h\neq h'}} \|h-h'\|_\infty
    \end{equation*}
    if $|H| \geq 2$, and set $\ell_{\bm{N}} = \infty$ otherwise. Fix $\delta > 0$ such that $\delta < \min\{\epsilon, \ell_{\bm{N}}/2\}$. Then, by \autoref{klem:nearby-grids}, for all $[\bm{N}'] \in \aut_0 (I^n, \leq).[\bm{M}]$ such that $d_I(\bm{N},\bm{N}') \leq \delta$, the minimal grid of $\bm{N}'$, which we denote by $H'$, contains a subgrid $H'' \subseteq H'$ that is $\epsilon$-close to $H$ and satisfies $|H_j| \leq |H_j''| \leq |H_j'|$ in each coordinate $j$. Since $\bm{N}, \bm{N}'$ are in the same orbit, we know $|H_j'| = |H_j|$ for all $j$, and hence $H'' = H'$, and
    \begin{equation*}
        d_\infty(\chi([\bm{N}]), \chi([\bm{N}'])) \leq \delta < \epsilon.
    \end{equation*}
    
    Finally, the continuity of $\nu$ follows directly from the continuity of the action. If $x, y \in \mathring\Delta^{\kappa_1(\bm{M})} \times \cdots \times \mathring\Delta^{\kappa_n(\bm{M})}$ satisfy $\|x-y\|_\infty \leq \epsilon$, then clearly $d_\infty(\Phi(x), \Phi(y)) \leq \epsilon$ and by \autoref{thm:continuity-action-modules}(i), $d_I(\Phi(x).\bm{M},\Phi(y).\bm{M}) = d_I(\nu(x),\nu(y)) \leq \epsilon$.

    If $\kappa_j(\bm{M}) = 0$ for some $j$, the coordinate set $G_j$ is contained in the boundary, so for every $\bm{N} \in {\aut_0(I^n,\leq).\bm{M}}$ with minimal grid $H$, $H_j = G_j$. We modify the components of the morphisms $\chi$ and $\nu$ for which $\kappa_j(\bm{M}) = 0$. For $\chi$, instead of recording the values $H_j \cap (0,1)$, we simply map the persistence module to the unique element of $ \mathring{\Delta}^0$. For $\nu$, we only need to modify $\Phi$: if $x^j \in \Delta^0$, then $\Phi(x)_j = \id$. The modified maps are still continuous and inverses of each other.
\end{proof}

\begin{rem}
    The map $\chi$, which records the positions of the points in the minimal grid of a module, cannot be extended continuously to $\modspace{n}$. For any given module, consider the direct sum of this module and an interval module with support in $I^n$, arbitrarily close to the 0 module in the interleaving distance. However close to the original module, the resulting minimal grid can be at a fixed distance from the original.
\end{rem}

Next, we show that the closure of the orbit corresponds to the orbit of the closure.

\begin{prop}
\label{prop:closure-orbit-modules}
    Let $[\bm{N}] \in \modspace{n}^{\dim K + 1}$. Then
    \begin{equation*}
        \overline{\aut_0 (I^n, \leq).[\bm{N}]} = \overline{\aut_0 (I^n, \leq)}.[\bm{N}]
    \end{equation*}
\end{prop}

\begin{proof}
Let $[\bm{M}] \in \overline{\aut_0 (I^n, \leq).[\bm{N}]} \subseteq \modspace{n}^{\dim K + 1}$. There exists a sequence of maps $(\phi^k)_{k \in \mathbb{N}} \subseteq {\aut_0 (I^n, \leq)}$ such that $\phi^k.\bm{N}$ converges to $\bm{M}$ component-wise in the interleaving distance.

Let $G$ be the minimal grid for $\bm{M}$. Recall that we define an extended grid $\overline{G}$ by adding the boundary points 0 and 1 to every coordinate set: $\overline{G}_j = G_j \cup \{0,1\}$ for all $j$, and set $\bar\ell_{\bm{M}} = \ell_{\overline{G}}$. Let $H \neq \varnothing$ be a grid for $\bm{N}$.

Note that by \autoref{prop:grid-under-action}, the sizes of the grid remain unchanged in the sequence: for all $k \in \mathbb{N}$, $\phi^k(H)$ is a grid for $\phi^k.\bm{N}$, and for all $j = 1,\dots, n$, $|\phi^k(H)_j|=|H_j|$. Fix a strictly positive
\begin{equation*}
    \epsilon < \frac{\bar\ell_{\bm{M}}}{2\max_j|H_j|}.
\end{equation*}

Since the sequence converges, there exists $k_0 \in \mathbb{N}$ such that for all $k \geq k_0$, $d_I(\phi^k.\bm{N}, \bm{M}) \leq \epsilon$. Therefore, we can apply \autoref{lem:squeeze-grid}, which gives us a map $\psi \in  \overline{\aut_0(I^n, \leq)}$ such that
\begin{equation*}
    \psi.\phi^{k_0}.\bm{N} \cong \bm{M},
\end{equation*}
so $[\bm{M}] \in \overline{\aut_0 (I^n, \leq)}.[\bm{N}]$.

To prove the other inclusion, consider $\phi.[\bm{N}] \in \overline{\aut_0 (I^n, \leq)}.[\bm{N}]$. There exists a sequence of maps $(\phi^k)_{k \in \mathbb{N}} \subset \aut_0 (I^n, \leq)$ that converges to $\phi$ uniformly. 
By the continuity of the action (\autoref{thm:continuity-action-modules}, \autoref{rem:continuity-iso-classes}), the sequence $([\phi^k.\bm{N}])_{k \in \mathbb{N}} \subset \aut_0 (I^n, \leq).[\bm{N}]$ converges to $[\phi.\bm{N}]$ component-wise in the interleaving distance, so
\begin{equation*}
    [\phi.\bm{N}] \in \overline{\aut_0 (I^n, \leq).[\bm{N}]}, 
\end{equation*}
which finishes the proof.
\end{proof}

\begin{rem}
\label{rem:nu-extension}
    Recall the homeomorphism $\nu$ introduced in the proof of \autoref{prop:homeo-orbits-modules}. Note that $\overline{\aut_0(I^n, \leq)}$ acts on $\Delta^{\kappa_1(\bm{M})} \times \cdots \times \Delta^{\kappa_n(\bm{M})}$ block-diagonally. Since $\nu$ is equivariant with respect to this action restricted to $\aut_0(I^n, \leq)$, we can extend it to an equivariant map from the closed product of simplices:
\begin{equation*}
    \overline{\nu}: \Delta^{\kappa_1(\bm{M})} \times \cdots \times \Delta^{\kappa_n(\bm{M})} \to \overline{\aut_0 (I^n, \leq).[\bm{M}]} = \overline{\aut_0 (I^n, \leq)}.[\bm{M}].
\end{equation*}

Any point in the closure $ \Delta^{\kappa_1(\bm{M})} \times \dots \times \Delta^{\kappa_n(\bm{M})}$ can be expressed as $\phi.x$ for some map $\phi \in \overline{\aut_0(I^n, \leq)}$ and interior point $x\in \mathring\Delta^{\kappa_1(\bm{M})} \times \dots \times \mathring\Delta^{\kappa_n(\bm{M})}$. Then, let
\begin{equation*}
    \overline{\nu}(\phi.x) = \phi.\overline{\nu}(x).
\end{equation*}
To check the well-definedness of $\overline{\nu}$, consider $\phi, \phi' \in \overline{\aut_0(I^n, \leq)}$ with $\phi.x = \phi'.x'$. Recall that $\nu$ is defined through a lift $\Phi: \mathring\Delta^{\kappa_1(\bm{M})} \times \cdots \times \mathring\Delta^{\kappa_n(\bm{M})} \to \aut_0 (I^n, \leq)$. Let $G$ be the minimal grid for $\bm{M}$. Then $\Phi(x)_j$ fixes 0 and 1, and sends elements of $G_j\cap (0,1)$ to the coordinates of $x^{j}$ in order. Hence, $\phi \circ \Phi(x)|_G = \phi' \circ \Phi(x')|_G$, and by \autoref{prop:same-grid-action-iso},
\begin{equation*}
    \nu(\phi.x) = [\phi.(\Phi(x).\bm{M})] = [\phi'.(\Phi(x').\bm{M})] = \nu(\phi'.x').
\end{equation*}

The map $\nu$ is a homeomorphism, and its extension $\overline{\nu}$ is continuous and surjective, but not necessarily injective on the boundary.
\end{rem}

Finally, we define a stratification of $\modspace{n}^K$ based on the orbits of the action of $\aut_0(I^n, \leq)$ on $\modspace{n}^{\dim K +1}$.
\begin{prop}
\label{prop:mod-strat}
    For each $i \geq 0$, let $X_i$ be the union of the $\aut_0(I^n, \leq)$-orbits of dimension at most $i$ in $\modspace{n}^K$. This defines a stratification of $\modspace{n}^K$. The $i$-strata are given by the orbits of dimension $i$.
\end{prop}

\begin{proof}
First, note that by the equivariance of $MPH$ (\autoref{thm:equivariance}), the image $\modspace{n}^K$ can be seen as a union of $\aut_0(I^n, \leq)$-orbits. Let $f \in \filt^n$, then the orbit of $MPH(f)$ is also in the image:
\begin{align}
\label{eq:strongly-stratified-map}
    \aut_0(I^n, \leq).MPH(f) &= \{\phi.MPH(f) = MPH(\phi.f) \mid \phi \in \aut_0(I^n, \leq)\} = \nonumber \\
    &= MPH(\aut_0(I^n, \leq).f)
\end{align}

The minimal grid containing $\im f$ is a grid for $MPH(f)$ which is not necessarily minimal. Therefore, $\odim(MPH(f)) \leq \odim(f)$. The maximum dimension of the orbits in $\modspace{n}^K$ is thus $n|K|$, and we get a sequence
\begin{equation*}
    \varnothing = X_{-1} \subseteq X_0 \subseteq \dots \subseteq X_{n|K|} = \modspace{n}^K.
\end{equation*}

To prove that each $X_i$ is closed, we show its complement is open. Let $[\bm{M}] \in \modspace{n}^K \setminus X_i$, equivalently $\odim(\bm M) > i$. By \autoref{klem:nearby-grids}(i), for any $[\bm{N}] \in \modspace{n}^K$ with $d_I(\bm{M},\bm{N}) \leq \epsilon < \bar\ell_{\bm{M}}/2$, the minimal grid for $\bm{N}$ contains at least as many interior points in each coordinate set as the minimal grid for $\bm{M}$, hence $\odim(\bm{N}) > i$. Hence $B([\bm{M}], \bar\ell_{\bm{M}}/2) \subseteq \modspace{n}^K \setminus X_i$, and the complement is open.

To see that $X_i \setminus X_{i-1}$ is a topological manifold, we first check that $X_i \setminus X_{i-1}$ is locally Euclidean. Let $[\bm{M}] \in X_i \setminus X_{i-1}$, equivalently $\odim(\bm{M}) = i$.
Let $\epsilon < \bar\ell_{\bm{M}}/4$ and let $[\bm{N}] \in X_i \setminus X_{i-1}$ be such that $d_I(\bm{M},\bm{N}) \leq \epsilon$, with minimal grid $H$. 
For every $j$ and every $x\in G_j\cap(0,1)$,
\autoref{klem:nearby-grids}(i) provides a point $y\in H_j$ such that $|x-y|\leq\epsilon$. By the choice of $\epsilon$, this ensures $\kappa_j(\bm{N})\geq\kappa_j(\bm{M})$. Since the sum is equal $\odim(\bm{N}) = \odim(\bm{M})$, then $\kappa_j(\bm{N})= \kappa_j(\bm{M})$ for each $j$.

We next show that $\epsilon < \bar\ell_{\bm N}/2$. If $y\neq y'$ are interior points of $H_j$, let $x,x'$ be their
corresponding points in $G_j$. Then
\begin{equation*}
    |y-y'|
    \geq |x-x'|-|x-y|-|y'-x'|
    \geq |x-x'|-2\epsilon
    \geq \ell_{\overline G}-2\epsilon.
\end{equation*}

If one of $y'$ is a boundary point instead, then
\begin{equation*}
    |y-y'|
    \geq |x-y'|-|x-y|
    \geq \ell_{\overline G}-\epsilon.
\end{equation*}

Consequently,
\begin{equation*}
    \bar\ell_{\bm{N}} = \ell_{\overline H}
    \geq\ell_{\overline G}-2\epsilon
    >\frac{\ell_{\overline G}}2 = \frac{\bar\ell_{\bm{M}}}{2},
\end{equation*}
and hence $\epsilon < \bar\ell_{\bm N}/2$ and $\epsilon< \min\{\bar\ell_{\bm{M}},\bar\ell_{\bm{N}}\}/2$. By \autoref{lem:closer-than-granularity}, $[\bm{N}] \in \aut_0(I^n, \leq).[\bm{M}]$, so
\begin{equation*}
    B([\bm{M}], \bar\ell_{\bm{M}}/4) \cap (X_i \setminus X_{i-1}) = B([\bm{M}], \bar\ell_{\bm{M}}/4) \cap \aut_0(I^n, \leq).[\bm{M}].
\end{equation*}
By \autoref{prop:homeo-orbits-modules}, the orbit is homeomorphic to a product of open simplices. The right-hand side is open in the orbit, hence homeomorphic to an open subset of $\mathbb{R}^i$. This shows $X_i \setminus X_{i-1}$ is locally Euclidean.

By the above reasoning, the $i$-dimensional orbits are open in $X_i \setminus X_{i-1}$ and disjoint, and each is path-connected. Therefore, they are the path-connected components of their union $X_i \setminus X_{i-1}$, or the $i$-strata.

Since $\modspace{n}^{\dim K + 1}$ is a metric space, it is also Hausdorff, and so is $\modspace{n}^K$. As we have seen, each orbit in $\modspace{n}^K$ is the image of an orbit in $\filt^n$. Hence, $X_i \setminus X_{i-1}$ contains only finitely many $i$-dimensional orbits. Each $i$-dimensional orbit is second countable, and there are only finitely many, so $X_i \setminus X_{i-1}$ is second countable too, and thus it is a topological manifold.
\end{proof}

\begin{rem}
We can consider this construction for the whole space $\modspace{n}^{\dim K + 1}$ instead of only for the image $\modspace{n}^K$. In this case, the nested sequence of $X_i$ is infinite, since modules in $\modspace{n}^{\dim K + 1}$ can have infinitely fine minimal grids. One can easily check that the proof above still holds, except for the last step, proving second countability of the $X_i \setminus X_{i-1}$. Each of the orbits forming $X_i \setminus X_{i-1}$ are homeomorphic to a product of open simplices, so they are second countable, but in general, we have no control over the number of distinct orbits.

For an arbitrary base field $\Bbbk$, there might be continuous families of isomorphism classes of modules with a certain fixed grid size, which means continuous families of orbits in $X_i \setminus X_{i-1}$, which would then not necessarily second countable.
In the case $n = 1$, the space of finite barcodes is second countable \cite[Theorem 4]{Bubenik2018}.
More generally, it was recently proved in \cite[Proposition 4.27]{Bauer2026} that the closure of the space of finitely presentable $n$-parameter persistence modules is separable if and only if the base field $\Bbbk$ is countable or $n=1$. Since separability and second-countability are equivalent for extended pseudometric spaces \cite[Lemma 17]{Bubenik2018}, we get then that $\modspace{n}$ is second-countable if and only if $\Bbbk$ is countable or $n=1$.
\end{rem}

\subsection{\texorpdfstring{$MPH$}{MPH} as a stratified map and the moduli space \texorpdfstring{$\modspace{n}^K$}{M} as a quotient space}

\begin{thm}
\label{thm:MPH-stratified-map}
    The multiparameter persistent homology map is a strongly stratified map, that is, the image of an $i$-stratum in $\filt^n$ is a $j$-stratum in $\modspace{n}^{\dim K + 1}$, with $j \leq i$.
\end{thm}

\begin{proof}
    This is a direct consequence of the fact that both domain and codomain are stratified by the orbits of their respective $\aut_0(I^n,\leq)$ actions (\autoref{prop:filt-strat}, \autoref{prop:mod-strat}) and the equivariance of $MPH$ (\autoref{thm:equivariance}), which ensures the image of an orbit in $\filt^n$ is an orbit in $\modspace{n}^K$ (see \eqref{eq:strongly-stratified-map}).
\end{proof}

\begin{prop}
    The preimages of $MPH$ over elements of the same stratum are pairwise homeomorphic.
\end{prop}

\begin{proof}
    Let $[\bm{M}], [\bm{N}] \in \modspace{n}^K$ with $\bm{N} = \phi.\bm{M}$ for some $\phi \in \aut_0(I^n, \leq)$, then from the equivariance (\autoref{thm:equivariance}), $MPH^{-1}([\bm{N}]) = MPH^{-1}([\phi.\bm{M}]) = \phi.MPH^{-1}([\bm{M}])$, so post-composition by the map $\phi$ induces a homeomorphism between $MPH^{-1}([\bm{M}])$ and $MPH^{-1}([\bm{N}])$.
\end{proof}

\begin{prop}
    The quotient topology on $\modspace{n}^K := MPH(\filt^n)$ induced by the $MPH$ map coincides with the topology induced by the interleaving distance
    \begin{equation*}
        (\filt^n / \sim ) \xrightarrow{\cong} \modspace{n}^K,
    \end{equation*}
    where $f\sim g$ if and only if $MPH(f) = MPH(g)$.
\end{prop}

\begin{proof}
    By \autoref{thm:stability}, $MPH$ is a continuous map, so by the universal property of the quotient, it induces a continuous bijection between $\filt^n/\sim$ and $\modspace{n}^K$.

    To show that the inverse is continuous, note that $\filt^n$ is compact, and thus so is $\filt^n/\sim$, and $\modspace{n}^{\dim K+1}$ is a metric space so it is Hausdorff, and so is the image $\modspace{n}^K$. It is a standard result that a continuous bijective map from a compact space to a Hausdorff space is a homeomorphism. This suffices to finish the proof.
\end{proof}

\begin{lem}
\label{lem:grid-cw-injective-filter}
    Let $f \in\filt^n$ be injective in all coordinates, and let $G$ be the minimal grid of $MPH(f)$. Then $\im f \subset G$ and $\odim(MPH(f)) = \odim(f)$. If, moreover, $\im f \subset (0,1)^n$, then $MPH(f) \in \modspace{n}^K$ lies in a top dimensional stratum $\odim(MPH(f)) = n|K|$.
\end{lem}

\begin{proof}
    Take a point in the image $u = f(\sigma) \in \im f$. Since $f$ is coordinate-wise injective, for any $j = 1, \dots, n$ and any $\epsilon >0$ small enough, $K(f)_{u} = K(f)_{u-\epsilon e_j} \cup \{\sigma\}$.

    Since a single simplex is added, the Euler characteristic of the filtration necessarily changes when going from grade $u-\epsilon e_j$ to $u$. This forces a change in dimension in at least one homology degree. More specifically, either $MPH_{\dim \sigma}(f)_{u-\epsilon e_j, u}$ is not surjective or $MPH_{\dim \sigma-1}(f)_{u-\epsilon e_j, u}$ is not injective.
    
    By the characterisation from \autoref{lem:ess-fin-characterisation}, $u_j$ is contained in the $j$-th coordinate set of the minimal grid $G$, and we conclude that $\im f_j \subseteq G_j$ for all $j$. Since a grid containing $\im f$ is a grid for $MPH(f)$, necessarily $G_j \subseteq \im f_j$, so we find $G_j = \im f_j$, and $\odim(MPH(f)) = \odim(f)$. If $\im f \subset (0,1)^n$, in particular $\odim(f) = n|K|$.
\end{proof}

\begin{rem}
    The result $\odim(MPH(f)) = \odim(f)$ remains true if each $f_j$ is only injective when restricted to $f_j^{-1}((0,1))$.
\end{rem}

\begin{prop}
    The space $\modspace{n}^K$ is the closure of its top dimensional strata.
\end{prop}

\begin{proof}
    Let $MPH(g) \in \modspace{n}^K$. Then there exists $f \in \filt^n$ injective in all coordinates, with $\im f \subset (0,1)^n$ and $\phi \in \overline{\aut_0(I^n,\leq)}$ such that $g = \phi \circ f$, by collapsing values of $f_j$ to each other or to the boundary points.
    Then $MPH(g) = MPH(\phi.f) = \phi.MPH(f)$. By the above argument, $MPH(f)$ belongs to a top-dimensional stratum of $\modspace{n}^K$, and by \autoref{prop:closure-orbit-modules}, $MPH(g)$ belongs to its closure.
\end{proof}

\newpage

\section{The polyhedral structure of \texorpdfstring{$MPH^{-1}([\bm{M}])$}{the fiber of MPH}}
\label{sec:poly}

In this section we describe the fiber of $MPH$ over a point $[\bm{M}] \in \modspace{n}^K$ in geometric terms. Using the fact that $MPH$ is a strongly stratified map, as established in \autoref{sec:strat}, we find that $MPH^{-1}([\bm{M}])$ is a polyhedral complex, with each polyhedron given by the intersection of the fiber with the closure of a filter stratum. Moreover, over each module stratum of $\modspace{n}^K$, the restriction of $MPH$ is a trivial fiber bundle, and the polyhedral structure of the fiber is constant along the stratum.

We then turn to the dimension of the fiber. After introducing Betti tables for essentially finite modules, we use recent results on homological Morse numbers to obtain an upper bound on $\dim MPH^{-1}([\bm M])$ that depends only on the Betti tables of $\bm M$. In the one-parameter case, this recovers the bound obtained in \cite{LeygonieTillmann2022}.

\subsection{Polyhedral structure on the fiber}

We recall the definitions of polyhedron and polyhedral complex, and we characterise $MPH$ restricted over a single module stratum as a trivial fiber bundle with fiber a polyhedral complex.

\begin{defn}
    A polyhedron $\mathcal P$ is either the empty set or a bounded finite
    intersection of closed half-spaces in a Euclidean space.
    The dimension of a non-empty polyhedron is the dimension of its affine hull,
    and we set $\dim(\varnothing)=-1$.
    A face of a polyhedron is either the empty set or the intersection of the polyhedron with a supporting hyperplane, that is, a hyperplane that contains at least one boundary point of $\mathcal{P}$, and such that $\mathcal{P}$ is entirely contained in one of the two closed half-spaces bounded by it.
    We denote the relative interior and relative boundary of $\mathcal{P}$ by $\mathring{\mathcal{P}}$ and $\partial \mathcal{P}$, respectively.
\end{defn}

\begin{defn}
    A polyhedral complex is a finite collection of polyhedra $\Pi$ in a Euclidean space $\mathbb{R}^n$, such that
    \begin{enumerate}[label=(\roman*)]
        \item if $F$ is a face of $\mathcal{P} \in \Pi$, then $F \in \Pi$
        \item for any $\mathcal{P}, \mathcal{P'} \in \Pi$, the intersection $\mathcal{P}\cap\mathcal{P'}$ is a face of both $\mathcal{P}$ and $\mathcal{P'}$.
    \end{enumerate}
    The dimension of a polyhedral complex is the maximum dimension of its constituent polyhedra.

    A polyhedral map is a map that sends a polyhedron of the domain to a polyhedron of the codomain surjectively, and whose restriction to each polyhedron is affine.
\end{defn}

\begin{thm}
\label{thm:topology-of-fiber}
    Let $\mathcal{T} \subseteq \modspace{n}^K$ be a stratum in the image of $MPH$, and let $[\bm{M}] \in \mathcal{T}$. For each stratum $\mathcal{S} \subseteq \filt^n$, let
    \begin{equation*}
        \preimov{M}{S} = \overline{MPH^{-1} ([\bm{M}]) \cap \mathcal{S}} \subseteq \mathbb{R}^{n|K|}
    \end{equation*}
    be the closure of the restriction to the stratum $\mathcal{S}$ of the fiber of $MPH$ over $[\bm{M}]$. Then:
    \begin{enumerate}[label=(\alph*)]
        \item Each non-empty $\preimov{M}{S}$ is a polyhedron in $\mathbb{R}^{n|K|}$ of dimension $\dim(\mathcal{S}) - \dim (\mathcal{T})$, affinely isomorphic to the product of $\dim(\mathcal{T})+n$ closed simplices of various dimensions. Moreover, the relative interior of $\preimov{M}{S}$ is $\preim{M}{S}$;

        \item The fiber $MPH^{-1}([\bm{M}])$ is the support of the polyhedral complex
        \begin{equation*}
            \left\{ \preimov{M}{S} \mid \mathcal{S} \subseteq \filt^n \text{ a filter stratum}\right\};
        \end{equation*}

        \item Over each stratum of $\modspace{n}^K$, $MPH$ can be seen as a trivial fiber bundle, with a polyhedral complex as the fiber. That is, there is a homeomorphism $\Psi$ which makes the following diagram commute,
        \begin{equation}
        \label{eq:trivial-bundle-diagram}
        \begin{tikzcd}
        	{\mathcal{T} \times MPH^{-1}([\bm{M}])} && {MPH^{-1}(\mathcal{T})} \\
        	& {\mathcal{T}}
        	\arrow["\Psi", from=1-1, to=1-3]
        	\arrow["{\pi_1}"', from=1-1, to=2-2]
        	\arrow["MPH", from=1-3, to=2-2]
        \end{tikzcd}
        \end{equation}
        where $\pi_1$ is the projection onto the first factor. Additionally, for any other module $[\bm{N}] \in \mathcal{T}$ and filter stratum $\mathcal{S}$, the restricted map $\Psi([\bm{N}], -)|_{\mathcal{S}}$ is an affine isomorphism between $\preim{M}{S}$ and $\preim{N}{S}$, so in particular, the polyhedral structure of the fiber is the same for all modules in $\mathcal{T}$.
    \end{enumerate}
\end{thm}

\begin{proof}
To prove this result, we need the key observation that, restricted to a filter stratum, the $MPH$ map can be seen as a projection that records the values of a filter which belong to the minimal grid. Let $f \in MPH^{-1}([\bm{M}])\cap \mathcal{S}$, and let
\begin{equation*}
    \mu:  \mathring\Delta^{\kappa_1(f)} \times \dots \times \mathring\Delta^{\kappa_n(f)} \rightarrow \mathcal{S},
    \quad
    \nu: \mathring\Delta^{\kappa_1(\bm{M})} \times \cdots \times \mathring\Delta^{\kappa_n(\bm{M})} \to \mathcal{T}
\end{equation*}
be the coordinate charts of the filter orbit and its orbit image, as defined in \eqref{eq:filter-orbit-coord-chart} and \eqref{eq:module-orbit-coord-chart} respectively. Let $x = \mu^{-1}(f)$, with $x = (x^1, \dots, x^n)$. For $\kappa_j(f) \neq 0$, $x^j = (x^j_1, \dots, x^j_{\kappa_j(f)})\in \mathring\Delta^{\kappa_j(f)}$, the interior values that $f_j$ takes, in order. If $\kappa_j(f) = 0$, $x^j \in \mathring\Delta^0 = \{*\}$.

The minimal grid of $[\bm{M}] = MPH(f)$, $G$, is contained in the minimal grid containing $\im f$. For $\kappa_j(\bm{M}) \neq 0$, we have
\begin{equation*}
    G_j \cap (0,1) = (x^j_{i_{j,1}}, \dots, x^j_{i_{j,\kappa_j(\bm{M})}}),
\end{equation*}
for some ordered list of indices $i_{j,1}, \dots, i_{j,\kappa_j(\bm{M})} \in \{1, \dots, \kappa_j(f)\}$.
We define a product of projections that projects onto the coordinates of these indices:
\begin{equation*}
    \pi^{\mathcal{S}}_{\mathcal{T}}:
    \mathring\Delta^{\kappa_1(f)} \times \dots \times \mathring\Delta^{\kappa_n(f)} 
    \to
    \mathring\Delta^{\kappa_1(\bm{M})} \times \cdots \times \mathring\Delta^{\kappa_n(\bm{M})}.
\end{equation*}
Let $y = (y^1, \dots, y^n)\in \mathring\Delta^{\kappa_1(f)} \times \dots \times \mathring\Delta^{\kappa_n(f)}$.
If $\kappa_j(\bm{M}) = 0$, $\pi^{\mathcal{S}}_{\mathcal{T}}(y)^j$ is the single point of $\mathring\Delta^{\kappa_j(\bm{M})}$. Otherwise
\begin{equation*}
    \pi^{\mathcal{S}}_{\mathcal{T}}(y)^j = (y^j_{i_{j,1}}, \dots, y^j_{i_{j,\kappa_j(\bm{M})}}).
\end{equation*}

In particular, $\pi^{\mathcal{S}}_{\mathcal{T}}(x)^j = G_j \cap (0,1)$ for all $j$ such that $\kappa_j(\bm{M}) \neq 0$. The projection $\pi^{\mathcal{S}}_{\mathcal{T}}$, as well as the homeomorphisms $\mu$, $\nu$, are equivariant with respect to the $\aut_0(I^n, \leq)$ actions, which ensures that the following diagram commutes,
\begin{equation}
\label{eq:diagram-fiber}
\begin{tikzcd}
	{\mathring\Delta^{\kappa_1(f)} \times \dots \times \mathring\Delta^{\kappa_n(f)}} & {\mathcal{S}} \\
	{\mathring\Delta^{\kappa_1(\bm{M})} \times \cdots \times \mathring\Delta^{\kappa_n(\bm{M})}} & {\mathcal{T}}
	\arrow["\mu", from=1-1, to=1-2]
	\arrow["{\pi^{\mathcal{S}}_{\mathcal{T}}}"', from=1-1, to=2-1]
	\arrow["MPH|_{\mathcal{S}}", from=1-2, to=2-2]
	\arrow["\nu", from=2-1, to=2-2]
\end{tikzcd}
\end{equation}
Indeed, for any $y \in \mathring\Delta^{\kappa_1(f)} \times \dots \times \mathring\Delta^{\kappa_n(f)}$, there exists $\phi \in \aut_0(I^n, \leq)$ with $y = \phi.x$. Then,
\begin{equation*}
    MPH \circ \mu(y) = \phi.(MPH(f)) = \phi.\nu(\pi^{\mathcal{S}}_{\mathcal{T}}(x)) = \nu \circ \pi^{\mathcal{S}}_{\mathcal{T}}(y).
\end{equation*}

Let us prove assertion (a). Since $\mu^{-1}$ is an affine isomorphism, it restricts to an affine isomorphism of the fiber. This, combined with the commutativity of \eqref{eq:diagram-fiber}, yields
\begin{equation}
\label{eq:aff-homeo-preim}
    \preim{M}{S} \cong^{\text{aff}} \mu^{-1}(\preim{M}{S}) = (\pi^{\mathcal{S}}_{\mathcal{T}})^{-1} (\nu^{-1}([\bm{M}])).
\end{equation}
On the other hand, since we have that $\nu^{-1}([\bm{M}]) = \pi^{\mathcal{S}}_{\mathcal{T}}(x)$, so we can describe the preimage of the projection as a product:
\begin{equation*}
    (\pi^{\mathcal{S}}_{\mathcal{T}})^{-1}(\pi^{\mathcal{S}}_{\mathcal{T}}(x)) =
    \prod_{j=1}^n P_j,
\end{equation*}
where $P_j = \mathring\Delta^{\kappa_j(f)}$ for each $j$ with $\kappa_j(\bm{M}) = 0$, and otherwise

\begin{equation*}
\begin{aligned}
P_j
&=
\{z \in \mathring\Delta^{\kappa_j(f)} \mid z_{i_{j,k}} = x^j_{i_{j,k}} \quad \forall k = 1,\dots,\kappa_j(\bm{M})\}
\\
&=
\left\{
z\in(0,1)^{\kappa_j(f)}
\;\middle|\;
\begin{aligned}
&(0<z_1<\cdots<z_{i_{j,1}-1}
  <z_{i_{j,1}}=x^j_{i_{j,1}})
\\[-1mm]
&\cap
(x^j_{i_{j,1}}<z_{i_{j,1}+1}<\cdots
  <z_{i_{j,2}-1}<z_{i_{j,2}}=x^j_{i_{j,2}})
\\[-1mm]
&\cap\cdots\cap
\\[-1mm]
&
(x^j_{i_{j,\kappa_j(\bm M)}}
 <z_{i_{j,\kappa_j(\bm M)}+1}<\cdots
 <z_{\kappa_j(f)}<1)
\end{aligned}
\right\}
\\
&\cong^{\mathrm{aff}}
\mathring\Delta^{i_{j,1}-1}
\times
\left(
\prod_{k=1}^{\kappa_j(\bm M)-1}  \mathring\Delta^{i_{j,k+1}-i_{j,k}-1}
\right)
\times
\mathring\Delta^{
\kappa_j(f)-i_{j,\kappa_j(\bm M)}
}.
\end{aligned}
\end{equation*}

For each $j$ in the product, we get a product of $\kappa_j(\bm{M}) + 1$ simplices, with total dimension $\kappa_j(f)-\kappa_j(\bm{M})$. Therefore, $\preim{M}{S}$ is affinely isomorphic to a product of $\dim(\mathcal{T})+n$ open simplices of total dimension $\odim(f) - \odim(\bm{M}) = \dim(\mathcal{S}) - \dim (\mathcal{T})$.
By \autoref{prop:homeo-orbits-filt}, we can extend $\mu$ to an affine homeomorphism between the closures $\overline{\mu}: \Delta^{\kappa_1(f)} \times \dots \times \Delta^{\kappa_n(f)} \rightarrow \overline{\mathcal{S}}$, hence the affine homeomorphism in \eqref{eq:aff-homeo-preim} extends to an affine homeomorphism of the closures.
Setting $J_0 = \{j\in \{1,\dots,n\} \mid \kappa_j(\bm M) = 0\}$ and $J_+ = \{1,\dots,n\} \setminus J_0$,
\begin{equation}
\label{eq:aff-homeo-preimov}
    \preimov{M}{S}
    \cong^{\text{aff}}
    \prod_{j \in J_0} \Delta^{\kappa_j(f)} 
    \times
    \prod_{j\in J_+}
    \left(
    \Delta^{i_{j,1}-1}
    \times
    \prod_{k=1}^{\kappa_j(\bm M)-1}  \Delta^{i_{j,k+1}-i_{j,k}-1}
    \times
    \Delta^{
    \kappa_j(f)-i_{j,\kappa_j(\bm M)}
    }
    \right),
\end{equation}
which finishes the proof of assertion (a).

Before proceeding to assertion (b), we will show that if $\preim{M}{S} \neq \varnothing$ for some filter stratum $\mathcal{S}$, then
\begin{equation*}
    \preimov{M}{S} = MPH^{-1} ([\bm{M}]) \cap \overline{\mathcal{S}}
\end{equation*}

The inclusion $\subseteq$ is direct. Now consider $g \in MPH^{-1} ([\bm{M}]) \cap \overline{\mathcal{S}}$, and let $f \in \preim{M}{S}$. By the structure of $\overline{\mathcal{S}}$ (\autoref{prop:closure-orbit-filt}), we know there exists $\phi \in \overline{\aut_0(I^n, \leq)}$ with $g = \phi.f$. Since $[\bm{M}] = MPH(g) = \phi.MPH(f) = \phi.[\bm{M}]$, $\phi(G)$ is a grid for $\bm{M}$, so $G \subseteq \phi(G)$. Since $|G| \geq |\phi(G)|$, necessarily $\phi(G) = G$ and $\phi|_G = \id|_G$.

Now consider the continuous path $t \mapsto f_t = ((1-t)\id + t\phi).f$ on $\overline{\mathcal{S}}$, for $t \in [0,1]$. Note that $f_0 = f$ and $f_1 = g$. Since the restriction to the minimal grid satisfies $((1-t)\id + t\phi)|_G = \id|_G$, the image by $MPH$ is fixed: $MPH(f_t) = ((1-t)\id + t\phi).[\bm{M}] = [\bm{M}]$ and $f_t \in MPH^{-1}([\bm{M}])$ for all $t \in [0,1]$. Moreover, for $t \in [0,1)$, each coordinate map of $(1-t)\id + t\phi$ fixes endpoints and is strictly increasing, so it belongs to the group $\aut_0(I^n, \leq)$. Thus, $f_t \in \mathcal{S}$ for all $t \in [0,1)$ and $f_1 = g \in \preimov{M}{S}$.

As a consequence, for any boundary stratum $\mathcal{S'} \subseteq \partial \overline{\mathcal{S}}$, we find that
\begin{equation*}
    \preimov{M}{S'} \subseteq MPH^{-1} ([\bm{M}]) \cap \overline{\mathcal{S}} =\preimov{M}{S}.
\end{equation*}

Since $\mathcal{S'} \cap \mathcal{S} = \varnothing$, necessarily $\preimov{M}{S'} \subseteq \partial \preimov{M}{S}$. From this observation, together with the characterisation in \eqref{eq:aff-homeo-preimov}, assertion (b) follows by the same argument as in the single-parameter case $n=1$. For the details, see the proof of Theorem 2.2(b) in \cite{LeygonieTillmann2022}.

To prove assertion (c), let $[\bm{N}] \in \mathcal{T}$ with minimal grid $H$. Let $\psi_j \in \aut(I, \leq)$ be the unique linear interpolation that sends the elements of $G_j \cap (0,1)$ to the elements of $H_j \cap (0,1)$ in order, and let $\psi_{[\bm{N}]} = \psi_1 \times\dots\times \psi_n$. Consider the map
\begin{equation*}
\begin{array}{rrcl}
    \Psi: & \mathcal{T} \times MPH^{-1}([\bm{M}])& \to &MPH^{-1}(\mathcal{T}) \\
     & ([\bm{N}], f) & \mapsto & \psi_{[\bm{N}]}.f
\end{array}
\end{equation*}

First, we check that it is well-defined, 
\begin{equation}
\label{eq:well-defn-trivial-bundle}
    MPH(\Psi([\bm{N}], f)) = \psi_{[\bm{N}]}.MPH(f) = \psi_{[\bm{N}]}.[\bm{M}] = [\bm{N}] \in \mathcal{T}.
\end{equation}
It is also clear from this computation that the diagram \eqref{eq:trivial-bundle-diagram} commutes.

Since the action on $n$-filters is continuous, to show continuity of $\Psi$ it suffices to show continuity of the assignment $[\bm{N}] \mapsto \psi_{[\bm{N}]}$ in $\mathcal{T}$.
We can factor this assignment through the homeomorphism $\nu^{-1}: \mathcal{T} \to \mathring\Delta^{\kappa_1(\bm{M})} \times \cdots \times \mathring\Delta^{\kappa_n(\bm{M})}$, first assigning a module $[\bm{N}]$ to the interior points of its minimal grid $(H_1\cap(0,1),\dots,H_n\cap(0,1))$, and then assigning each grid to the product of linear interpolations $\psi_{[\bm{N}]}$. Both maps are continuous, and thus so is $\Psi$.

Finally, we can build a continuous inverse by
\begin{equation*}
\begin{array}{rcl}
    MPH^{-1}(\mathcal{T}) & \to & \mathcal{T} \times MPH^{-1}([\bm{M}]) \\
    g & \mapsto & (MPH(g), \psi_{MPH(g)}^{-1}.g).
\end{array}
\end{equation*}

To prove the last part of (c), fix $[\bm{N}] \in \mathcal{T}$ and a filter stratum $\mathcal{S}$ intersecting the fiber of $[\bm{M}]$ nontrivially. From the computation in \eqref{eq:well-defn-trivial-bundle}, we can restrict the codomain of the restriction: $\Psi([\bm{N}], -)|_{\mathcal{S}} : \preim{M}{S} \to \preim{N}{S}$, and this will be a homeomorphism. We will check that it is in fact an affine isomorphism.

Let $f \in \preim{M}{S}$, and fix $j = 1,\dots,n$. 
If $\kappa_j(\bm{M}) = 0$, $\psi_j = \id$. Otherwise, let $\nu^{-1}([\bm{M}])_j = (x_1,\dots,x_{\kappa_j(\bm{M})})$, and fix $x_0 = 0$, $x_{\kappa_j(\bm{M}) + 1} =1$. 
For any $\sigma \in K$, there exist $x_i \leq f_j(\sigma) \leq x_{i+1}$.
Due to the equivariance of $MPH$, for any other filter $f' = \phi.f \in \preim{M}{S}$, we have that necessarily $\phi|_G = \id$, and thus $x_i \leq f_j'(\sigma) \leq x_{i+1}$. 
Since $\psi_j$ is affine in $[x_i,x_{i+1}]$, we get that $\Psi([\bm{N}], -)|_{\mathcal{S}}$ is affine in the coordinate corresponding to $\sigma$. This finishes the proof.
\end{proof}

\subsection{Bound on the dimension of the fiber}
\label{subsec:bound}

In what follows, we deduce a bound for the dimension of the fiber of $MPH$ over a given $[\bm{M}] \in \modspace{n}^K$. By \autoref{thm:topology-of-fiber}, $MPH^{-1}([\bm{M}])$ is a polyhedral complex composed of a polyhedron
\begin{equation*}
    \overline{MPH^{-1} ([\bm{M}]) \cap \mathcal{S}}
\end{equation*}
for each filter stratum $\mathcal S$ with non-empty intersection with the fiber. Let $f \in \mathcal{S}$, then 
\begin{equation}
    \dim \overline{MPH^{-1} ([\bm{M}]) \cap \mathcal{S}} = \odim(f) - \odim(\bm{M}).
\end{equation}

Our strategy will consist of finding a bound on the dimension of the polyhedra that does not depend on the filter $f$. We can find a first bound that depends only on the size of $G$, the minimal grid of $[\bm{M}]$. Since $G \subset \im f_1 \times\dots\times \im f_n$, and $|\im f_j| \leq |K|$, we find
\begin{equation*}
    \odim(f) - \odim(\bm{M})
        = \sum_{j=1}^n (|\im f_j \cap (0,1)|-|G_j \cap (0,1)|) \leq \sum_{j=1}^n (|\im f_j|-|G_j|) \leq n|K| - \sum_{j=1}^n |G_j|
\end{equation*}

Since this inequality does not depend on the choice of stratum $\mathcal{S}$, we get
\begin{equation}
\label{eq:first-bound-fiber}
    \dim MPH^{-1}([\bm{M}]) \leq n|K| - \sum_{j=1}^n |G_j|
\end{equation}

With this bound, we can readily see that the fiber over the image of a coordinate-wise injective filter is 0-dimensional.

\begin{lem}
\label{lem:first-bound-cw-inj}
    If $f\in \filt^n$ is injective in every coordinate and $[\bm{M}] = MPH(f)$, then $MPH^{-1}([\bm{M}])$ is a finite set.
\end{lem}

\begin{proof}
    Let $G$ be the minimal grid of $\bm{M}$. By \autoref{lem:grid-cw-injective-filter}, $\im f_j \subset G_j$, and since $f$ is coordinate-wise injective, $|G_j| = |K|$. The bound in \eqref{eq:first-bound-fiber} reduces to 0. The fiber $MPH^{-1}([\bm{M}])$ is a 0-dimensional polyhedral complex, thus a finite number of points.
\end{proof}

We can obtain a second bound by refining the information carried by the minimal grid. In the one-parameter setting, this refinement is given by the multiplicities of endpoints of bars \cite[Proposition 2.5]{LeygonieTillmann2022}. In the multiparameter setting, the corresponding role is played by Betti tables (equivalently, multigraded Betti numbers), introduced in \autoref{subsec:betti}.

In the single parameter case, the appearance of a non-zero entry in the Betti tables at a certain grade (i.e. the start or ending of a bar in the barcode) necessarily implies the birth of a simplex at that grade. For multiple parameters, however, this relation is not as straightforward, as has been recently studied in \cite{GuidolinLandi2023}. Their results translate directly to essentially finite modules.

\begin{defn}[homological Morse numbers, adapted from \cite{GuidolinLandi2023}]
\label{defn:hom-Morse-num}
    Let $f \in \filt^n$ and take
    \begin{equation*}
        0<\epsilon<
        \min_{\substack{1\leq j\leq n\\|\operatorname{im}f_j|\geq2}} \min_{\substack{x,y\in\operatorname{im}f_j\\x\neq y}}|x-y|,
    \end{equation*}
    if $|\im f_j| \geq 2$ for some $j$, and any $\epsilon > 0$ otherwise.
    Recall that the sublevel set filtration at value $u \in \mathbb{R}^n$ is denoted by $K(f)_u$. The homological Morse numbers $\mu_q(u)$ of degree $q$ at $u \in \mathbb{R}^n$ are defined as:
    \begin{equation*}
        \mu_q(u) := \dim H_q \left(K(f)_u, \bigcup_{j = 1}^n K(f)_{u - \epsilon e_j} \right).
    \end{equation*}
\end{defn}

The homological Morse number $\mu_q(u)$ can be seen as the ``natural" lower bound to the number of critical simplices of dimension $q$ entering the filtration at grade $u$, for every choice of compatible discrete gradient vector field \cite[Prop. 1]{Landi2021}. In particular, this gives a weaker bound that will be of interest.
\begin{equation}
\label{eq:hom-morse-numbers-cells-born}
    \mu_q(u) \leq |\{\sigma \in K \mid \dim \sigma = q, f(\sigma) = u\}|.
\end{equation}

We reproduce here a result from \cite{GuidolinLandi2023}, where the authors give a bound to the homological Morse numbers in terms of the Betti tables of the associated multiparameter persistence modules.

\begin{thm}[\cite{GuidolinLandi2023}, Theorems 7.3 and 7.5]
\label{thm:guidolin-bound}
    Let $f \in \filt^n$ and let $\xi_i^q: \mathbb{R}^n \to \mathbb{N}$ be the $i$-th Betti table of $MPH_q(f)$, for each homological degree $q \geq 0$. Then the homological Morse numbers are bounded above and below as follows:
    \begin{equation*}
        \xi_0^q(u) + \xi_1^{q-1}(u) - \sum_{p = 1}^{n-1} \xi_{p+1}^{q-p}(u) \leq \mu_q(u) \leq \sum_{p = 0}^{n} \xi_{p}^{q-p}(u).
    \end{equation*}
\end{thm}

Intuitively, the lower bound shows that having non-trivial Betti tables at grade $u$ does not necessarily imply that many critical simplices appear at the filtration at $u$, unless the $i$-th Betti tables with $i \geq 2$ vanish at $u$.

\begin{defn}
    We denote
    \begin{equation*}
        b_q(u) := \max \left\{ 0, \xi_0^q(u) + \xi_1^{q-1}(u) - \sum_{p = 1}^{n-1} \xi_{p+1}^{q-p}(u) \right\}.
    \end{equation*}
    By \autoref{thm:guidolin-bound}, $b_q(u) \leq \mu_q(u)$.
\end{defn}

\begin{rem}
\label{rem:b_q-0-high-q}
    By the bound above, for all $q > \dim K$, $b_q(u) = 0$.
\end{rem}

Finally, we are ready to state our bound on the dimension of the fiber of $MPH$.

\begin{prop}
\label{prop:bound-dim-fiber}
    Let $[\bm{M}]\in \modspace{n}^K$ with minimal grid $G$, and let $\xi_i^q$ be the $i$-th Betti table of $M^q$. Let
    \begin{equation*}
        \mathcal{L}(\bm{M}) = \sum_{g \in G} \sum_{q=0}^{\dim K} b_q(g).
    \end{equation*}
    Then as a polyhedral complex,
    \begin{equation*}
        \dim MPH^{-1}([\bm{M}]) \leq \frac{n}{2} (|K|-\mathcal{L}(\bm{M})).
    \end{equation*}
\end{prop}

\begin{proof}
    This proof follows the structure of the proof of the bound for $n=1$ in Proposition 2.5 of \cite{LeygonieTillmann2022}, incorporating results from \cite{GuidolinLandi2023} to get a generalisation.

    Recall from \autoref{thm:topology-of-fiber} that the fiber is a polyhedral complex with a polyhedron for every non-empty intersection with a filter stratum $\mathcal{S}$. Let $f \in \mathcal{S}$, then 
    \begin{equation}
    \label{eq:dim-fiber}
        \dim \overline{MPH^{-1} ([\bm{M}]) \cap \mathcal{S}} = \odim(f) - \odim(\bm{M})
        \leq \sum_{j=1}^n (|\im f_j|-|G_j|)
    \end{equation}

    We now use the restriction to one parameter to find a bound to the number of values in $\im f_j$ that do not belong to $G_j$.
    
    In the single-parameter setting, given a filter $g: K \to I$, every value $x \in \im g$ that is not an endpoint of a bar in any homological degree of $PH(g)$ must be attained by at least two distinct simplices $\sigma, \tau \in K$: $f(\sigma) = f(\tau) = x$. Indeed, if exactly one simplex appeared at value $x$, then passing through $x$ would change the Euler characteristic of the filtration.
    This forces a change in dimension in some homology degree, creating a birth or death in homology at $x$, and hence an endpoint in the barcode.

    We can exploit this property in the multiparameter setting. Let $\sigma\in K$ satisfy $f_j(\sigma) \notin G_j$, and for each $q = 0, \dots, \dim K$ pick a representative $[M^q] = MPH_q(f)$. Then there exists $\delta>0$ small enough such that the maps $M^q_{f(\sigma)-\delta e_j, f(\sigma)}$ are isomorphisms for all $q\geq 0$. Now consider the restriction of $M^q$ to the line $\{f(\sigma) +(t-f_j(\sigma))e_j \mid t\in \mathbb{R}\}$,
    \begin{equation*}
        (D_q)_t = M^q_{f(\sigma) +(t-f_j(\sigma))e_j}.
    \end{equation*}

    By setting $\tilde{K} = \{\tau\in K\mid f_k(\tau)\leq f_k(\sigma) \text{ for every }k\neq j\}$ and $g: \tilde{K} \to I$ as $g(\tau) = f_j(\tau)$, we see that $D = PH(g)$. Since $(D_q)_{f_j(\sigma)-\delta, f_j(\sigma)} = M^q_{f(\sigma)-\delta e_j, f(\sigma)}$ is an isomorphism, we know that $f_j(\sigma) \in \im g$ is not an endpoint of the barcode of $D_q$ for any $q\geq 0$. Hence, by the discussion above, there exists $\tau \neq \sigma \in \tilde{K} \subseteq K$ such that $f_j(\tau) = f_j(\sigma)$. Therefore, for every $x \in \im f_j \setminus G_j$,
    \begin{equation*}
        2 \leq |\{\sigma \in K \mid f_j(\sigma) = x\}|.
    \end{equation*}

    Summing over all values $x \in (\im f_j \setminus G_j)$, we get
    \begin{equation*}
        2(|\im f_j|-|G_j|) \leq 
        |\{\sigma \in K \mid f_j(\sigma) \notin G_j\}| \leq 
        |\{\sigma \in K \mid f(\sigma) \notin G\}|.
    \end{equation*}

    Summing over $j = 1,\dots,n$, and combining this inequality with \eqref{eq:dim-fiber}, we find
    \begin{equation}
    \label{eq:images-out}
        \odim(f) - \odim(\bm{M}) \leq \frac{n}{2}|\{\sigma \in K \mid f(\sigma) \notin G\}|.
    \end{equation}

    We have now an upper bound of \eqref{eq:dim-fiber} in terms of the number of simplices of $K$ whose image by $f$ does not fall on the grid $G$. In order to eliminate the dependency in $f$, we now consider the number of simplices of $K$ whose image by $f$ does fall on the grid $G$.

    By combining \autoref{thm:guidolin-bound} and the inequality in \eqref{eq:hom-morse-numbers-cells-born}, we get that for every $g\in G$ and every $q\geq 0$,
    \begin{equation*}
        b_q(g) \leq |\{\sigma \in K \mid \dim \sigma = q, f(\sigma)=g\}|
    \end{equation*}

    Summing over all values of $g\in G$ and all contributing homological degrees (see \autoref{rem:b_q-0-high-q}), we get
    \begin{equation}
    \label{eq:images-in}
        \mathcal{L}(\bm{M}) = \sum_{g \in G} \sum_{q=0}^{\dim K} b_q(g) \leq |\{\sigma \in K \mid f(\sigma) \in G\}|
    \end{equation}
    
    By combining \eqref{eq:images-out} and \eqref{eq:images-in}, we get the desired bound.
    \begin{equation*}
        \odim(f) - \odim(\bm{M}) \leq \frac{n}{2} (|K|-\mathcal{L}(\bm{M})).
    \end{equation*}
    Since this bound is independent of the filter stratum, it bounds the maximum dimension of the polyhedra that form the fiber, and hence $\dim MPH^{-1}([\bm{M}])$.
\end{proof}

We would like to check that for the case of coordinate-wise injective filters, this new bound matches the result in \autoref{lem:first-bound-cw-inj}, which we found with our simpler first approach using \eqref{eq:first-bound-fiber}.

\begin{lem}
    If $f\in \filt^n$ is injective in every coordinate, then $\mathcal{L}(MPH(f)) = |K|$ and the bound on the dimension of the fiber from \autoref{prop:bound-dim-fiber} is 0.
\end{lem}

\begin{proof}
    In this case, we can explicitly compute both $\mu_q(u)$ and the bound $b_q(u)$ for all $u \in \im f$. Let $u = f(\sigma)$. For $\epsilon>0$ small enough (chosen as in \autoref{defn:hom-Morse-num}), $K(f)_u = K(f)_{u - \epsilon e_j} \cup \{ \sigma \}$ for all $j$. Thus, we can directly compute
    \begin{equation*}
    \mu_q(u) =
        \begin{cases}
            1 & \text{if } q = \dim \sigma, \\
            0 & \text{otherwise.}
        \end{cases}
    \end{equation*}

    We now use the strategy presented in \autoref{defn:koszul-cpx}, \autoref{lem:koszul-cpx}, to show $\xi^q_{i \geq 2}(u) = 0$. Let  $\alpha \subset \{1, \dots, n\}$, and $e_\alpha = \sum_{j \in \alpha} e_j$. By the coordinate-wise injectivity of $f$, for any $\alpha, \beta \neq \varnothing$, $K(f)_{u -\epsilon e_\alpha} = K(f)_{u -\epsilon e_\beta}$. We denote this complex as $K(f)_{u^-}$. At the level of homology all spaces are equal, and structure maps are the identity. For $i \geq 1$, the Koszul complex of $M$ at $u$ is
    \begin{equation*}
        \mathbb{K}_i(M)(u)
        = \bigoplus_{\substack{\alpha\subset \{1, \dots, n\} \\ |\alpha| = i}} MPH(f)_{u - \epsilon e_\alpha}
        = \bigoplus_{\substack{\alpha\subset \{1, \dots, n\} \\ |\alpha| = i}} H(K(f)_{u -\epsilon e_\alpha})
        \cong \left( \bigoplus_{\substack{\alpha\subset \{1, \dots, n\} \\ |\alpha| = i}} \Bbbk \right) \otimes H(K(f)_{u^-}),
    \end{equation*}

    and for all $i \geq 2$, the map $d_i: \mathbb{K}_i(M)(u) \to \mathbb{K}_{i-1}(M)(u)$ acts on each generator $1_{\alpha} \otimes w$, with $\alpha = \{j_1< \dots < j_i\}$, and $w \in H(K(f)_{u^-})$, as
    \begin{equation*}
        d_i(1_{\alpha} \otimes w) = \left( \sum_{k = 1}^i (-1)^{i-k} e_{\alpha \setminus \{j_k\}} \right) \otimes w.
    \end{equation*}

    The chain complex from $i\geq 1$ decomposes as a tensor product $\mathbb{K}_{* \geq 1}(M)(u) \cong C_{* \geq 1} \otimes H(K(f)_{u^-})$, where $C_{* \geq 1}$ is isomorphic to the (truncated) chain complex resulting from the boundary maps of an $n-1$-simplex, which is exact except at the bottom of the truncation. Since any vector space is a flat module, the homology groups are all 0 for all $i\geq 2$, and hence $\xi^q_{i \geq 2}(u) = 0$.
    
    Then, the inequalities from \autoref{thm:guidolin-bound} become equalities, and we get
    \begin{equation*}
        \mu_q(u) = \xi_0^q(u) + \xi_1^{q-1}(u) = b_q(u).
    \end{equation*}

    By the previous computation, we find that the bound is $b_{q} (u) = 1$ if $q = \dim \sigma$, and 0 otherwise, and we get
    \begin{equation*}
        \mathcal{L}(\bm{M})  \geq \sum_{\sigma \in K} b_{\dim \sigma}(f(\sigma)) = |K|
    \end{equation*}

    But recall from \eqref{eq:images-in} that $\mathcal{L}(\bm{M}) \leq |K|$, so $\mathcal{L}(\bm{M}) = |K|$, and by \autoref{prop:bound-dim-fiber}, $\dim MPH^{-1}([\bm{M}]) = 0$. This provides an alternative way of showing that the fiber $MPH^{-1}([\bm{M}])$ consists of a finite number of points.
\end{proof}

\begin{rem}
    We learn from the proof above that for $f$ a coordinate-wise injective filter, either $\xi_0^{\dim \sigma}(f(\sigma)) = 1$ or $\xi_1^{\dim \sigma -1}(f(\sigma)) = 1$, and the rest of Betti tables for all homological degrees are 0.
\end{rem}

\begin{rem}
    The question of whether the bound in \autoref{prop:bound-dim-fiber} always improves the simpler bound obtained in \eqref{eq:first-bound-fiber} remains open. A sufficient condition would be that for any $[\bm{M}]\in \modspace{n}^K$ with minimal grid $G$, $|G_j| \leq \mathcal{L}(\bm{M})$ for each $j = 1,\dots,n$. This is satisfied both for $n=1$ and for coordinate-wise injective filters, and we pose the general case as a conjecture.
\end{rem}

\newpage
\section{A detailed example: 2-filters on \texorpdfstring{$\Delta^1$}{the 1-simplex}}
\label{sec:example-computation}

We include here the stratifications of $\filt^2$ and $\modspace{2}^K$ for a first example.

\begin{ex}[1-simplex with 2 parameters]
\label{ex:1-simplex}
Let $K$ be the 1-simplex, that is, an edge with its two endpoints. In this case, the only non-zero homological degree is $p = 0$, so we set $MPH = MPH_0: \filt^2 \to \modspace{2}$. Let $f:K \to I^2$ be a 2-filter which assigns values $a,b, \sigma \in I^2$ to the first vertex, second vertex and edge respectively. Note that by the filter condition $a,b \leq \sigma$.

The top-dimensional strata of $\filt^2$ are four six-dimensional strata, corresponding to the four admissible pairs of preorders on the two coordinates. We plot examples of a filter in each top-dimensional stratum in \autoref{fig:filters}.
\begin{figure}[h!]
    \centering
    \hfill
    \begin{subfigure}[b]{0.17\textwidth}
        \centering
        \includegraphics[width=\textwidth]{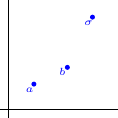}
        \caption[]
        {{\small $a_1 < b_1 < \sigma_1$, $a_2 < b_2 < \sigma_2$}}    
    \end{subfigure}
    \hfill
    \begin{subfigure}[b]{0.17\textwidth}  
        \centering 
        \includegraphics[width=\textwidth]{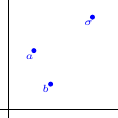}
        \caption[]
        {{\small $a_1 < b_1 < \sigma_1$, $b_2 < a_2 < \sigma_2$}}    
    \end{subfigure}
    \hfill
    \begin{subfigure}[b]{0.17\textwidth}   
        \centering 
        \includegraphics[width=\textwidth]{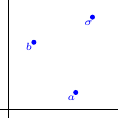}
        \caption[]
        {{\small $b_1 < a_1 < \sigma_1$, $a_2 < b_2 < \sigma_2$}}    
    \end{subfigure}
    \hfill
    \begin{subfigure}[b]{0.17\textwidth}   
        \centering 
        \includegraphics[width=\textwidth]{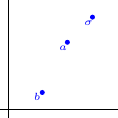}
        \caption[]
        {{\small $b_1 < a_1 < \sigma_1$, $b_2 < a_2 < \sigma_2$}}    
    \end{subfigure}
    \hfill
    \caption{Examples of filters in each of the four top-dimensional strata of $\filt^2$. We plot the values of $f$ at the two vertices, $a$, $b$, and at the edge $\sigma$.}
    \label{fig:filters}
\end{figure}

Each 6-stratum is homeomorphic to a product of two open 3-simplices,
and their closures share 5-faces of the form $\Delta^2 \times \Delta^3$ or $\Delta^3 \times \Delta^2$ corresponding to filters where $a_1 = b_1$ or $a_2 = b_2$. We can represent each filter in $\filt^2$ by a pair of points in 3-dimensional space in \autoref{fig:filter-space-example}. 

\begin{figure}[h!]
    \centering
    \includegraphics[width=0.7\linewidth]{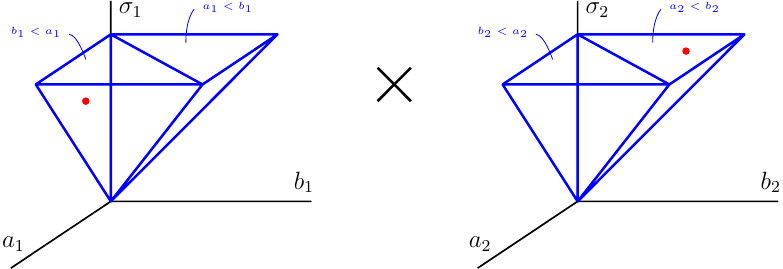}
    \caption{Representation of $\filt^2$ as a union of 4 products $\Delta^3 \times \Delta^3$ with face identifications. Here, the pair of red points uniquely identifies a filter.}
    \label{fig:filter-space-example}
\end{figure}

The map $MPH$ computes connected components and how they merge in the filtration. Since it is invariant under symmetries of $K$, in the image $\modspace{2}^K$ the 6-strata (a) and (d) in \autoref{fig:filters} are identified, and the same occurs with strata (b) and (c). These identifications are not easy to illustrate. The resulting space contains two top-dimensional strata of dimension six. There are further identifications in some of the 4-strata in which the edge appears at the same value as one of the vertices, $a=\sigma$ or $b=\sigma$. In this case, there is not a second connected component for any parameter value, which causes a collapse of a 4-dimensional stratum to a 2-dimensional stratum.

At the top of \autoref{fig:mod-space-example} we give an example of a persistence module in each of the two 6-strata. Below them there are the 5-strata that form their boundaries, and so on. We leave out the boundary strata that correspond to coordinates in the grid being 0 or 1. To represent each persistence module, we indicate the Betti numbers at each region of the parameter space $I^2$. The internal maps are not drawn: they are the linear maps induced on homology by the inclusions. We indicate with black points the values of $f$ at the vertices, and with a red point the value of $f$ at the edge. To indicate that two or three of these values coincide, we draw multiple points very close to each other. The grey arrows between strata indicate inclusions to the boundary of higher dimensional strata. The red arrows indicate the collapse of a 4-dimensional stratum to a 2-dimensional stratum, due to the persistence modules being isomorphic. 

In this example, every module in a stratum of dimension 3 or bigger has a minimal presentation with two generators, at grades $a$ and $b$, and a relation at grade $\sigma$. The bound from \autoref{prop:bound-dim-fiber} is thus 0, which is sharp. Modules in the 2-dimensional stratum can be presented with a single generator, so the bound on the dimension of the fiber is 2. By the collapses described above, we know the dimension of the fiber is exactly 2, so the bound is also sharp. We sketch an example of one such fiber in \autoref{fig:fiber-example}.

\begin{figure}[h!]
    \centering
    \includegraphics[width=\linewidth]{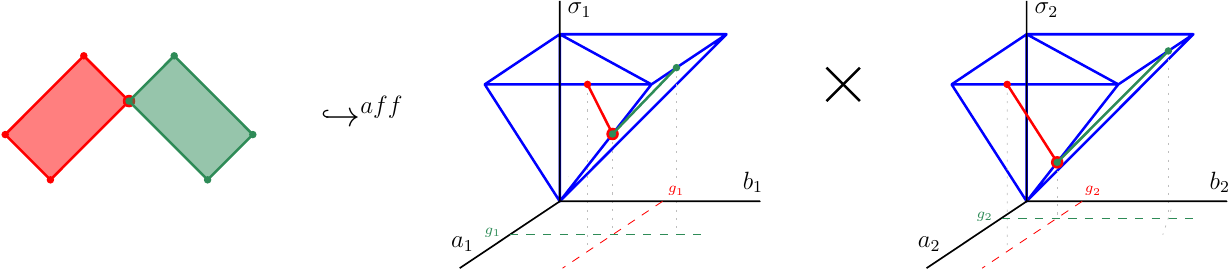}
    \caption{Representation of the fiber of a module with a single generator at grade $g$, which belongs to the 2-dimensional stratum in $\modspace{2}^K$. The fiber consists of two rectangles: the one corresponding to $g=b \leq a = \sigma$, in red, and $g=a \leq b = \sigma$, in green, joined by one of their corners.}
    \label{fig:fiber-example}
\end{figure}

\end{ex}

\begin{landscape}
\begin{figure}
    \centering
\includegraphics[width=0.8\linewidth]{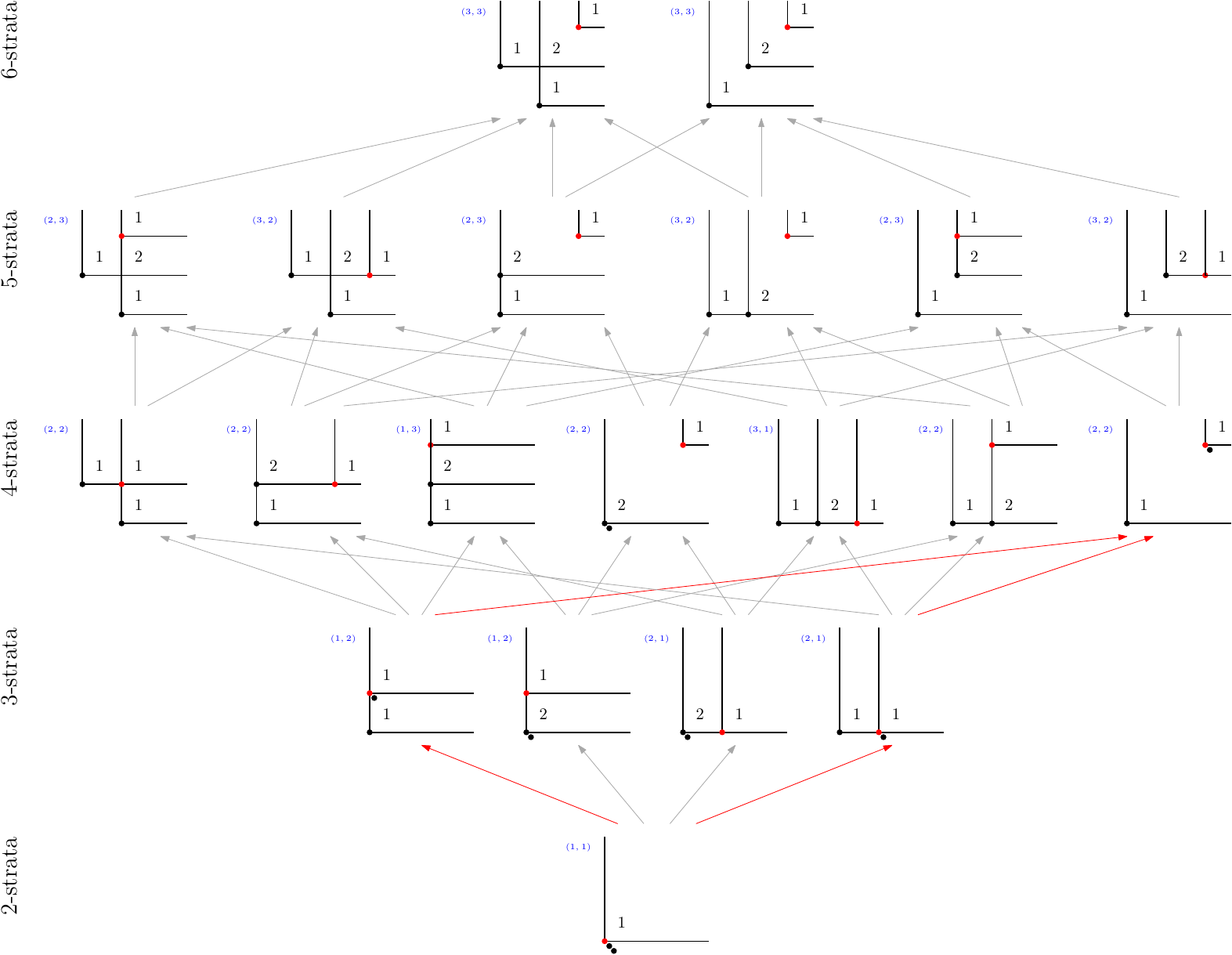}
    \caption{Representation of the strata in $\modspace{2}^K$ for \autoref{ex:1-simplex}.}
    \label{fig:mod-space-example}
\end{figure}
\end{landscape}

\printbibliography

\end{document}